\documentclass[11pt]{article}       
\usepackage{graphicx}
\usepackage{latexsym}
\usepackage{amsmath,amsfonts}
\usepackage{amssymb,amsthm}

\usepackage{graphicx}
\usepackage[english]{babel}
\usepackage{color}
\usepackage{dsfont}
\usepackage{geometry}
\usepackage[nottoc]{tocbibind}
\usepackage{indentfirst}
\usepackage{float}

\graphicspath{{./figure/}}

\usepackage{pgfplots}

\usepackage{tikz}

\usepackage{amsmath,amsfonts,amssymb,amsthm}
\usepackage{latexsym}

\newcommand{\supp}{\mbox{supp}}
\newcommand{\R}{{\mathbb R}}

\numberwithin{equation}{section}

\newcommand{\eu}[1]{\textcolor{black}{#1}}

\newtheorem{theorem}{Theorem}[section]{\bf}{\it}

\newtheorem{lemma}[theorem]{Lemma}{\bf}{\it}
{\bf}{\it}
\newtheorem{proposition}[theorem]{Proposition}{\bf}{\it}
{\bf}{\it}

\newtheorem{reminder*}[theorem]{Reminder}

\newtheorem{details*}[theorem]{Details}

\newtheorem{comm*}[theorem]{Comment}
\newtheorem{definition}[theorem]{Definition} 
\newtheorem{definition*}[theorem]{Definition}

\newtheorem{notation*}[theorem]{Notation}

\newtheorem{remark}[theorem]{Remark}

\renewcommand{\theequation}{\arabic{section}.\arabic{equation}}

	\title{On the self-similar behaviour of coagulation systems with injection}
\date{}
\begin{document}
	
	\author{Marina A. Ferreira \and Eugenia Franco \and Juan J.L. Vel\'azquez} 
%
%
%

	\maketitle





\begin{abstract}
In this paper we prove the existence of a family of self-similar solutions for a class of coagulation equations with a constant flux of particles from the origin. These solutions are expected to describe the longtime asymptotics of Smoluchowski's coagulation equations with a time independent source of clusters concentrated in small sizes.
The self-similar profiles are shown to be smooth, provided the coagulation kernel is also smooth. Moreover, the self-similar profiles are estimated from above and from below by $x^{-(\gamma+3)/2}$  \eu{as $x \to 0$, where $\gamma <1 $ is the homogeneity of the kernel}, and are proven to decay at least exponentially as $x \to \infty$. 
\end{abstract}


\section{Introduction}
\subsection{\eu{Aim of the paper}}
Smoluchowski's coagulation equation, introduced by the physicist Marian von Smoluchowski in 1916, (cf.\cite{smoluchowski1916drei}),
is a mean field model describing a system of clusters evolving in time due to coagulation upon binary collision between clusters. 
The solution of the Smoluchowski's equation, $f(t,x)$, represents the number density of clusters of size $x$ at time $t$ and it is governed by the following integro-differential equation
\begin{equation} \label{eq:coagulation}
\partial_t f(t,x)=\mathbb K [f](t,x),
\end{equation}
where $\mathbb K $ is the coagulation operator defined by
\begin{equation}\label{eq:coag_term}
\mathbb K [f](t,x) := \frac{1}{2} \int_0^x K(x-y,y) f(t,x-y)f(t,y) dy - \int_0^\infty K(x,y) f(t,x) f(t,y) dy.
\end{equation}

The kernel $K(x,y)$ is the coagulation rate of a cluster of size $y$ with a cluster of size $x$, and it summarizes the microscopical mechanisms underlying coagulation.  
Different kernels may induce completely different dynamics. For an overview on Smoluchowski's coagulation equations with different kernels we refer to \cite{aldous1999deterministic} and \cite{banasiak2019analytic}. 

Equation \eqref{eq:coagulation}, as well as its discrete counterpart, has been extensively used as a modeling tool. Polymerization \cite{blatz1945note}, animal grouping \cite{gueron1995dynamics}, hemagglutination \cite{samsel1982kinetics}, planetesimal aggregation \cite{johansen2008coagulation}, atmospheric aerosol formation \cite{friedlander2000smoke}, \cite{mcgrath2012atmospheric}, \cite{seinfeld2016atmospheric} are just some examples of applications.  

Several classes of coagulation kernels have been derived in the physical/chemical literature. The specific form of the kernels depends on mechanisms yielding the aggregation of the particles (cf. \cite{friedlander2000smoke}). Many of the kernels relevant in applications satisfy
\begin{equation}\label{eq:kernel}
c_1\left(  x^{\gamma+\lambda} y^{-\lambda} + y^{\gamma+\lambda} x^{-\lambda} \right)\leq K(x,y) \leq c_2 \left( x^{\gamma+\lambda} y^{-\lambda} + y^{\gamma+\lambda} x^{-\lambda}\right),
\end{equation}
where $0 < c_1 \leq c_2 < \infty $ and $ \gamma, \lambda \in \mathbb R$. 

Two relevant kernels in aerosol science are the Brownian kernel and the free molecular coagulation kernel \cite{friedlander2000smoke}. 
The exponents $\gamma $ and $\lambda $ in \eqref{eq:kernel}, associated to these two kernels are $(\gamma, \lambda)=(0, 1/3)$  and $(\gamma, \lambda)=(1/6,1/2) $ respectively. 
Both kernels yield the aggregation rate for a set of molecules (monomers) immersed in the air. The difference between the two kernels is that, in the first case, the mean free path of the molecules is smaller than the cluster sizes, while it is much larger in the second case (for a more detailed discussion see \eu{ \cite{ferreira2019stationary} and \cite{friedlander2000smoke}}).

We focus on the coagulation equation with source
\begin{equation} \label{eq:coagulation with source} 
\partial_t f(t,x)=\mathbb K [f](t,x)+ \eta (x)
\end{equation}
with $\mathbb K $ defined by \eqref{eq:coag_term}.
The term $\eta $ in \eqref{eq:coagulation with source} is a measure that
 represents a time independent source of particles, which is the main difference of this equation compared to the classical pure coagulation equation \eqref{eq:coagulation}. 

Equation \eqref{eq:coagulation with source} has not been studied as much as the classical coagulation equation in the mathematical literature, despite its relevance in atmospheric physics (see for instance \cite{mcgrath2012atmospheric}). 
The existence of weak time-dependent solutions of \eqref{eq:coagulation with source} has been proven, under specific assumptions on the coagulation kernel, in \cite{dubovskiui1994mathematical} and \cite{escobedo2006dust}.
The long term asymptotic behaviour of \eqref{eq:coagulation with source} has been studied in \cite{davies1999smoluchowski} with a combination of numerical simulations and matched asymptotics expansions for the kernels
\begin{equation}\label{kernel king}
K(x,y)=x^a y^b + y^a x^b, \quad a, b \geq 0. 
\end{equation}
This corresponds to $\gamma=a+b$ and $b=-\lambda$ in \eqref{eq:kernel}.

In the non-gelling regime, 
$
a+b < 1,
$
(and $a, b \geq 0$) the results of \cite{davies1999smoluchowski} suggest that the long term behaviour of a large class of solutions to \eqref{eq:coagulation with source} behave as the self-similar solution 
\begin{equation}\label{ss profile}
f_s(t,x)= \frac{1}{t^{\frac{3+\gamma}{1-\gamma}}} \phi_s\left(y\right) \text{ with } y=\frac{x}{t^{\frac{2}{1-\gamma}}}
\end{equation}
as time goes to infinity. 
 According to \cite{davies1999smoluchowski} the self-similar profile $\phi_s$ behaves as the power law 
$ y^{-\frac{\gamma+3}{2}} $ as $y$ tends to zero and decays at least exponentially as $y$ tends to infinity.

For $x$ of order $1$ we expect $f$ to behave as a steady state solution to \eqref{eq:coagulation with source} i.e. a solution to 
\begin{equation}\label{eq:steady state}
\mathbb K[f](x)+\eta(x)=0.
\end{equation}

The solutions of \eqref{eq:steady state} have been studied in \cite{dubovskiui1994mathematical} for bounded kernels, and in \cite{hayakawa1987irreversible} for the discrete coagulation equation for kernels of the form \eqref{kernel king} when $a$ and $b$ can take both positive and negative values. It is then proved that a solution to the discrete version of \eqref{eq:steady state} exists for the range of exponents $\max\{\gamma+\lambda, -\lambda\} < 1 $, $-1 \leq \gamma \leq 2 $ and $|\gamma+2\lambda|<1$. 
 
More recently, the existence of a solution to \eqref{eq:steady state} has been studied in \cite{ferreira2019stationary}, both for the continuous and the discrete cases for kernels of the form \eqref{eq:kernel}.
Specifically, it has been proven there that, if $| \gamma+2\lambda |<1$, then there exists at least a solution of \eqref{eq:steady state} and, instead, if $| \gamma+2\lambda |\geq 1 $, then equation \eqref{eq:steady state} does not have any solution. 
This implies that if the coagulation kernel is the Brownian kernel then there exists a solution of \eqref{eq:steady state} and, if the coagulation kernel is the free molecular kernel, then \eqref{eq:steady state} does not have any solution. We remark that for the kernels of the form \eqref{kernel king} considered in \cite{davies1999smoluchowski} if $a+b <1 $, then, since $a,b \geq 0$, we have that $| \gamma + 2\lambda |<1$ and a steady state solving \eqref{eq:steady state} always exists. 

Results analogous to those in \cite{ferreira2019stationary} have been obtained in \cite{laurenccot2020stationary} under different regularity assumptions on the coagulation kernels, the source $\eta$ and the solutions, as well as an additional monotonicity assumption on the kernel.

In this paper we prove, under assumptions on the parameters $\gamma $ and $\lambda$, the existence of a self-similar solution of the coagulation equation with constant flux coming from the origin, that can be formally written as 
\begin{equation} \label{intro: source with flux}
\partial_t (x  f (x))= x\mathbb K [f](x) + \delta_0, 
\end{equation}
where $\delta_0$ is the Dirac mass at $\{0\}.$
A precise definition of equation \eqref{intro: source with flux} will be presented in Section \ref{sec:setting}, Definition \ref{def:coag eq flu origin}. 
This result on the existence of self-similar solutions of equation \eqref{intro: source with flux} is the main novelty presented in this work.
In agreement with the results obtained in \cite{davies1999smoluchowski}, these self-similar solutions are expected to represent the longtime behaviour of the solutions of \eqref{eq:coagulation with source}, where we consider  $\eta$ to be a Radon measure with bounded first moment and decreasing fast enough for large sizes.   

The self-similar profiles $\phi_s$ characterizing the self-similar solutions are constructed as the limit as $\varepsilon \rightarrow 0$, of a sequence of stationary solutions of certain coagulation equations with source $\eta_\varepsilon$, where we assume that $x \eta_\varepsilon(x) $ tends to the Dirac measure supported at $\{0 \} $ as $\varepsilon \rightarrow 0.$ This is the main technical novelty of this paper and it requires to prove uniform estimates for the solutions. We also present some results on the regularity of the self-similar profiles and on their asymptotic behaviour for small and large clusters. 

We focus on non-gelling kernels, (see \cite{escobedo2002gelation} for a complete explanation of the gelation regimes),  for which a stationary solution exists, i.e., we will consider 
\begin{equation}\label{parameter assumption}
\gamma <1, \quad |\gamma+2\lambda |<1. 
\end{equation}

\eu{ 
By \cite{ferreira2019stationary} we know that, since $|\gamma +2 \lambda |< 1$, equation \eqref{eq:coagulation with source} admits a steady state solution $f$.  
Hence, we expect the solutions of equation \eqref{eq:coagulation with source} to approach a steady state solution when $x$ is of order $1$ and time goes to infinity. 
For every integrable function $f$ we denote by $J_f $ the flux associated with equation \eqref{eq:steady state}
\begin{equation}\label{flux}
J_f(z):= \int_0^z \int_{z-x}^\infty x K(y,x) f(y) f(x) dy dx.
\end{equation} }
\eu{ The analysis in \cite{ferreira2019stationary} shows that the reason because there exists at least a solution of \eqref{eq:steady state} is that the contribution to $J_f$ due to the interaction of particles of very different sizes,  $x \ll y $ or $y \ll x $, is negligible compared to the contribution to $J_f$ due to the interaction of particles of comparable sizes, $x \approx y$. 
On the contrary, when $|\gamma + 2 \lambda | \geq 1 $ (to be considered in another paper), the fact that a solution to \eqref{eq:steady state} does not exist is due to the fact that the collisions between particles of very different sizes drive very quickly the mass towards infinity. 
Therefore, in the time dependent problem we expect that if $|\gamma + 2 \lambda | \geq 1$ one will need to take into account the interaction between particles of different sizes and the behaviour of the time dependent solutions of \eqref{eq:coagulation with source} is expected to be different. In this case we expect that $f(t,x) \rightarrow 0 $ for $x$ of order  $1 $ as time goes to infinity. For this range of parameters, the existence of self-similar solutions is not studied in \cite{davies1999smoluchowski} and might be object of a future work.}

\subsection{Notation and plan of the paper}
Before beginning with the technical contents of the paper, hoping to help the reader, we clarify the notation that we adopt in this work.

First of all we employ the notation $\mathbb R_*:=(0, \infty)$, $\mathbb R_+:=[0, \infty)$. Moreover, we denote by $\mathcal L $ the Lebesgue measure. 
For any interval $I \subset \mathbb R $, we denote by $C_c(I)$ the space of continuous functions with compact support endowed with the supremum norm, denoted with $\| \cdot \|_\infty $. We denote by $C_0(\mathbb R_*)$ the space of continuous functions vanishing at infinity, which is the completion of $C_c(\mathbb R_*).$ 
As before, we endow $C_0(\mathbb R_*) $ with the supremum norm.  

We denote by $\mathcal M_+(I) $ the space of nonnegative Radon measures on $I$. 
In the following, justified by the Riesz-Markov-Kakutani representation theorem, we frequently identify $\mathcal M_+(I)$ with the set of the positive linear functionals on $C_c(\mathbb R_*).$
We adopt the notation $\mathcal M_{+,b}(I):=\{ \mu \in \mathcal M_+(I) : \mu(I) <\infty \}$ and endow this space with the total variation norm, that we denote by $\| \cdot \|$. 
Since we consider positive measures, we can easily compute the total variation norm of any measure $\mu \in \mathcal M_+(I) $, indeed $\| \mu  \|= \mu(I).$
We will sometimes endow $\mathcal M_{+,b}(\mathbb R_*)$ with the \eu{weak$-*$} topology generated by the functionals $\langle \varphi , \mu \rangle = \int_I \varphi(x) \mu(dx) $.

We denote by $C(I, \mathcal M_{+,b}(\mathbb R_*) ) $ the space of continuous functions from the compact set $I \subset \mathbb R_+$ to the space of Radon bounded measures. We endow this space with the norm 
$\| f \|_I:= \sup_{t \in I } \| f(t, \cdot ) \|.$

Notice that $\|f \|_I < \infty $ because $I$ is a compact set and $f(t, \cdot)$ is a Radon bounded measure. 
Assume $Y$ is a normed space and $S \subset Y$. 
We use the notation $C^1([0,T];S;Y)$ for the collection of maps $f : [0, T]  \rightarrow S$ such that $f$ is continuous and there
is $\dot{f} \in C([0, T]; Y )$ for which the Fr\'echet derivative of $f$ at any point $t \in [0,T] $ is given by $\dot{f}$.
We also drop the normed space $Y$ from the notation if it is clear from the context, in
particular, if $S = \mathcal M_{+,b}(I)$ and $Y = \mathcal M_+(I)$ or $Y = S$.
Clearly, if $f \in C^1([0, T]; S; Y )$, the function $\dot{f}$ is unique and it can be found by requiring
that for all $t \in (0,1)$
\begin{equation*}
\lim_{\varepsilon \rightarrow 0} \frac{\| f(t+\varepsilon)-f(t)- \varepsilon \dot{f}(t)\|_Y}{|\varepsilon|}=0
\end{equation*}
and then taking the left and right limits to obtain the values $\dot{f}(0)$ and $\dot{f}(T)$.
To keep the notation lighter, in some of the proofs, we will denote by $C$ or $c$ a constant that may be different from line to line.

We denote by $\hat{f}$ the Fourier transform of $f: \mathbb R \rightarrow \mathbb R$ defined by
\[
\eu{\hat{f}(y )}:=\frac{1}{(2\pi)^{1/2}} \int_{\mathbb R} e^{- i x y} f(x) dx. 
\]
We define the Sobolev spaces of fractional order $s$ (negative or positive) as
\[
H^s(\mathbb R):=\{ f \in \mathcal S'(\mathbb R):  \| f \|_{H^{s}(\mathbb R) } < \infty \} 
\]
where, $\mathcal S (\mathbb R)$ \eu{ is the space of
infinitely differentiable and rapidly decreasing functions} and $\mathcal S' (\mathbb R)$ \eu{is its dual}, we refer to Definition 14.6 in \cite{duistermaatdistributions} \eu{for a precise definition}, and the norm $\| \cdot \|_{H^{s}(\mathbb R) }$ is given by
\[
\| f \|_{H^{s}(\mathbb R) }^2 := \int_{\mathbb R }  (1+|x |^2)^{s} |\hat{f}(x)|^2 dx. 
\]
Let $\Omega $ be an open set and let $s \in \mathbb R$. 
The space $H^s(\Omega) $ is the set of the restricted functions $f_{| \Omega}$ with $f: \mathbb R \rightarrow \mathbb R$  belonging to $H^s(\mathbb R)$. It is a Banach space equipped with the norm 
\[
\|f \|_{H^s(\Omega)}:= \inf \{ g \in H^s(\Omega) : g_{|\Omega} =f  \}. 
\] 
This definition is the same as in \cite{mclean2000strongly}. 

Finally, we will use the notation $f \sim g$ as $x \rightarrow x_0 $ to indicate the asymptotic equivalence between the function $f$ and the function $g$, i.e. $\lim_{x \rightarrow x_0} \frac{f(x)}{g(x)} = 1$. Instead, we will use the notation $f \approx g $ to say \eu{that there exists a constant $M>0$ such that   $\frac{1}{M}\leq \frac{f}{g} \leq  M $}. 

The organization of the paper is the following. 
In Section \ref{sec:heuristic} we discuss the heuristic justification to study the self-similar solutions constructed in this paper. 
In Section \ref{sec:setting}, we explain in detail the setting in which we work, we present the definition of self-similar profile and state the main results of the paper.
In Section \ref{sec:existence section} we prove the existence of a self-similar profile, while in Section \ref{sec:moments} and Section \ref{sec:regularity ss} we prove its properties. Finally, in Section \ref{sec:coag with influx}, we prove that the self-similar solution $f_s$ defined by \eqref{ss profile}, solves equation \eqref{intro: source with flux}.   

\section{Heuristic argument} \label{sec:heuristic}
\subsection{Scaling parameters} \label{subsec:par}
In this subsection, we explain how the exponents in \eqref{ss profile} are computed.
Since we are only considering non-gelling kernels, the change of mass in the system is only due to the contribution of the source. Indeed, multiplying equation \eqref{eq:coagulation with source} by $x$ and integrating in $x$ from $0$ to $\infty$ we obtain formally
\begin{equation}\label{mass time dep}
\frac{d }{dt } \int_0^\infty x f(t,x) dx = \int_0^\infty  x \eta (dx)< \infty .
\end{equation}

As a consequence, using the change of variables 
\begin{equation}\label{time dep self-similar profile}
f(t,x) = \frac{1}{t^\alpha }\phi\left(\ln t, \frac{x}{t^\beta}\right), \quad \xi = \frac{x}{t^\beta}, \quad \tau = \ln t
\end{equation}
 and considering the initial condition $f(0,x) = 0$, motivated by the fact that the total mass in the system is proportional to time, we conclude that
\[
t^{2\beta-\alpha} \int_0^\infty \xi \phi\left(\tau, \xi \right) d\xi   =\int_0^\infty \frac{x}{t^\alpha} \phi\left(\eu{\ln t}, \frac{x}{t^\beta}\right) dx=\int_0^\infty  x f(t,x) dx = t \int_0^\infty  x \eta (dx), 
\] 
which implies that $2\beta-\alpha=1.$
Hence the mass of the source is equal to the mass of $\phi$.

Moreover, using \eqref{time dep self-similar profile} in the coagulation term \eqref{eq:coag_term} and using \eqref{eq:kernel} yields the scaling
\begin{align*}
\mathbb K [f](t,x) \sim t^{\beta-2\alpha +\beta \gamma} \mathbb K [\phi](\tau, \xi). 
\end{align*}

On the other hand, the fact that
\begin{align*}
\partial_t\left(t^ {-\alpha} \phi \left(\ln t, \frac{x}{ t^{\beta}}, \right)\right) = - t^{-\alpha-1} \left(\alpha \phi\left( \ln t , \frac{x}{ t^{\beta}}\right) +\beta \xi \partial_\xi \phi\left(\ln t, \frac{x}{ t^{\beta}} \right) -\partial_\tau \phi\left(\ln t, \frac{x}{ t^{\beta}}\right) \right)
\end{align*}
implies that $-1-\alpha = \beta-2\alpha +\beta \gamma$. 
Recalling
the condition $ 2\beta-\alpha=1 $
we conclude that 
\begin{equation}\label{eq:scaling}
\beta= \frac{2}{1-\gamma}, \quad  \alpha= \frac{3+\gamma}{1-\gamma}. 
\end{equation}
Notice that this scaling is in agreement with \cite{davies1999smoluchowski}.

We now conclude this Section by noticing that equation \eqref{eq:coagulation with source} has a scaling-invariance property. Indeed if $f$ solves \eqref{eq:coagulation with source} with a source of mass $\int_0^\infty x \eta (dx) = M_\eta $, then $\tilde f(t,y)= (M_\eta)^{\gamma/(1-\gamma)}f(t, M_\eta^{1/(1-\gamma)}  y )$ solves \eqref{eq:coagulation with source} with a source of mass $1.$ Without loss of generality we therefore assume from now on that the source $\eta $ has mass equal to $1.$

\subsection{Formal derivation of the equation for the self-similar profile}

In this Section, we explain why we expect the solution of equation \eqref{eq:coagulation with source} to approach, as time tends to infinity, a self-similar solution for $x\gg 1 $ and a steady state for $x \approx 1$, following an argument inspired by the one in \cite{davies1999smoluchowski}.

\eu{
When $|\gamma +2 \lambda |< 1$, holds we know, from \cite{ferreira2019stationary}, that a stationary solution $\overline{f}$ exists and we expect, in view of the numerical simulations in \cite{davies1999smoluchowski}, that
\begin{equation} \label{outer behaviour}
f(t,x) \rightarrow \overline{f} (x) \text{ as } t  \rightarrow \infty \text{ for $x$ of order 1}.
\end{equation}
}

The fact that $J_{\overline{f}} (x) \rightarrow 1 $  as $x \rightarrow \infty $ (proven in \cite{ferreira2019stationary}) and that the simplest solution of $J_{\phi} (x)=1 $ is $\phi(x) = c_0 x^{-\frac{3+\gamma}{2}}$, with $c_0=\left(  \int_0^1 \int_1^\infty K(y,z) z^{-\frac{\gamma+3}{2}} y^{-\frac{\gamma+1}{2}} dy dz  \right)^{- 1/2}$, suggests that $\overline{f} (x) \sim c_0 x^{-\frac{3+\gamma}{2}} $ as $x \rightarrow \infty.$ This yields the following matching condition: 
\begin{equation}\label{matching condition} 
f(t,x) \sim c_0 x^{-\frac{3+\gamma}{2}}
\end{equation} 
 for $1 \ll x \ll t^{ \frac{2}{1-\gamma}} $ or equivalently  $1 \ll x \ll e^{\frac{2}{1-\gamma} \tau} .$ 

We now describe the asymptotic behaviour of $f(t, x) $ in the self-similar region $x \approx  t^{\frac{2}{1-\gamma}}$. Using the self-similar change of variables \eqref{time dep self-similar profile} in \eqref{eq:coagulation with source}, we deduce that $\phi$ satisfies the following transport-coagulation equation with source
\begin{equation} \label{time dep eq with time-source intro} 
 \partial_\tau \phi (\tau, \xi) = \frac{3+\gamma}{1-\gamma} \phi(\tau, \xi) + \frac{2}{1-\gamma}  \xi \partial_\xi \phi(\tau, \xi) + \mathbb K [\phi] (\tau, \xi) + e^{\frac{4}{1-\gamma} \tau }\eta\left(\xi e^{\frac{2}{1-\gamma} \tau}\right)
\end{equation}
for $\tau >0$ and $ \xi >0. $

Since $\int_0^\infty x \eta(dx) =1$, then in the region $\xi \approx 1 $ the term $e^{\frac{4}{1-\gamma} \tau }\eta\left(\xi e^{\frac{2}{1-\gamma} \tau}\right) $ tends to zero in the sense of measures as $\tau \rightarrow \infty $.

 We make the self-similar ansatz, i.e we assume that there exists a self-similar profile $\phi_s$ such that
\begin{equation} \label{phi tends to phis}
\phi(\tau,\xi) \rightarrow \phi_s(\xi) \text{ as } \tau \rightarrow \infty
\end{equation}
and conclude that $\phi_s$ solves
\begin{equation}\label{coag eq self sim} 
0=\frac{3+\gamma}{1-\gamma} \phi_s( \xi) + \frac{2}{1-\gamma}  \xi \partial_\xi \phi_s( \xi) + \mathbb K [\phi_s] (\xi) \quad  \text{ for }\xi>0.
\end{equation}
By the matching condition \eqref{matching condition}, we know that 
\begin{equation} \label{boundary}
\phi_s \sim c_0 \xi^{-\frac{3+\gamma}{2}} \text{ as } \xi \rightarrow 0. 
\end{equation} 

The fact that $\int_0^\infty \xi \phi_s(\xi) d\xi<\infty $ and the shape of the equation suggest that $\phi_s$ decays at least exponentially (see also the statement of Theorem \ref{thm: existence of a ss sol}), later we will justify the precise ansatz  
\begin{equation}\label{expo}
\phi_s(\xi) \sim c e^{-L \xi} \xi^{-\gamma} \text{ as } \xi \rightarrow \infty, 
\end{equation} 
for some $L>0$. Rigorous upper estimates for $\phi_s$ supporting this asymptotic behaviour will be also derived in Theorem \ref{thm: existence of a ss sol}. 
The matching condition \eqref{boundary}, together with \eqref{expo} then implies
\begin{equation}\label{flux boundary cond}
\lim_{\xi \rightarrow 0} J_{\phi_s}(\xi)=1 . 
\end{equation}
We conclude that the self-similar profile satisfies equation \eqref{coag eq self sim} with the boundary condition \eqref{flux boundary cond}. 

We now derive a relation between the flux coming from the origin, \eqref{flux boundary cond}, and $\int_0^\infty \xi \phi_s(d\xi). $
Multiplying \eqref{coag eq self sim} by $\xi $ and noticing that $ \xi \mathbb K[\phi_s] (\xi)= - \partial_\xi J_{\phi_s} (\xi) $ we obtain
\begin{align*}
0=-  \xi  \phi_s(\xi)  + \partial_\xi \left( \frac{2}{1-\gamma}\xi^2 \phi_s(\xi)  -  J_{\phi_s} (\xi) \right) . 
\end{align*}
Integrating from $0$ to infinity we deduce that
\begin{align}\label{eq:ss con limiti}
 \int_0^\infty \xi  \phi_s(\xi) d\xi = \lim_{\xi \rightarrow \infty} \left( \frac{2}{1-\gamma}\xi^2 \phi_s(\xi)  -  J_{\phi_s} (\xi) \right) -
 \lim_{\xi \rightarrow 0} \left( \frac{2}{1-\gamma}\xi^2 \phi_s(\xi)  -  J_{\phi_s}(\xi) \right). 
\end{align}

Thanks to \eqref{boundary}, and the assumption $\gamma <1$ we deduce that $ \lim_{\xi \rightarrow 0} \xi^2 \phi_s(\xi) =0$. 
Moreover, thanks to \eqref{expo} we deduce that $\lim_{\xi \rightarrow \infty} J_{\phi_s}(\xi)=0$ and $ \lim_{\xi \rightarrow \infty} \xi^2 \phi_s(\xi) =0$. Therefore, combining \eqref{eq:ss con limiti} with  \eqref{flux boundary cond} we obtain that 
\[
\int_0^\infty \xi \phi_s(\xi) d\xi=1. 
\]
Using the self-similar change of variables \eqref{time dep self-similar profile} we deduce that the self-similar solution $f_s$ satisfies 
\[
\int_0^\infty x f_s(t,x) dx =t, 
\]
that is consistent with \eqref{mass time dep}.

Finally we justify \eqref{expo}. 
To this end we substitute in equation \eqref{coag eq self sim} the ansatz $\phi_s(\xi) \sim c e^{-L \xi} \xi^{a} $ as $\xi \rightarrow \infty$ and formally estimate the behaviour at infinity of all the terms in \eqref{coag eq self sim}, to deduce that $a=- \gamma. $
We start by considering the coagulation term 
\begin{align*}
&\mathbb K[e^{-L \xi} \xi^a ] = \frac{c^2}{2} e^{-L \xi} \xi^a \int_0^\xi K(y,\xi-y) y^a (\xi-y)^a dy - c^2 e^{-L \xi } \xi^a \int_0^\infty K(\xi,y)  e^{-L y} y^a  dy  \\
&=c^2 e^{-L \xi} \xi^{\gamma+1+2a}\left(  \frac{1}{2}\int_0^1 K(y,1-y) y^a (1-y)^a dy - \int_0^\infty K(1,y)  e^{-L \xi y} y^a  dy \right)  \\
& \sim  \frac{c^2}{2} e^{-L \xi} \xi^{\gamma+1+2a} \int_0^1 K(y,1-y) y^a (1-y)^a dy\\
\end{align*} 
and therefore 
\begin{equation} \label{behaviour of K}
\mathbb K[e^{-L \xi} \xi^a ] \sim c^2 \left(\int_0^1 K(y,1-y) y^a (1-y)^a dy\right) e^{-L\xi } \xi^{\gamma +1+ 2a} \quad { as } \quad  \xi \rightarrow \infty. 
\end{equation} 
On the other hand, 
\begin{equation}\label{behaviour of transport}
\frac{3+\gamma}{1-\gamma} \phi_s(\xi) + \frac{2}{1-\gamma} \xi \phi'_s(\xi) \sim - \frac{2cL}{1-\gamma} e^{-L \xi } \xi^{1+a} \quad { as } \quad \xi \rightarrow \infty. 
\end{equation}
Combining \eqref{behaviour of K} and \eqref{behaviour of transport} we deduce that $\phi_s (\xi) \sim c e^{-L \xi } \xi^{- \gamma }$ and \[
c=\frac{2L }{1-\gamma}\left(\int_0^1 K(y,1-y) y^{-\gamma} (1-y)^{-\gamma} dy \right)^{-1}. \]

The power law behaviour near the origin and the exponential behaviour at infinity are supported by the numerical simulations in \cite{davies1999smoluchowski}. 
The estimates \eqref{eq:estimate_Phi},  \eqref{eq:estimate_Phi above} show that the mean of $\phi_s$ behaves as $x^{-(\gamma+3)/2}$ near the origin and the inequality \eqref{point estimate} shows that $\phi_s $ decays at least exponentially for large sizes.

\subsection{Longtime asymptotics}
In this Section we present a different argument justifying the self-similar behaviour of solutions of \eqref{eq:coagulation with source}.
We show that, if the self-similar profile can be uniquely identified as the solution of a coagulation equation with constant flux coming from the origin (equation \eqref{intro: source with flux}), then the self-similar solution describes the longtime behaviour of the solutions to the coagulation equation with source \eqref{eq:coagulation with source}.
Since the investigation of the uniqueness of the coagulation equation with constant flux coming from the origin is still an open problem, the following argument represents only a formal heuristic motivation for the study of self-similar solutions of equation \eqref{eq:coagulation with source}.

Nevertheless, the self-similar ansatz \eqref{ss profile} is corroborated, at least for some of the kernels considered here, by the numerical simulations and heuristic explanations in \cite{davies1999smoluchowski} and by \cite{lehtinen2001self} from the point of view of physics.

Let us consider a solution $f$ to \eqref{eq:coagulation with source} with initial condition $f_0$ such that $f_0(y) < c  y^{\omega}$ with $\omega < - \frac{\gamma+3 }{2}$ and $c>0$, and a positive constant $R$.
Since we are interested in the longtime behaviour and we are assuming $\gamma <1$,  we consider the following change of variables
\begin{equation*}
x = \xi R,\quad t = R^{\frac{1-\gamma}{2}}, s \quad s>0.
\end{equation*}
The scaling in size balances the scaling in time in such a way that the function $F_R$ defined by 
\begin{equation} \label{scaling FR}
 F_R(s,\xi) :=  R^{(\gamma+3)/2} f(R\xi, R^{\frac{1-\gamma}{2}} s) 
\end{equation}
satisfies the coagulation equation
\begin{equation} \label{eq:F_R}
\partial_s F_R(s,\xi) = \mathbb{K}[F_R](s,\xi) +  \eta_R(\xi)
\end{equation}
with the source 
\[
 \eta_R(\xi) := R^2\eta(R\xi)
\]
and the initial condition $F_R(0,\xi) = R^{\frac{\gamma+3}{2}} f_0(R\xi)$.

Integrating equation \eqref{eq:F_R} against a test function $\varphi \in C([0,T] \times \mathbb R_+) $ we obtain that
\begin{align} \label{eq:F_R weak}
& \int_{\mathbb R_*}  \xi \varphi(t, \xi)  F_R(t,\xi) d\xi = \int_{\mathbb R_*}  \xi \varphi(0, \xi)  F_R(0,\xi) d\xi +\int_0^t \int_{\mathbb R_*}  \xi \partial_s \varphi(s, \xi)  F_R(s,\xi) d\xi ds \\
&+ \int_0^t \int_{\mathbb R_*} \int_{\mathbb R_*}\frac{ K(y, \xi)}{2} \left[\varphi\left( s, \xi + y \right) (\xi +y) - \xi \varphi(s, \xi) -y \varphi(s,y) \right] \cdot  \nonumber \\
&\cdot F_R (s,\xi)  F_R (s,y) d\xi dy ds  +\int_0^t  \int_{\mathbb R_*} \xi \varphi(s,\xi) \eta_R(\xi) d\xi ds \nonumber
\end{align}

Since the source $\eta_R$ decays fast enough and 
$$\int_0^\infty \xi \eta_R(\xi) d\xi = 1$$
we infer that 
\[
\lim_{R \rightarrow \infty} \int_{\mathbb R_*} \xi \varphi(s,\xi) \eta_R(\xi) d\xi =\lim_{R \rightarrow \infty} \int_{\mathbb R_*} \varphi\left(s,\frac{y}{R}\right) y \eta(y) dy= \varphi(s,0)
\]
and $y \eta_R(y) \rightarrow \delta_{0} (y)$ as $R \rightarrow \infty.$

Assuming that the solution $F_R$ is unique and that  the limit of $F_R$ when $R\to \infty$ exists,
passing to the limit as $R \rightarrow \infty $ in equation \eqref{eq:F_R}, we deduce that $F$, the limit of $F_R$, solves the coagulation equation with constant flux coming from the origin in the sense of Definition \ref{def:coag eq flu origin} (Section \ref{sec:setting}). 

 We will prove in Theorem \ref{thm:asymptotic} that the function $ s^{- \frac{\gamma+3}{1-\gamma}} \phi_s( \xi s^{-\frac{2}{1-\gamma}})$, with $\phi_s$ solving \eqref{coag eq self sim} with the boundary condition \eqref{flux boundary cond} satisfies equation \eqref{eq:F} for every test function $\varphi \in C^1([0,T], C(\mathbb R_+))$. 
Assuming that \eqref{eq:F} has a unique solution we conclude that
\begin{equation}\label{F as limit}
F(s, \xi)=\lim_{R \rightarrow \infty }F_R(s,\xi)= s^{-\frac{\gamma+3}{1-\gamma}} \phi_s(\xi s^{-\frac{2}{1-\gamma}}). 
\end{equation}
Choosing $R= e^{\frac{2}{1-\gamma}\tau}$ and $s=1$ in \eqref{scaling FR}, we conclude that 
\begin{equation*}
 F_{e^{\frac{2\tau}{1-\gamma}} }(1,\xi) = e^{\frac{3+\gamma}{1-\gamma} \tau} f(e^\tau, \xi e^{\frac{2\tau}{1-\gamma}}), 
\end{equation*}
where $f$ is a solution of equation \eqref{eq:coagulation with source}. 

From \eqref{F as limit} and the fact that $F(1,\xi) = \phi_s(\xi)$, we deduce that, as $\tau $ tends to infinity, 
\[
e^{\frac{3+\gamma}{1-\gamma} \tau} f(e^\tau, \xi e^{\frac{2}{1-\gamma} \tau}) \rightarrow \phi_s(\xi). 
\]
This implies the self-similar ansatz \eqref{ss profile}.

\section{Setting and main results} \label{sec:setting}
Given $\gamma, \lambda \in \mathbb R$ and $c_1, c_2>0$, a continuous symmetric function $K: \mathbb R_* \times \mathbb R_*\rightarrow \mathbb R_+$, is a coagulation kernel of parameters $\gamma $ and $\lambda$ if it satisfies the following inequalities
\begin{equation} \label{kernel gamma lambda}
c_1 w(x,y) \leq  K(x,y) \leq  c_2 w(x,y ) \quad  \text{ for all } x, y \in \mathbb R_*,
\end{equation} 
where
 \begin{equation*}
 w(x,y)= \frac{x^{\gamma+\lambda}}{y^\lambda}+ \frac{y^{\gamma+\lambda}}{x^\lambda}.
\end{equation*}
We assume that the coagulation kernel is homogeneous, with homogeneity $\gamma$, that is, for any  $b>0$ it satisfies 
\[
K(b x,b y)= b^\gamma K(x,y) \quad \text{ for all } x,y \in \mathbb R_*. 
\]

We now give the definition of self-similar profile for equation \eqref{intro: source with flux}. 
\begin{definition}\label{def:ss_sol}
Let $K$ be a homogeneous symmetric coagulation kernel $K \in C(\mathbb R_* \times \mathbb R_*)$ satisfying \eqref{kernel gamma lambda} with homogeneity $\gamma < 1$.
A self-similar profile of equation \eqref{intro: source with flux} with respect to the kernel $K$, is a measurable function $\phi \geq 0$ with 
\[
\int_{\mathbb R_*} x \phi(x) dx =1 \quad \text{ and } \quad J_\phi \in L^\infty_{loc}(\mathbb R_*), 
\]
where $J_\phi$ is defined by \eqref{flux}, and it satisfies
\begin{eqnarray}\label{eq:ss_eq}
J_\phi(z)=1
-  \int_0^z x\phi(x) dx  + \frac{2}{1-\gamma}  z^2 \phi (z) \quad  a.e.\  z>0. 
\end{eqnarray}
\end{definition}

\begin{theorem}[Existence]\label{thm: existence of a ss sol}
Let $K $ be a homogeneous symmetric coagulation kernel $K \in C(\mathbb R_* \times \mathbb R_*)$ satisfying \eqref{kernel gamma lambda}, with homogeneity $\gamma <1 $ and $ | \gamma+ 2 \lambda | <1.$
Then, there exists a self-similar profile $\phi$ as in Definition \ref{def:ss_sol}. 
Moreover, there exist positive constants $C$ and $b_1$ with $b_1< 1$ such that 
\begin{equation}\label{eq:estimate_Phi}
\frac{1}{z} \int_{ b_1z}^z  \phi(x) dx \leq \frac{C}{z^{(3+\gamma)/2}}\quad \text{ for any } z >0 
\end{equation} 
 and there exist two constants $b_2 \in (0,1)$ and $c>0$, depending on the parameters of the kernel $\gamma, \lambda $ as well as on $c_1, c_2$ in \eqref{kernel gamma lambda}, such that 
\begin{equation}\label{eq:estimate_Phi above}
 \frac{1}{z} \int_{b_2z}^z  \phi(x)dx \geq  \frac{c}{z^{(3+\gamma)/2}},\quad z \in (0,1]. 
\end{equation} 
There exists 
 a positive constant $L$ such that  
\begin{equation}\label{exp bound}
\int_1^\infty e^{Lx} \phi(x) dx < \infty
\end{equation}
and a positive constant $\rho$ such that 
\begin{equation}\label{point estimate}
\limsup_{z \rightarrow \infty}\phi(z) e^{\rho z } < \infty. 
\end{equation}
\end{theorem}

\begin{theorem}[Regularity] \label{thm:regularity} 
Assume $K$, $\lambda$ and $\gamma$ to be as in the assumptions of Theorem \ref{thm: existence of a ss sol} and assume $\phi$ to be the self-similar profile whose existence is proven in Theorem \ref{thm: existence of a ss sol}.
Assume that for every $0< y \leq 1$ we have $K(\cdot,y) \in H^l  (\mathbb R_*)$ with
\begin{equation}\label{kernel derivatives}
 \sup_{  0< y\leq 1 } y^{- \min\{\gamma+\lambda, -\lambda\}} \left\| K(\cdot,y) \right\|_{H^l  (1/2,2)} < \infty. 
\end{equation}
Then the self-similar profile $\phi$ is such that $x^2 \phi(x) \in H^{l}(\mathbb R_*)$. 

If $l \geq 3/2$ then $\phi \in C^1(\mathbb R_*) $ and it satisfies
\begin{align}\label{self similar smooth} 
&\int_0^{x/2} \left[ K(x-y,y) \phi(x-y) -  K(x,y) \phi(x) \right]  \phi(y) dy + \int_{x/2}^\infty K(x,y) \phi(x) \phi(y) dy \\
&+ \frac{3+\gamma}{1-\gamma} \phi(x) + \frac{2}{1-\gamma} x \partial_x \phi(x) =0 \nonumber \quad  \forall x \in \mathbb R_*
\end{align}
with the boundary condition \eqref{flux boundary cond}. 
\end{theorem}
\begin{remark}
Due to the homogeneity and the symmetry of the kernel, inequality \eqref{kernel derivatives} implies that for every $(a,b) \subset \mathbb R_+$ with $a>b>0$, we have
\[
 \sup_{  0< y\leq 1} y^{- \min\{\gamma+\lambda, -\lambda\}} \left\| K(\cdot,y) \right\|_{H^l ((a,b))} < \infty.
\]
\end{remark}

\begin{definition}[Coagulation equation with constant flux coming from the origin]\label{def:coag eq flu origin}
Let $K$ be a homogeneous symmetric coagulation kernel $K \in C(\mathbb R_* \times \mathbb R_*)$ satisfying \eqref{kernel gamma lambda} with homogeneity $\gamma < 1$, and let $T>0.$
We say that $F \in C([0,T], \mathcal M_+ (\mathbb R_+))$ is a solution of the coagulation equation with constant flux coming from the origin with initial condition $F(0, \cdot)=0$, if 
\begin{equation}\label{F bound}
\sup_{s \in [0,T]} \int_{(0,1]} \xi^{q} F(s,d\xi) < \infty  \quad \sup_{s \in [0,T]} \int_{[1, \infty)}\xi^{p} F(s,d\xi) < \infty
\end{equation}
for $q=\min\{1+\gamma+\lambda, 1-\lambda,1\}$, and $p =\max\{\gamma+\lambda, -\lambda \}$ and if it solves equation
\begin{align}\label{eq:F}
& \int_{\mathbb R_*} \xi \varphi(t, \xi)  F(t,d\xi) =  \int_0^t  \int_{\mathbb R_*}  \xi \partial_s \varphi(s, \xi)  F(s,d\xi) ds + \int_0^t \varphi(s, 0) ds \\
& + \frac{1}{2}\int_0^t \int_{\mathbb R_*} \int_{\mathbb R_*} K(\eta, \xi) \left[\left( \xi + \eta \right) \varphi(s, \xi +\eta) - \xi \varphi(s, \xi) - \eta \varphi(s, \eta) \right] F (s,d\xi)  F (s,d\eta )  ds \nonumber
\end{align}   
 for every test function $\varphi \in C^1([0,T], C^1_c(\mathbb R_+))$ and every $0 \leq t < T $. 
\end{definition}

\begin{theorem}\label{thm:asymptotic} 
Assume $K$, $\lambda$ and $\gamma$ to be as in the assumptions of Theorem \ref{thm: existence of a ss sol}. 
Let $\phi$ be a self-similar profile as in Definition \ref{def:ss_sol}.
Then $F(t, d\xi ):=t^{-\frac{\gamma+3}{1-\gamma}} \phi(\xi t^{-\frac{2}{1-\gamma}} )d\xi$ solves the coagulation equation with flux coming from the origin in the sense of Definition \ref{def:coag eq flu origin}. 
\end{theorem}

\begin{remark}
We underline that the moment bounds \eqref{eq:estimate_Phi}, \eqref{eq:estimate_Phi above}, \eqref{exp bound}, estimate \eqref{point estimate} and the regularity properties in Theorem \ref{thm:regularity} are proven only for the self-similar profile whose existence is stated in Theorem \ref{thm: existence of a ss sol}. 
We do not know if these properties are more general, i.e. if they hold for all the self-similar profiles as in Definition \ref{def:ss_sol}. (Since we do not prove uniqueness of the self-similar profile the existence of many self-similar profiles is not excluded). 
\end{remark}

\section{Existence of a self-similar profile} \label{sec:existence section}
We briefly explain the technique we adopt to prove the existence of a solution of equation \eqref{coag eq self sim} with the boundary condition \eqref{flux boundary cond}.

The main idea is to approximate a solution of \eqref{coag eq self sim} with a sequence of solutions $\phi_{\varepsilon}$ of equation
\begin{equation} \label{eq with eps_source} 
0 = \frac{3+\gamma}{1-\gamma} \phi_{\varepsilon} (\xi) +\frac{2}{1-\gamma} \xi \partial_\xi \phi_\varepsilon (\xi) + \mathbb K [\phi_\varepsilon] ( \xi) + \eta_\varepsilon(\xi) 
\end{equation} 
where $\eta_\varepsilon$ is a smooth function with support $[\varepsilon,2 \varepsilon ]$ and such that 
\[
\int_{\mathbb R_*} x \eta_\varepsilon (x) dx = \int_{\mathbb R_*}  y \eta (dy)=1. 
\] 

To prove the existence of a solution of \eqref{eq with eps_source} we follow an approach which is extensively used in the analysis of kinetic equations, see for instance \cite{escobedo2005self} and \cite{gamba2004boltzmann}. 
We first prove the existence of solutions of \eqref{eq with eps_source} by considering the corresponding truncated time dependent evolution problem: 
\begin{equation} \label{time dep eq with eps_source} 
 \partial_t \phi_\varepsilon (\tau, \xi) = -  \phi_\varepsilon(\tau, \xi) + \frac{2}{1-\gamma}  \frac{1}{\xi} \partial_\xi\left(  \xi^2 \Xi_\varepsilon(\xi) \phi_\varepsilon(\tau, \xi)\right)  + \mathbb K_{a,R} [\phi_\varepsilon] (\tau, \xi) + \eta_\varepsilon(\xi) 
\end{equation} 
where $\Xi_\varepsilon $ is a smooth monotone function $\Xi_\varepsilon: \mathbb R_* \rightarrow \mathbb R_+$ such that $\Xi_\varepsilon (x) =1$ if $x \geq 2 \varepsilon$ and $\Xi_\varepsilon(x)=0$ if $x \leq \varepsilon$, while  $\mathbb K_{R,a}$ is the truncated coagulation operator defined by
\begin{align*}
\mathbb K_{R,a}[f](\tau,\xi)&:=\frac{\zeta_{R}(\xi)}{2} \int_0^\xi K_a(\xi-z,z) f(t,\xi-z)f(t,z) dz \\
&- \int_0^\infty K_a(\xi,z) f(t,\xi) f(t,z) dz,
\end{align*}
where $K_a$ is a kernel bounded from above by $a$ and from below by $1/a$ and $\zeta_R $ is a truncation function of parameter $R>0$, i.e. it is a smooth function $\zeta_{R}: \mathbb R_* \rightarrow \mathbb R_+$ such that $\zeta_{R}(x) =1$ if $x \leq R$ while $\zeta_{R}(x) =0$ if $x \geq2 R$. 
The specific truncation in the growth term of \eqref{time dep eq with eps_source} has been chosen in order to ensure that the set $\{\phi:  \int_0^\infty \xi \phi(\xi) d\xi \leq 1\} $ is invariant under the evolution equation corresponding to \eqref{time dep eq with eps_source}. 

In Section \ref{sec:trunc existence} we prove the existence of a weak solution to \eqref{time dep eq with eps_source}. 
In Section \ref{sec:steady state} we prove, using Tychonoff fixed point theorem, the existence of a stationary weak solution of equation \eqref{time dep eq with eps_source} 
\begin{equation} \label{eq with eps_source trunc} 
0 =  -  \phi_\varepsilon(\xi) + \frac{2}{1-\gamma}  \frac{1}{\xi} \partial_\xi\left(  \xi^2 \Xi_\varepsilon(\xi) \phi_\varepsilon( \xi)\right)  + \mathbb K_{a,R} [\phi_\varepsilon] (\xi) + \eta_\varepsilon(\xi),
\end{equation} 
and in Section \ref{sec:properties steady state} we study the properties of $\phi_\varepsilon$. 
In Section \ref{sec:removal of the truncation} we show that the solutions of \eqref{eq with eps_source trunc} approximate a self-similar profile in the sense of Definition \ref{def:ss_sol} as $R\rightarrow \infty, a \rightarrow \infty, \varepsilon \rightarrow 0$. 

\subsection{Truncated time dependent coagulation equation with source written in self-similar variables} \label{sec:trunc existence}
We introduce now some terminology which will enable us to define the truncated equation. 
As we will see in the proof of Proposition \ref{prop: exist time dep}, to prove the existence of a solution of equation \eqref{time dep eq with eps_source}, we prove the existence of a solution of an auxiliary equation, obtained via a time dependent change of variables.

The aim of the rest of this section is to prove the existence of a solution, for every $\varphi \in C^1([0,T], C^1_c(\mathbb R_*) )$, of the following truncated equation
\begin{align} \label{weak time dep} 
& \int_{\mathbb R_*} \Phi(t,d\xi ) \varphi(t,\xi)- \int_{\mathbb R_*} \Phi_0(d \xi) \varphi(0,\xi)  - \int_0^t \int_{\mathbb R_*} \partial_s \varphi(s,\xi)  \Phi(s, d\xi ) ds \\ \nonumber
& = \frac{1}{2}   \int_0^t \int_{\mathbb R_*} \int_{\mathbb R_*} K_{a}(\xi, z) [\varphi(s,\xi+z) \zeta_R(\xi+ z)-\varphi(s,\xi)- \varphi(s,z) ] \Phi(s,d\xi) \Phi(s,dz) ds  \\
& - \int_0^t \int_{\mathbb R_*} \varphi(s,\xi) \Phi(s,d\xi)ds -  \frac{2}{1-\gamma} \int_0^t  \int_{\mathbb R_*} \Xi_\varepsilon (\xi) \partial_\xi \varphi(s,\xi) \xi \Phi(s,d\xi ) ds \nonumber \\
& + \frac{2}{1-\gamma} \int_0^t  \int_{\mathbb R_*} \Xi_\varepsilon (\xi) \varphi(s,\xi) \Phi(s,d\xi )+  \int_0^t \int_{\mathbb R_*} \varphi(s,\xi) \eta_\varepsilon(\xi) d\xi ds \nonumber
\end{align}
where we are assuming that $\gamma < 1$ and $\Phi_0 \in \mathcal M_{+,b}(\mathbb R_*)$ with $\Phi_0((0, \varepsilon ] \cup [ 2 R, \infty))=0$.

The proof of the existence of a solution of equation \eqref{weak time dep} is standard, in the sense that it is based on Banach fixed point theorem.
More precisely, we will use a change of variables to obtain an equation which is easier to analyze (cf \eqref{fixed point eq}). This equation looks complicated, but it does not contain any transport term, and, therefore, it is suitable for a fixed point argument.

To prepare the change of variables, we introduce the following notation. We denote by $X(t,x)$ the solution of the characteristic ODE
\begin{equation}\label{charact ODE}
\frac{ \partial X(t,x) }{\partial t}= - \beta X(t,x) \Xi_\varepsilon(X(t,x)), \quad X(0,x)=x . 
\end{equation}

We also introduce the function $\ell: [0,T] \times \mathbb R_* \times \mathbb R_* \rightarrow \mathbb R_+$ which is the function that satisfies
\begin{equation}\label{ell}
X(t,\ell(t,x,y))= X( t,x) + X( t,y). 
\end{equation}
The function $\ell$ is well defined because the map $x \mapsto X(t,x)$ is a diffeomorphism for every time $t.$

A time dependent truncation of \eu{parameter $R$} is a function $\theta_R \in C^\infty( [1, T] \times \mathbb R_*  \times \mathbb R_*)$ defined by
\begin{equation}\label{def time dep truncation}
\theta_{R}(t,x, y) :=\zeta_{R }(X( \ln t,x)+X(\ln t,y)) 
\end{equation} 
 where $\zeta_{R}$ is a truncation function of parameter $R$.

We also define a truncated time dependent kernel.
A time dependent coagulation kernel of parameter $a>0$ is a continuous map $K_{a,T}: [1, T]  \times \mathbb R_* \times \mathbb R_* \rightarrow \mathbb R_+$ defined by
\begin{equation} \label{time dep kernel non vs bounded kernel}
K_{a,T}(t,x,y):= t^{-2} K_a(X(\ln t , x),X(\ln t , y))
\end{equation}
where $K_a$ is a bounded coagulation kernel of bound $a$. 
We also introduce a new auxiliary source $\tilde{\eta_\varepsilon}$ defined by
\begin{align}\label{relation sources}
\tilde{\eta_\varepsilon} (t, x)=e^{- g(\ln t, x)} \frac{\partial X(\ln t,x)}{\partial x } \eta_\varepsilon(X(\ln t,x) ), \quad t>1, x >0,
\end{align}
where $g$ is defined by
\begin{equation}\label{g} 
 g(\tau,x):= \beta \int_0^{ \tau} \Xi_\varepsilon (X(s,x)) ds \ \text{ for every }\ \tau>0, x>0. 
\end{equation}
Notice that, by the \eu{change of variables} formula,  this implies that for every test function $\varphi$ and every time $t>0$
\begin{align}\label{change of var source}
\int_{\mathbb R_*} \varphi(\xi) \eta_\varepsilon(\xi) d\xi=
\int_{\mathbb R_*} \varphi(X(t ,x)) e^{g(t,x) }\tilde{\eta_\varepsilon} ( e^t, x) dx. 
\end{align}

\begin{proposition}\label{prop:exist_unique sol}
Let $T>1$, $\gamma <1 $, $\beta=2/(1-\gamma)$, and consider a source $\tilde{\eta_\varepsilon}$, a kernel $K_{a,T}$, a truncation $\theta_{R}$ and a truncation $\Xi_\varepsilon$. 
Let $\ell $ be the function defined by \eqref{ell}.
Consider an initial condition $f_1 \in \mathcal M_{+,b}(\mathbb R_*)$ with $f_1((0 , \varepsilon] \cup (2R, \infty ))=0.$
Then there exists a unique solution to the equation 
\begin{align} \label{fixed point eq}
&\int_{\mathbb R_*} \varphi(x) \dot{f}(t,dx) =  \int_{\mathbb R_*} \varphi(x) \tilde{\eta_\varepsilon}(t, x) dx  \\
&+ \int_{\mathbb R_*} \int_{\mathbb R_*} \frac{K_{a,T}(t,x,y)}{2} \left( \Lambda[\varphi] (t,x,y)- \varphi(x) e^{g(\ln t,y)} - \varphi(y) e^{g(\ln t,x)}\right)   f(t,dx) f(t,dy) \nonumber 
\end{align} 
for any $\varphi \in C_c(\mathbb R_*)$ and $t \in [1,T]$, 
with $f(1, \cdot)=f_1(\cdot)$ and where 
\begin{align}
&\Lambda[\varphi] (t,x,y):= \varphi( \ell(\ln t,x,y) ) e^{- g(\ln t, \ell (\ln t,x,y) )+ g(\ln t,x)+g(\ln t,y)} \theta_R(t, x, y) . 
\end{align}
The solution $f \in C^1([1,T], \mathcal M_{+,b}(\mathbb R_*))$ has the following properties for every $t \in [1,T]$
\begin{equation} \label{eq:bound for f}
\int_{\mathbb R_*} f(t, dx) \leq T \| \eta_\varepsilon \| + \| f_1\|,
\end{equation}
 and
\begin{equation} \label{support f}
f(t, (0,\varepsilon] \cup (2R  t^{ \beta }, \infty ))=0. 
\end{equation}
\end{proposition} 
Before starting with the proof of this proposition we provide some definitions and two auxiliary lemmas that help \eu{in its proof. }
Let us define the operator $\mathcal F$, that is a contraction, as it  will be shown in the proof of Proposition \ref{prop:exist_unique sol}, whose fixed point is the solution of \eqref{fixed point eq}.

Consider $f \in C([1,T], \mathcal M_{+,b}(\mathbb R_*))$ wih $f(1, \cdot)=f_1(\cdot)$. 
We denote by $b$ and $h_R$, the following expressions
\[
b[f](t,x):=\int_{\mathbb R_*} K_{a,T}(t,x,y) f(t,dy) e^{g(\ln t,x)} , \]
\[
h_R(t,s,x,y) := \theta_{R}(s,x,y) e^{- \int_s^t b[f] (\xi,x) d\xi }.
\]
The operator $\mathcal F [f] (t) : C_0(\mathbb R_*) \mapsto \mathbb R_*$ is defined by
\begin{align} \label{definition of F}
\langle \mathcal F [f] (t) , \varphi \rangle := \langle \mathcal F_1 [f] (t) , \varphi \rangle  + \langle \mathcal F_2 [f] (t) , \varphi \rangle  + \langle \mathcal F_3 [f] (t) , \varphi \rangle  
\end{align} 
for $t \in [1,T] ,$ where the operators $\mathcal F_i: C_0(\mathbb R_*) \mapsto \mathbb R_*$ are defined by
\[
\langle \mathcal F_1 [f] (t) , \varphi \rangle := \int_{\mathbb R_*} \varphi(x) e^{-\int_1^t b[f](s,x) ds } f_1(dx), 
\]
\begin{align*}
&\langle \mathcal F_2 [f] (t) , \varphi \rangle := \int_1^t \int_{\mathbb R_*} \int_{\mathbb R_*} \Lambda[\varphi](s,x,y) e^{- \int_s^t b[f] (\xi,x) d\xi } \frac{K_{a,T}(s,x, y)}{2} f(s,dx)f(s,dy)ds, 
\end{align*} 
\[
\langle \mathcal F_3 [f] (t) , \varphi \rangle := \int_{\mathbb R_*} \varphi(x) \int_1^t e^{-\int_s^t b[f](v,x) dv}  \tilde{\eta_{\varepsilon}}\left(s,x \right)ds dx.
\]

We define the set $\mathcal X_\varepsilon$ as 
\begin{equation}\label{xepsilon}
\mathcal X_\varepsilon:=\{ f \in \mathcal M_{+}(\mathbb R_*) : f((0,\varepsilon])=0\}. 
\end{equation} 
The set $\mathcal X_\varepsilon$ is a closed set both with respect to the \eu{weak$-*$} topology and the norm topology on $\mathcal M_{b}(\mathbb R_*)$, thus it is a Banach space with respect to the total variation norm.

\begin{lemma}\label{lem:continuity F} 
Assume $ \gamma,\beta , \eta_\varepsilon, K_{a,T}$, $\Xi_\varepsilon$, $\ell$ and $\theta_R$  to be as in Proposition \ref{prop:exist_unique sol}.  
The operator $\mathcal F$ defined by \eqref{definition of F}, for any initial condition $f_1\in \mathcal X_\varepsilon $,  maps 
$
C([1,T], \mathcal X_\varepsilon) 
$ into itself. 
\end{lemma}

Consider an initial condition $f_1$ for equation \eqref{fixed point eq}, we denote by $X_T $ the set defined by
\begin{equation} \label{set X_T} 
X_T:= \{ f \in C([1,T], \mathcal X_\varepsilon) : \|f -f_1 \|_{[1,T]} \leq 1 + \|f_1 \| \}. 
\end{equation}

\begin{lemma} \label{lem:F contraction}
Under the assumptions of Lemma \ref{lem:continuity F}, we deduce that if 
\begin{equation} \label{final bound for T}
T-1 \leq \frac{C(\eta_\varepsilon,a )}{1+ \| f_1\|}
\end{equation}
for a suitable constant $C(\eta_\varepsilon, a) >0$, then,  for every $f, g \in X_T$, it holds that
\begin{equation}\label{bound C_T}
\| \mathcal F[f](\cdot)-  \mathcal F[g](\cdot) \|_{[1,T]} \leq  C_T \| f-g\|_{[1,T]}
\end{equation}
with $0<C_T < \frac{1}{2}$ and 
\begin{align} \label{D_T}
\| \mathcal F[f_1]- f_1 \|_{[1,T]}  \leq  D_T \left( 1+\|  f_1 \|_{[1,T]}\right) 
\end{align}
with $0<D_T< \frac{1}{2}.$
\end{lemma}

The proofs of these Lemmas are postponed to the Appendix, as they are based on elementary sequences of inequalities. 
\begin{proof}[Proof of Proposition \ref{prop:exist_unique sol}]
Thanks to Lemma \ref{lem:continuity F} we already know that $\mathcal F$ maps $C([1,T], \mathcal X_\varepsilon)$ into itself. 

By Lemma \ref{lem:F contraction} we also know that if $T $ satisfies \eqref{final bound for T}, then $\mathcal F$ is a contraction and for every $f \in X_T$, 
\begin{align*}
\| \mathcal F[f]- f_1 \|_{[1,T]}   \leq \left(C_T+D_T\right) \left( 1+ \| f_1 \|_{[1,T]} \right) 
\end{align*} 
with $C_T+D_T <1.$

By Banach fixed point theorem we deduce that if $T $ satisfies \eqref{final bound for T}, then the operator $\mathcal F$ has a unique fixed point $f$ in $X_T$.  
Notice that if $f \in C([1,T], \mathcal X_\varepsilon) $, then $\mathcal F[f] \in C^1([1,T], \mathcal X_\varepsilon )$, therefore the map $f: [1,T] \rightarrow \mathcal X_\varepsilon$ is Fr\'echet differentiable. 
Differentiating $\mathcal F[f]=f$ we deduce that $f$ satisfies equation \eqref{fixed point eq}. (For the details of this computation we refer to the proof of Lemma 5.6 in \cite{vuoksenmaa2020existence}.)

For the moment we only know that the solution of equation \eqref{fixed point eq} exists if $T$ is small enough. 
Let us show, following the strategy of \cite{ferreira2019stationary} and \cite{vuoksenmaa2020existence}, that we can extend this solution to the whole line $[1, \infty ).$
To this end we first prove inequality \eqref{eq:bound for f}. 
Considering in \eqref{fixed point eq} a test function $\varphi \in C_c(\mathbb R_+)$ such that $0 \leq \varphi \leq 1 $ and $\varphi=1 $ on $[\varepsilon, 2T^{ \beta} R ]$ we obtain the following apriori estimate, for every $t>1$
\begin{align*}
&\int_{\mathbb R_*} \varphi(x) \dot{f}(t,dx) =  \int_{\mathbb R_*} \varphi(x) \tilde{\eta_\varepsilon}(t, x) dx \\
&+ \int_{(\varepsilon, \infty )} \int_{(\varepsilon, \infty )} \frac{K_{a,T}(t,x,y)}{2} \left( \Lambda[\varphi] (t,x,y) - \varphi(x) e^{g(\ln t,y)} - \varphi(y) e^{g(\ln t,x)} \right) f(t,dx) f(t,dy)\\
& \leq \|\eta_\varepsilon \|, 
\end{align*} 
where the last inequality is a consequence of \eqref{change of var source} and of the fact that, due to the definition of $\ell$ and the monotonicity of $\Xi_\varepsilon$, $\max \{  g(\ln t,x) ,  g(\ln t,y) \}- g(\ln t, \ell (\ln t,x,y) )\leq 0$, hence
\begin{equation}  \label{monotone g}
 \Lambda[\varphi] (t,x,y) - \varphi(x) e^{g(\ln t,y)} - \varphi(y) e^{g(\ln t,x)} \leq 0 \quad t >0, \quad x,y > \varepsilon
\end{equation}
for our choice of test function. 

Therefore, for every $t \in [1,T]$, inequality \eqref{eq:bound for f} holds and a unique solution of \eqref{fixed point eq} exists in the interval $[1,T_1]$ with $
T_1-1 :=\frac{C(\eta_\varepsilon,a )}{1+ \| f_1\|}. $
To extend the solution, preserving uniqueness and differentiability, we update the initial condition to $f_2(\cdot):=f (\frac{T_1}{2}, \cdot)$ and by \eqref{final bound for T} and \eqref{eq:bound for f} we deduce that there exists a unique solution on $[1,\frac{T_1}{2}+T_2]$ 
with 
\[
T_2-1:=  \frac{C(\eta_\varepsilon,a )}{1+ \| f_1\| + \frac{C(\eta_\varepsilon,a )\| \eta_\varepsilon \| } {1+ \| f_1\|} } \leq \frac{C(\eta_\varepsilon,a )}{1+ \| f_2\|}. 
\]
Iterating this argument, thanks to \eqref{eq:bound for f} and \eqref{final bound for T} we deduce that a unique solution exists on the whole real line. 
We refer to the proof of Proposition 5.8 in \cite{vuoksenmaa2020existence} for the details. 

 We show that $f(t, (2Rt^\beta, \infty )) =0$. Considering a test function $\varphi_n $ which approaches $\chi_{ (2Rt^\beta, \infty )}$ in equation \eqref{fixed point eq}, we deduce that $\int_{\mathbb R_*} \varphi_n(x) \dot{f} (t,dx) \leq 0, $ for every $n \geq 0$
and the desired conclusion \eu{follows by} Lebesgue's dominated convergence theorem. 

\end{proof}
\begin{proposition} \label{prop: exist time dep}
Let $\gamma, \beta, K_{a,T}, \eta_\varepsilon$, $\Xi_\varepsilon$ and $\theta_R$ be as in Proposition \ref{prop:exist_unique sol} and let $f$ denote the solution of equation \eqref{fixed point eq} with respect to $K_{a,T}$, $\eta_\varepsilon$, $\Xi_\varepsilon$ and $\theta_R$. 
Then, $\Phi \in C^1([0,T], \mathcal M_{+,b}(\mathbb R_*) ) $, defined, by duality, as the function with measure values that satisfies, for every $\varphi \in C_b(\mathbb R_*)$ and every $\tau \in [0,T]$, the equality
\begin{equation} \label{eq:Phi function of f} 
\int_{\mathbb R_*} \varphi(\xi) \Phi(\tau, d \xi)= \int_{\mathbb R_*} \varphi(X(\tau,x)) e^{- \tau +g(\tau,x) }f(e^\tau, dx),
\end{equation}  
where $g$ is given by \eqref{g},
is a solution of equation \eqref{weak time dep} with respect to $K_a$, $\eta_\varepsilon$, $\zeta_R$ and $ \Xi_\varepsilon$ and the initial condition $\Phi_0 \in \mathcal M_+(\mathbb R_*)$ with $\Phi_0((0,\varepsilon] \cup (2R, \infty ))=0.$
The solution $\Phi$ has the following properties for every $\tau \in [0,T]$,
 \[
\Phi(\tau, (0, \varepsilon] \cup (2R , \infty ))=0, 
\]
\begin{equation}\label{eq:bounds for Phi}
 \int_{\mathbb R_*} \Phi(\tau, d\xi) \leq  e^{\frac{1+\gamma}{1-\gamma} T } \left( e^T \| \eta_\varepsilon\|+\| \Phi_0\| \right). 
\end{equation}
\end{proposition}
\begin{proof}
By the change of variables \eqref{eq:Phi function of f}, for every $\varphi \in C^1([0,T], C^1_c(\mathbb R_*) )$
\begin{align*}
& \frac{d }{d \tau} \left( \int_{\mathbb R_*}  \varphi(\tau, \xi ) \Phi(\tau, d \xi)  \right) =  \frac{d}{d\tau} \left( \int_{\mathbb R_*} \varphi(\tau,X(\tau,x))  e^{- \tau + g(\tau,x)} f(e^\tau, dx)\right) \\
=&  \int_{\mathbb R_*} \frac{d}{d\tau}  \left(\varphi(\tau,X(\tau,x))  e^{- \tau +g(\tau,x)} \right)  f(e^\tau,dx) + \int_{\mathbb R_*} \varphi(\tau,X(\tau,x)) e^{ g(\tau,x)} \frac{d}{d e^\tau} f(e^\tau,dx) . 
\end{align*}
\eu{Expanding the first term and using the definition of $g$, \eqref{g}, we obtain}
\begin{align*}
& \int_{\mathbb R_*} \frac{d}{d\tau}  \left(\varphi(\tau,X(\tau,x))  e^{- \tau + \beta \int_0^\tau \Xi_\varepsilon (X(s,x))  ds} \right)  f(e^\tau,dx)  \\
&= \int_{\mathbb R_*}  [ \partial_1 \varphi(\tau,X(\tau,x))-\beta X(\tau,x) \Xi_\varepsilon( X(\tau,x) )  \partial_2 \varphi(\tau,X(\tau,x)) \\
&- \varphi(\tau,X(\tau,x)) + \beta \Xi_\varepsilon (X(\tau,x)) \varphi(\tau,X(\tau,x))  ]    e^{- \tau + \beta \int_0^\tau \Xi_\varepsilon (X(s,x)) ds} f(e^\tau,dx)  \\
&=  \int_{\mathbb R_*}  [ \partial_1 \varphi(\tau,z)-\beta z \Xi_\varepsilon( z )  \partial_2 \varphi(\tau,z) - \varphi(\tau,z) + \beta  \Xi_\varepsilon (z)\varphi(\tau,z) ] \Phi(\tau,dz). 
\end{align*} 
On the other hand, since $f$ is the fixed point of $\mathcal F$, choosing the test function 
\begin{equation}\label{psi_tau}
\psi(\tau,x):=\varphi(\tau, X(\tau,x))e^{g(\tau,x)} 
\end{equation}
in \eqref{fixed point eq} we deduce that 
\begin{align*}
&\int_{\mathbb R_*} \psi(\tau,x) \frac{d}{d e^\tau} f(e^\tau,dx) =  \int_{\mathbb R_*} \psi(\tau,x)  \tilde{\eta_\varepsilon}(e^\tau, x ) dx  \\
& \int_{\mathbb R_*} \int_{\mathbb R_*} \frac{K_{a,T}(e^\tau, x, y)}{2} \left( \Lambda [\psi(\tau,\cdot)] (e^\tau,x,y)- \psi( \tau, x) e^{g(\tau, y)}  - \psi( \tau, y)e^{g(\tau,x)} \right)\cdot \\
&\cdot  f(e^\tau,dx) f(e^\tau,dy)
\end{align*} 
which together with \eqref{time dep kernel non vs bounded kernel}, \eqref{change of var source} and \eqref{psi_tau} implies that 
\begin{align*}
&\int_{\mathbb R_*} \varphi(\tau, X(\tau,x))e^{g(\tau,x)}  \frac{d}{d e^\tau} f(e^\tau,dx) = \int_{\mathbb R_*} \varphi(\tau, X(\tau,x))e^{g(\tau,x)} \tilde{\eta_\varepsilon}(e^\tau, x ) dx \\ 
& + \int_{\mathbb R_*} \int_{\mathbb R_*}\frac{K_{a,T}(e^\tau, x, y)}{2} e^{g(\tau,y) + g(\tau,x)}  [ \varphi(\tau, X(\tau,x) + X(\tau,y) ) \theta_R(e^\tau,x,y) \\
& - \varphi(\tau, X(\tau,x))- \varphi(\tau, X(\tau,y))] f(e^\tau,dx) f(e^\tau,dy) \\
&=  \int_{\mathbb R_*} \varphi(\tau, 	\xi ) \eta_\varepsilon(\xi  ) d\xi\\
&+ \int_{\mathbb R_*} \int_{\mathbb R_*}\frac{K_{a}(X(\tau,x) , X(\tau,y)) }{2} [ \varphi(\tau, X(\tau,x) + X(\tau,y) ) \zeta_R(X(\tau,x)+X(\tau,y)) \\
& - \varphi(\tau, X(\tau,x))- \varphi(\tau, X(\tau,y))] e^{- \tau +  g(\tau,x)}  f(e^\tau,dx) e^{-\tau + g(\tau,y)}  f(e^\tau,dy). \\
\end{align*} 
By the  \eu{change of variables} \eqref{eq:Phi function of f}, we deduce that 
\begin{align*}
& \int_{\mathbb R_*} \varphi(\tau, X(\tau,x))e^{g(\tau,x)} \frac{d}{d e^\tau} f(e^\tau,dx) = \int_{\mathbb R_*} \varphi( \tau, \xi) \eta_{\varepsilon}(\xi ) d\xi \\
&= \int_{\mathbb R_*} \int_{\mathbb R_*} \frac{K_a(\xi,z)}{2} \left[ \zeta_R(z+\xi) \varphi(\tau, z+\xi) - \varphi( \tau, \xi) - \varphi(\tau, z ) \right] \Phi(\tau,d \xi) \Phi(\tau,dz). 
\end{align*} 
Summarizing, we have proven that $\Phi$ satisfies \eqref{weak time dep} for all $\varphi \in  C^1([0,T], C^1_c(\mathbb R_*) )$.

Since $f$ is a solution of equation \eqref{fixed point eq}, then \eqref{support f} holds. 
Considering a test function $\varphi=1 $ in $(0,\varepsilon] \cup (2Rt^\beta, \infty) $ in \eqref{eq:Phi function of f} and recalling that, if $x > 2\varepsilon $, then $X(t,x) = t^{-\beta } x $ and for every $\varepsilon \geq x>0$ and every $t>1$ we have that $X(t,x) = x $ we conclude that for every $\tau \in [0,T]$,
\[
\Phi(\tau, (0,\varepsilon] \cup  (2R, \infty ) )=0. 
\]

By inequality \eqref{eq:bound for f} and by the fact that $f_1 = \Phi_0$, we deduce that  
\begin{align*}
&\sup_{ \tau \in [0, T] } \int_{\mathbb R_*} \Phi(\tau, d\xi) = \sup_{ \tau \in [0, T] } \int_{\mathbb R_*} e^{- \tau + \beta \int_0^\tau  \Xi_\varepsilon (X(s,x)) ds} f(e^\tau, dx) \\
& \leq  \sup_{t \in [1, e^T] } \int_{\mathbb R_*} t^{-1 +\beta} f(t,dx) \leq e^{ \frac{1+\gamma }{1-\gamma}T}\left(  e^T \|\eta_\varepsilon \| + \| \Phi_0 \|\right). 
\end{align*} 
The bound \eqref{eq:bounds for Phi} follows. 
\end{proof}

\subsection{Existence of steady state solutions for the truncated coagulation equation with source written in self-similar variables} \label{sec:steady state}
In this Subsection we prove the existence of steady state solutions of equation \eqref{weak time dep} \eu{(see Proposition \ref{pro:stat solution}).}

\begin{definition} \label{steady state eq} 
Let $\gamma < 1$. 
We say that a measure $\Phi \in \mathcal M_{+, b}(\mathbb R_*)$
is a steady state solution of equation \eqref{weak time dep} with respect to $K_a, \zeta_R, \Xi_\varepsilon$ and $\eta_\varepsilon$, if it solves the following equation for every $\varphi \in C^1_c(\mathbb R_*) $
\begin{align} \label{stationary solution trunc}
&  \int_{\mathbb R_*}\varphi(x) \Phi(dx)=\int_{\mathbb R_*} \varphi (x) \eta_\varepsilon (x) dx +  \frac{2}{1-\gamma} \int_{\mathbb R_*} \Xi_\varepsilon(x)\left(  \varphi(x)- x \varphi'(x) \right) \Phi (dx)  \\
&+  \int_{\mathbb R_*}  \int_{\mathbb R_*} \frac{ K_{a}(x,y)}{2} [\zeta_R(x+y)\varphi(x+y)-\varphi(x) -\varphi(y) ] \Phi(dx) \Phi(dy). \nonumber
\end{align}
\end{definition}
In the following we will denote by $\chi_{z,n} $ the mollification of the characteristic function $\chi_{(0, z]}$, 
\begin{align}\label{mollified char} 
\chi_{z,n} (x)&= \int_{\mathbb R_*} \chi_{(0,z]} (x-y) \rho_n(y) dy
\end{align}
where $\rho_n$ are the mollifiers considered in \cite{giusti1984minimal}. 
In the following, we use the notation
\begin{equation} \label{eq:g}
g_{z,n}(x) :=\int_{x-z}^{\frac{1}{n}}  c_n  e^{\frac{1}{(n y)^2-1}} dy . 
\end{equation}
Classical results for mollifiers yield $\chi_{z,n} \rightarrow \chi_{[0,z] } \text{ in } L^1(\mathbb R_*).$

\begin{proposition} \label{pro:stat solution}
Let $\gamma <1.$
There exists a steady state solution for equation \eqref{weak time dep} corresponding to $K_a$, $\eta_\varepsilon$, $\Xi_\varepsilon$ and $\zeta_R$, $\Phi \neq 0$, as in \eqref{steady state eq}, satisfying 
\begin{equation}\label{support steady state}
\Phi((0, \varepsilon] \cup (2R, \infty ))=0, 
\end{equation}
\begin{equation} \label{ineq:first mom trunc}
0<\int_{\mathbb R_*} x \Phi(dx) \leq  1.
\end{equation} 
\end{proposition}
Before proceeding with the proof of Proposition \ref{pro:stat solution} we present and prove two auxiliary results.
\begin{lemma}\label{lem:inv region}
Under the assumptions of Proposition \ref{prop: exist time dep}, we have that each solution $\Phi\in C^1([1,T], \mathcal M_{+,b}(\mathbb R_*))$ of \eqref{weak time dep} corresponding to an initial condition $\Phi_0 \in \mathcal M_{+,b}(\mathbb R_*)$ such that 
$ \int_{\mathbb R_*} \xi \Phi_0(d\xi) \leq 1$, 
and to $K_a,\eta_\varepsilon,\zeta_R, \Xi_\varepsilon$ satisfies $\int_{\mathbb R_*} \xi \Phi(t, d\xi) \leq 1$.
\end{lemma}
\begin{proof}
We consider the test function $\varphi_n^M(\xi):= \xi \chi_{M,n}(\xi)$ , with $\chi_{M,n} $ defined by \eqref{mollified char} in equation \eqref{weak time dep} and
we pass to the limit as $M$ tends to infinity in all the terms of \eqref{weak time dep}. 
Firstly, we notice that
\[
 \int_{(M, M+1/n]}\xi g_{M,n}(\xi)   \Phi(t,d\xi)  \rightarrow 0 \text{ as } M \rightarrow \infty
\] 
indeed, for every $\xi \in \mathbb R_*$, we have $\chi_{(M, M+1/n]} (\xi) g_{M,n}(\xi)  \leq 1  $
and 
\[
\int_{\mathbb R_*} \xi \Phi(t, d\xi) \leq 2R  \Phi(t, \mathbb R_*) < \infty. 
\]
Therefore, by Lebesgue's dominated convergence theorem, 
\[
\lim_{M \rightarrow \infty } \int_{\mathbb R_*}\varphi_n^M(\xi) \Phi(t,d\xi )  = \int_{\mathbb R_*} \xi \Phi(t, d\xi).
\] 
A similar argument can be used to prove that 
\[
 \int_{\mathbb R_*} \varphi^M_n (\xi)\eta_\varepsilon(\xi) d\xi ds \rightarrow   \int_{\mathbb R_*} \xi\eta_\varepsilon(\xi) d\xi \text{ as } M \rightarrow \infty. 
\]
If $M > 2R$, then 
\begin{align*}
 \varphi^M_n (\xi+z)\zeta_R(\xi+ z)-  \varphi^M_n(\xi)-  \varphi^M_n(z)= -\xi \chi_{M,n}(\xi)- z \chi_{M,n}(z). 
\end{align*} 
This implies that 
\begin{align*}
\lim_{M \rightarrow \infty} &  \int_0^t \int_{\mathbb R_*} \int_{\mathbb R_*} K_a(\xi,z) [ \varphi^M_n(\xi+z)\zeta_R(\xi+z)- \varphi^M_n(\xi) \\
&-  \varphi^M_n(z)] \Phi(s,d\xi) \Phi(s,dz) ds \leq 0.
\end{align*}
Let us now consider the term 
\[
\int_{[0, M+1/n]}\xi^2  \Xi_\varepsilon (\xi) \partial_\xi \left( \chi_{M,n}(\xi) \right) \Phi(t,d\xi )=  \int_{(M,M+1/n ]}   \Xi_\varepsilon (\xi)\xi^2  \partial_\xi \left( g_{M,n}(\xi) \right) \Phi(t,d\xi ).
\]
Since $g_{M,n}'(x) =  -  c_n  e^{\frac{1}{(n (x-M))^2-1}  } $, by Lebesgue's dominated convergence theorem we can conclude that
\begin{align*}
\lim_{M \rightarrow \infty}  \int_{\mathbb R_*}   \Xi_\varepsilon (\xi)  \chi_{(M, M+1/n]} (\xi) \xi^2  \partial_\xi \left( g_{M,n}(\xi) \right) \Phi(t,d\xi ) = 0. 
\end{align*} 

Passing to the limit as $M \rightarrow \infty $ in equation \eqref{weak time dep}, thanks to the specific form of the truncation for small particles in \eqref{time dep eq with eps_source}, we deduce that
\begin{align*} 
 \partial_t \int_0^\infty x \Phi(t, dx) \leq 1 - \int_0^\infty x \Phi(t,dx).
\end{align*}
Therefore, \eu{ since by assumption $\int_0^\infty x \Phi_0(dx) \leq 1 $, we deduce that}  $ \int_0^\infty x \Phi(t, dx) \leq 1 $ for every $t>0$. 
\end{proof}
\begin{lemma}\label{lem:for weak continuity}
Under the assumptions of Proposition \ref{pro:stat solution}, for every $T>0$ there exists a unique solution $\varphi \in C^1([0,T], C^1_c(\mathbb R_*)) $, with $\varphi(T, \cdot):=\psi \in C^1_c(\mathbb R_*)$, of the equation
\begin{align} \label{eq to solve for uniqueness} 
&\partial_s \varphi(s,x) - \varphi(s,x)  + \frac{2}{1-\gamma} \Xi_\varepsilon(x) \left( \varphi(s,x) - x \partial_x \varphi(s, x)\right) +  \eu{\mathbb L} [\varphi](s,x) =0, 
\end{align} 
 where $\eu{\mathbb L}$ is defined by
\begin{equation}\label{formula L}
\eu{\mathbb L}[\varphi] (s,x):= \frac{1}{2} \int_{\mathbb R_*} K_a(x,y) \left( \varphi(s,x+y) \zeta_R(x+y) - \varphi(s,x) -\varphi(s,y)  \right) \mu_s(dy)
\end{equation}
where $\mu_s=\Phi(s,\cdot)+ \Psi(s,\cdot)$ and $\Phi$ and $\Psi$ are two solutions of \eqref{weak time dep} with initial conditions $\Phi_0$ and $\Psi_0$ respectively.
\end{lemma}
\begin{proof}
Equation \eqref{eq to solve for uniqueness} includes a transport term. Therefore we proceed by integrating along the characteristics. 
Since the associated ODE is \eqref{charact ODE}, 
We can rewrite equation \eqref{eq to solve for uniqueness} as 
\[
\frac{d }{ds}  \varphi(s, X(s,y) ) =\varphi(s,X(s,y) )  - \eu{\mathbb L} [\varphi](s,X(s,y)) - \frac{2}{1-\gamma} \Xi_\varepsilon (X(s,y)) \varphi(s,X(s,y))  . 
\] 
We aim now at applying Banach fixed point theorem. 
 We therefore rewrite the equation in a fixed point form $ \varphi(s, X(s,y) ) =\mathcal T [\varphi ] (s,X(s,y) )$
where 
\begin{align*}
&\mathcal T [\varphi ] (s,y)  := \psi(X(t,y)) +  \int_s^t \left(   \frac{2}{1-\gamma} \Xi_\varepsilon (X(r,y))  \varphi(r,X(r,y))-  \varphi(r,X(r,y)) \right)  dr  \\
&+  \int_s^t \eu{\mathbb L} [\varphi](r,X(r,y)) dr. 
\end{align*} 
Recall that $\varphi \in Y:=C([0,T], C_c (\mathbb R_*))$ and $Y$ is a Banach space with the norm $\| \cdot  \|_Y:= \sup_{[0,T]} \sup_{\mathbb R_*} | \cdot |  $. 
We show that the operator $\mathcal T :Y \rightarrow Y$ is a contraction.

Let us consider $\varphi_1, \varphi_2 \in Y$, then 
\begin{align*} 
&  \mathcal T[\varphi_1](s,y) - \mathcal T [\varphi_2](s,y)= 
\int_s^t  (\eu{\mathbb L} [\varphi_1] - \eu{\mathbb L} [\varphi_2]) (r, X(r,y))   dr \\
&+ \int_s^t  \left( \frac{2  \Xi_\varepsilon(X(r,y)) }{1-\gamma}(\varphi_1(r,X(r,y))- \varphi_2(r,X(r,y))) - (\varphi_1(r,X(r,y))- \varphi_2(r,X(r,y))) \right) dr.  
\end{align*} 
\eu{
We start estimating the last term
\begin{align} \label{ineq first term}
 &\left| \int_s^t  \left( \frac{2  \Xi_\varepsilon(X(r,y)) }{1-\gamma}(\varphi_1(r,X(r,y))- \varphi_2(r,X(r,y))) - (\varphi_1(r,X(r,y))- \varphi_2(r,X(r,y))) \right) dr\right| \nonumber \\
 &  \leq T \left| \frac{3-\gamma}{1-\gamma} \right| \|\varphi_1 - \varphi_2 \|_Y 
\end{align}
and, then we analyse the first term
\begin{align*} 
& \int_s^t (\mathbb L [\varphi_1] - \mathbb L [\varphi_2]) (r, X(r,y)) dr \\
& \leq \frac{a}{2} \int_s^t \int_0^\infty \left[\left( \varphi_1-\varphi_2\right)(r,x+X(r,y)) \zeta_R(x+y)- \left(\varphi_1-\varphi_2\right)(r,x) \right. \\
& \left.- \left(\varphi_1-\varphi_2\right)(r,X(r,y))\right] \mu_r(dx) \\
& \leq \frac{3 a}{2}  \| \varphi_1 - \varphi_2 \|_Y \int_s^t \mu_r(\mathbb R_*) dr= \frac{3 a}{2}  \| \varphi_1 - \varphi_2 \|_Y \int_s^t  (\Phi(r, \mathbb R_*) + \Psi(r, \mathbb R_*) )dr. 
\end{align*} 
To prove that $\mathcal T$ is a contraction, we need to control the quantity 
\[
\int_s^t (\Phi(r, \mathbb R_*) + \Psi(r, \mathbb R_*) )dr. 
\] }

Thanks to \eqref{eq:bounds for Phi} and to the fact that $\gamma <1$, we know that for every $r>0$
\[
\Psi(r, \mathbb R_*) \leq e^{r \frac{1+\gamma}{1-\gamma}}\left( \|\Psi_0 \| + e^r \|\eta_\varepsilon \| \right) \leq e^{r \frac{2}{1-\gamma}}\left( \|\Psi_0 \| + \|\eta_\varepsilon \| \right) 
\] 
and 
\[
 \Phi(r, \mathbb R_*) \leq  e^{r \frac{1+\gamma}{1-\gamma}}\left( \|\Phi_0 \| +e^r \|\eta_\varepsilon \| \right) \leq e^{r \frac{2}{1-\gamma}}\left( \|\Phi_0 \| + \|\eta_\varepsilon \| \right) .
\]

Therefore,
\begin{align*} 
& \int_s^t (\Phi(r,\mathbb R_*) + \Psi(r,\mathbb R_*) ) dr \leq 2 \left(\max\{\|\Psi_0 \| ,\|\Phi_0 \|  \}  + \|\eta_\varepsilon \| \right) \int_s^t   e^{r\frac{2}{1-\gamma}} dr \\
&=2 \left( \max\{\|\Psi_0 \| ,\|\Phi_0 \|  \}  +\|\eta_\varepsilon \| \right) \frac{1-\gamma}{3-\gamma}\left(  e^{t \frac{3-\gamma}{1-\gamma} } -  e^{s \frac{3-\gamma}{1-\gamma}}\right).
\end{align*}  
Since for any $\alpha >0$ and $1> t > s >0$ then $e^{\alpha t}- e^{\alpha s} \leq e^{\alpha t}- 1 \leq t (e^\alpha-1)$,
we conclude, \eu{recalling \eqref{ineq first term},} that for every $T \geq 0$
\begin{align*} 
& \| \mathcal T[\varphi_1] - \mathcal T [\varphi_2] \|_Y \leq\| \varphi_1 - \varphi_2 \|_Y \left( \eu{\left| \frac{3-\gamma}{1-\gamma} \right|| }T
+  \frac{3 a}{2}  \int_s^t (\Phi(\mathbb R_*, r) + \Psi(\mathbb R_*, r) )dr \right) \\
& \leq c T \| \varphi_1 - \varphi_2 \|_Y, 
\end{align*} 
where
\[
c:=\eu{\left| \frac{3-\gamma}{1-\gamma} \right|}+ 3 a  \frac{1-\gamma}{3-\gamma} \left( e^{\frac{3-\gamma}{1-\gamma}}-1 \right) \left( \max\{\|\Psi_0 \| ,\|\Phi_0 \|  \} +\|\eta_\varepsilon \| \right). 
\]
If $T< \min\left\{ \frac{1}{c}, 1\right\}$, then the operator $\mathcal T$ is a contraction. 
By Banach fixed point theorem we conclude there exists a unique solution to the equation $\varphi(t, \cdot) = \mathcal T \varphi (\cdot, t )$ for $t \in [0,  T]$.

Since $\mathcal T \varphi(\cdot, x) $ defines a differentiable function for every $x \in \mathbb R_*$ we conclude that $\varphi \in C^1([0, T ], C_c(\mathbb R_*))$. 
We can now extend the existence of a fixed point for the operator $\mathcal T$ on $C^1([0,T], C_c(\mathbb R_*))$ for an arbitrary $T>0.$ 
For this it is enough to notice that, since $c$ and $T$ do not depend on the initial condition $\psi$, then we can iterate the argument and deduce that there exists a unique solution of \eqref{eq to solve for uniqueness}. 

\end{proof}
\begin{proof}[Proof of Proposition \ref{pro:stat solution}]
We introduce the semigroup $\{S(t)\}_{ t \geq 0 }  $ with values in $\mathcal M_{+,b}(\mathbb R_*)$ and defined by $S(0)\Phi_0(\cdot) =\Phi_0(\cdot) $  and $S(t) \Phi_0(\cdot)=\Phi(t, \cdot )$, where $\Phi$ is the solution of \eqref{weak time dep} with respect to $K_a$, $\Xi_\varepsilon$, $\eta_\varepsilon$, $\zeta_R$ and the initial condition $\Phi_0$, defined by \eqref{eq:Phi function of f}. 
Namely, $\Phi $ is the measure defined by equality \eqref{eq:Phi function of f} for every $\varphi \in C_c(\mathbb R_*)$ with $f$ being the unique solution of the fixed point equation \eqref{fixed point eq}

We split the proof into steps: first of all we show the existence of a \eu{weak$-*$} compact invariant region. Then we prove the \eu{weak$-*$} continuity of the operator $\Phi_0 \mapsto S(t) \Phi_0. $
By Tychonoff fixed point theorem we conclude that for every $t>0$ the operator $S(t)$ has a fixed point $\hat{\Phi}_t$. As a last step, we show that a steady state of \eqref{weak time dep} \eu{as in Definition \ref{steady state eq},} can be obtained from $\hat{\Phi}_t$ by passing to the limit as $t$ goes to zero.\\

\textbf{Step 1: Existence of an invariant region} 

Let us consider the set $P \subset \mathcal M_{+,b}(\mathbb R_*)$ defined by
\[
P :=\left\{ H \in \mathcal M_{+,b}(\mathbb R_*) : H((0, \varepsilon] \cup (2R, \infty ))=0, \int_0^\infty x H(dx) \leq 1\right\}. 
\]

Notice that
$
P \subset B\left(0,\frac{1}{\varepsilon}\right) := \left\{ H \in \mathcal M_+(\mathbb R_*) : \| H \| \leq \frac{1}{\varepsilon} \right\}
$
\eu{where $\| \cdot\|$ is the total variation norm.}
By Banach-Alaoglu theorem we conclude that $P$ is compact in the \eu{weak$-*$} topology, since it is a closed subset of the set $B\left(0,\frac{1}{\varepsilon}\right) $, which is compact in the \eu{weak$-*$} topology. 

Proposition \ref{prop: exist time dep} implies that if $ \Phi_0((0, \varepsilon] \cup (2R, \infty))=0$, then $S(t) \Phi_0((0, \varepsilon] \cup (2R,\infty))=0$.  
By Lemma \ref{lem:inv region} we conclude that $P$ is an invariant region.

\textbf{Step 2: \eu{weak$-*$} continuity} 

To be able to apply Tychonoff fixed point theorem, we need  to check that the map $\Phi_0 \mapsto S(t) \Phi_0 $
is continuous in the \eu{weak$-*$} topology. 

To this end, it is enough to show that, for every test function $\psi \in  C_c(\mathbb R_*)$, 
\begin{align} \label{eq for uniqueness}
\int_{\mathbb R_*} \psi(x) ( \Phi- \Psi)(t,dx)  =  \int_{\mathbb R_*} \psi(x) ( \Phi_0- \Psi_0)(dx). 
\end{align}  
where $\Phi$ and $\Psi$ are \eu{two solutions of \eqref{weak time dep}} corresponding to the two initial conditions $\Phi_0$ and $\Psi_0$, respectively.

With this aim, we notice that the measure $(\Phi- \Psi)(t,\cdot)$ is a solution of 
\begin{align*} 
& \int_{\mathbb R_*} \varphi(t,x) ( \Phi- \Psi)(t,dx) = \int_{\mathbb R_*} \varphi(t,x) ( \Phi_0- \Psi_0)(dx)\\
&+  \int_0^t \int_{\mathbb R_*} \partial_s \varphi(s,x) ( \Phi- \Psi)(s,dx) ds + \int_0^t \int_{\mathbb R_*} \mathcal L [\varphi](x,s) ( \Phi- \Psi)(s,dx) ds   \\
& + \ \int_0^t \int_{\mathbb R_*}\left[  \frac{2 \Xi_\varepsilon(x)}{1-\gamma} \left(  \varphi(s, x) -   \partial_x \varphi(s, x) x \right)  - \varphi(x) \right] ( \Phi- \Psi)(s,dx)ds  \\
\end{align*} 
where $\mathcal L $ is given by \eqref{formula L}. 
By Lemma \ref{lem:for weak continuity} we conclude that \eqref{eq for uniqueness} holds and, hence, the map $\Phi_0 \mapsto S(t) \Phi_0 $ is \eu{weak$-*$} continuous. 

\textbf{Step 3: Time continuity} 

The function $S(t)$ has a fixed point $\hat{\Phi}_t$  for every time $t \geq 0.$
We now show that the map 
\begin{equation} \label{cont S}
t \rightarrow S(t) \Phi_0 
\end{equation}
is \eu{weak$-*$} continuous for any $\Phi_0 $ in $\mathcal M_+(\mathbb R_*)$. 

Since $\Phi $ solves \eqref{weak time dep}, for every $\varphi \in C^1_c (\mathbb R_*)  $ it holds that
\begin{align*}
& \int_0^\infty  \varphi(x) [\Phi(\tau_1 , dx )- \Phi(\tau_2, dx )] = 
\int_{\tau_1}^{\tau_2} \int_{\mathbb R_*} \varphi(x) \Phi(\tau,dx) \\
& - \frac{2}{1-\gamma} \int_{\tau_1}^{\tau_2} \int_{\mathbb R_*} \Xi_\varepsilon(x) \partial_x \varphi(x) x \Phi(\tau,dx) d\tau +\int_{\tau_1}^{\tau_2} \int_{\mathbb R_*} \varphi(x)\eta_{\varepsilon}(dx)   \\
& +  \frac{1}{2} \int_{\tau_1}^{\tau_2} \int_{\mathbb R_*} \int_{\mathbb R_*} K_a(x,y) [\varphi(x+y)-\varphi(x) -\varphi(y) ] \Phi(\tau, dx ) \Phi(\tau, dy)d\tau  \\
& + \frac{2}{1-\gamma} \int_{\tau_1}^{\tau_2} \int_{\mathbb R_*} \Xi_\varepsilon(x)  \varphi(x) x \Phi(\tau,dx) d\tau \leq  (\tau_2-\tau_1) C(\eta_\varepsilon, \varphi, \gamma) + (\tau_2-\tau_1)^2 C(a, \varphi, \gamma) 
\end{align*}
where $C(\eta_\varepsilon, \varphi, \gamma) $ and $C(a, \varphi, \gamma) $ are positive constants. 
Therefore, the function \eqref{cont S}
 is continuous in the \eu{weak$-*$} topology. 

Since the set $P$ is compact and metrizable and the map \eqref{cont S} is continuous, we conclude by Theorem 1.2 in \cite{escobedo2005self}, that there exists a measure $\hat{\Phi}$ such that $S(t) \hat{\Phi} = \hat{\Phi}.$ The measure $\hat{\Phi}$ is a solution of equation \eqref{stationary solution trunc}.

Finally, $\Phi \neq 0$ because $0$ does not solve \eqref{stationary solution trunc}, whence $\int_0^\infty x \Phi(dx) >0.$
\end{proof}
\subsection{Properties of the steady state solutions} \label{sec:properties steady state}
We state now some important properties of the solutions of equation \eqref{stationary solution trunc}. 
\begin{lemma}[Regularity] \label{lem:regularity}
If $\Phi\in \mathcal M_+(\mathbb R_*)$ \eu{is a solution of equation \eqref{stationary solution trunc}} with respect to $K_a$, $\eta_\varepsilon$, $\Xi_\varepsilon$ and $\zeta_R$, as in \eqref{steady state eq}, then $\Phi \ll \mathcal L$. 
\end{lemma} 
\begin{proof}
We follow a similar strategy to the one used in \cite{kierkels2016self}.
If we consider the test function
\begin{equation} \label{test function regularity}
\varphi(x):=- \int_{(x, \infty ) } \ \frac{1}{z} \chi(z) dz,
\end{equation}
with $\chi \in C_c(\mathbb R_*)$  in \eqref{stationary solution trunc}, then we obtain that, for any $p, q \in \mathbb R_*$ such that $1/p + 1/q=1$, 
\begin{align*}
  \frac{2}{1-\gamma} \left| \int_{\mathbb R_*} \Xi_\varepsilon (x) \chi(x) \Phi (dx)\right| & \leq \left( \eta_\varepsilon ({\mathbb R_*})+ \frac{\left| 1+\gamma\right|}{1-\gamma}\Phi({\mathbb R_*})   +  \frac{3a}{2} \Phi(\mathbb R_*)^2  \right)  \left\| \frac{1}{z} \right\|_{L^p(\mathcal K )} \|\chi \|_{L^q(\mathcal K )} 
\end{align*}
where $\mathcal K $ is the support of $\chi$.
By the density of $C_c(\mathbb R_*)$  in $L^q(\mathbb R_*)$, we conclude that for any $q < \infty$ and any compact set $\mathcal K $ 
\begin{align*}
 \left| \int_{\mathbb R_*} \chi(x) \Xi_\varepsilon (x) \Phi (dx)\right| \leq  C(\mathcal K, \gamma, \varepsilon, \Phi) \|\chi \|_{L^q(\mathcal K)} \quad \forall \chi \in L^q(\mathcal K).
\end{align*} 
This implies that the measure $ \Xi_\varepsilon (x)  \Phi(dx) $ is absolutely continuous with respect to the Lebesgue measure. Thus $\Phi$
is absolutely continuous with respect to the Lebesgue measure on $(0,\varepsilon)$ and since $\Phi((0,\varepsilon])=0$ we deduce that  hence $\Phi \ll \mathcal L $. 
\end{proof}

 \begin{lemma}\label{lem:flux a_eps_R}
\eu{Every steady state solution $\Phi$ of \eqref{weak time dep}, defined as in Definition \ref{steady state eq}}, with density $\phi$, corresponding to $K_a$, $\eta_\varepsilon$, $\Xi_\varepsilon$ and $\zeta_R $, satisfies the following inequality
\begin{align} \label{eq:truncated flux}
\int_0^z \int_{z-x}^\infty K_a(x,y) x \phi(x) \phi(y) dy dx   \leq  \int_0^z  x \eta_\varepsilon(dx) + \frac{2}{1-\gamma}  z^2 \phi (z), 
\end{align}
for almost every $z >0.$
\end{lemma}
\begin{proof}
If we consider the test function $\varphi^n_z(x) = x \chi_{z,n}(x)$, with $\chi_{z,n}$ given by formula \eqref{mollified char}, in equation \eqref{stationary solution trunc}, we obtain that
\begin{align*}
& \int_0^{ z+ \frac{1}{n}} \varphi^n_z(x) \eta_\varepsilon (x) dx  - \int_0^{ z+\frac{1}{n}} \varphi^n_z(x)  \phi(x) dx - \frac{2}{1-\gamma} \int_0^{ z+\frac{1}{n}} \Xi_\varepsilon(x)  x  {\varphi^n_z}^{'}(x) \phi(x)  dx \\
& + \frac{2}{1-\gamma} \int_0^{ z+\frac{1}{n}} \Xi_\varepsilon(x)   {\varphi^n_z}(x) \phi(x)  dx \\
& =- \frac{1}{2} \int_0^{ z+\frac{1}{n}}\int_0^{z+\frac{1}{n}} K_a(x,y) [ \zeta_R (x+y) \varphi^n_z(x+y) - \varphi^n_z (x) - \varphi^n_z (y) ] \phi(x) dx \phi(y) dy .
 \end{align*}
Applying Lebesgue's dominated convergence theorem we prove that 
\[
\lim_{n \rightarrow \infty } \int_0^{\infty } \varphi^n_z(x)  \eta_\varepsilon (dx)  =\int_0^z  x \eta_\varepsilon (dx),  
\lim_{n \rightarrow \infty } \int_0^{\infty } \varphi^n_z (x) \phi(x)  dx = \int_0^z x \phi(x)  dx
\] 
and 
\[
\lim_{n \rightarrow \infty } \int_0^{\infty } \Xi_\varepsilon (x) \varphi^n_z (x) \phi(x) dx = \int_0^z \Xi_\varepsilon (x) x \phi(x)  dx . 
\]

 We now aim at showing that 
\begin{align*}
 \int_0^{ z+1/n} \Xi_\varepsilon(x) x \left( x \chi_{z,n}(x) \right)' \phi(x) dx \rightarrow \int_0^z  \Xi_\varepsilon(x)  x \phi(x) dx -  \Xi_\varepsilon(z) z^2 \phi(z),
\end{align*}
as $n \rightarrow \infty. $
Notice that $\chi'_{z,n} (x)= \rho_n(x-z)$, where $\rho_n$ are the mollifiers introduced in Section \ref{sec:steady state}.
As a consequence of the properties of the mollifiers, we know that for every $f \in L^1(\mathbb R_*)$
\[
\|\rho_n * f - f\|_1 \rightarrow 0 \text{ as } n \rightarrow \infty 
\]
where with $\| \cdot \|_1$ we denote the $L^1$ norm and with $* $ the classical convolution product. 
This implies that, up to a subsequence, 
\[
\int_{\mathbb R_*} \Xi_\varepsilon(x) x^2 \phi(x) \chi'_{z,n} (x) dx =\int_{\mathbb R_*}  \Xi_\varepsilon(x)  x^2 \phi(x) \rho_n(x-z) dx \rightarrow  \Xi_\varepsilon(z) z^2 \phi(z) \text{ a.e }
\]
as $n$ goes to infinity.

On the other hand, it is possible to prove as in \cite{ferreira2019stationary} (proof of Lemma 2.7) that 
\begin{align*}
 & - \lim_{n \rightarrow \infty }  \int_0^{ z+ \frac{1}{n}}\int_0^{ z + \frac{1}{n}} K_a(x,y) [\varphi^n_z(x+y)  -\varphi^n_z(x) -\varphi^n_z(y)]  \phi(x) dx\phi(y) dy \\
&= \int_0^z \int_{z-x}^\infty K_a(x,y) x \phi(x) dx \phi(y) dy  
\end{align*} 
as $n$ goes to infinity. 
Since, by definition of the truncation we have $\zeta_R \leq 1 $, then
\begin{align*}
 & - \int_0^{ z+ \frac{1}{n}}\int_0^{ z + \frac{1}{n}} K_a(x,y) [ \zeta_R(x+y) \varphi^n_z(x+y) - \varphi^n_z(x) - \varphi^n_z(y) ] \phi(x) dx\phi(y) dy \\
 &\geq -  \int_0^{ z+ \frac{1}{n}}\int_0^{ z + \frac{1}{n}} K_a(x,y) [\varphi^n_z (x+y)  - \varphi^n_z(x) - \varphi^n_z (y) ] \cdot \phi(x) dx\phi(y) dy, 
\end{align*} 
the statement of the lemma follows. 
\end{proof}

To prove the estimates for the solutions of \eqref{stationary solution trunc}, Lemma 2.10 of \cite{ferreira2019stationary} will be useful. 
We recall the statement here. 
\begin{lemma} \label{ferreira lemma 2.10}
Suppose $d>0$ and $b \in (0,1)$ and assume that $L \in (0,\infty] $ is such that $L \geq d.$
Consider some $\mu \in \mathcal M_+(\mathbb R_*)$ and $\varphi \in C(\mathbb R_*)$, with $\varphi \geq 0.$
\begin{enumerate}
\item  Suppose $L  < \infty$, and assume that there is $g \in L^1([d,L])$ such that $g \geq 0$ and
\begin{equation} \label{bound lemma ferreira g}
\frac{1}{z} \int_{[bz,z]} \varphi(x) \mu(dx) \leq g(z), \text{ for } z \in [d, L]. 
\end{equation} 
Then
\[
\int_{[d,L]} \varphi(x) \mu(dx) \leq \frac{\int_{[d,L]} g(z) dz}{|\ln b |} +Lg(L). 
\]
\item If $L=\infty$ and there is $g \in L^1([d,\infty) )$ such that $g \geq 0$ and
\[
\frac{1}{z}\int_{[bz,z]} \varphi(x) \mu(dx)  \leq g(z) \text{ for } z \geq d, 
\]
then
\[
\int_{[d, \infty)} \varphi(x) \mu(dx) \leq \frac{\int_{[d, \infty )} g(z) dz }{|\ln b|}. 
\] 
\end{enumerate} 
\end{lemma}

\begin{lemma} \label{lem:bounds eps_a_R}
The density $\phi$ of every solution of equation \eqref{stationary solution trunc} with respect to $K_a$, $\Xi_\varepsilon$, $\eta_\varepsilon$ and $\zeta_R$ satisfying \eqref{ineq:first mom trunc} and \eqref{support steady state}, is such that
\begin{equation} \label{decay bound best}
 \frac{1}{z} \int_{{8z}/{9}}^z \phi(x) dx \leq C  \left(\frac{a }{z^3}\right)^{1/2} \quad z \in[0, 2R]
\end{equation}
for some $C>0$ independent of $\varepsilon,a, R, $
and such that
\begin{equation} \label{decay bound a_eps_R}
\int_y^\infty \phi (x) dx \leq C_{a,\varepsilon} y^{-1/2}, \quad y \in[1,2R]
\end{equation} 
for a positive constant $C_{a,\varepsilon} >0$ independent of $R.$
\end{lemma}
\begin{proof}
Since $\phi$ satisfies \eqref{eq:truncated flux}, it follows that for almost every $z >0$
\begin{equation}\label{eq:upper_bound}
J_{\phi}(z) \ \leq 1 + \frac{2}{1-\gamma}  z^2 \phi(z).
\end{equation}
Noting that 
\begin{equation}\label{eq:Omega_R}
\left[{2z}/{3},z\right]^2 \subset \{(x,y) \in \R^2_* \ |\ 0<x \leq z,\ y> z-x \} =:\Omega_z
\end{equation}
as well as the lower bound for the kernel
\begin{eqnarray}
J_{\phi}(z) 
& \geq & \int_{\frac{2z}{3}}^z \int_{\frac{2z}{3}}^z   K_a(x,y) x \phi(x) \phi(y)dx dy \geq \frac{c z}{a} \left(\int_{\frac{2z}{3}}^z  \phi(x) dx \right)^2 \label{eq:lower_bound} \text{ for } z \leq 2R
\end{eqnarray}
for some constant $c>0$ independent of $a$, $R$ and $\varepsilon$. 

Combining \eqref{eq:upper_bound} and \eqref{eq:lower_bound} we conclude that
\begin{equation}\label{eq:first estimate_eps_a below}
 \int_{[2z/3,z]} \phi(x) dx \leq \left( \frac{a }{c}  \right)^{1/2} \left( \frac{ 1 +\frac{2}{1-\gamma} z^2 \phi(z) }{ z }\right)^{\frac{1}{2}},\quad a.e. \ z \leq 2R.
\end{equation} 
Since $\frac{2}{1-\gamma}  z^2 \phi (z) \geq 0$, integrating \eqref{eq:first estimate_eps_a below} over $[w,2w] $, with $w \in [0, R] $, we obtain that
\begin{align*}
\int_w^{2w} \int_{\frac{2z}{3}}^z \phi(x) dx dz \leq \left(  \frac{a }{c}  \right)^{1/2} \left( \int_w^{2w}  \left( \frac{ 1}{ z } \right)^{\frac{1}{2}} dz+ \int_w^{2w} \left(  \frac{2}{1-\gamma} z \phi(z)\right)^{\frac{1}{2}} dz \right).
\end{align*} 
By Cauchy-Schwartz inequality we conclude that
\begin{align*}
\int_w^{2w} \left(   z \phi(z)  \right)^{\frac{1}{2}} dz &\leq \left( \int_w^{2w } dz \right)^{1/2} \left( \int_w^{2w} z \phi (z) dz  \right)^{1/2} \leq   w^{1/2},
\end{align*} 
notice that, for the last inequality, we have used \eqref{ineq:first mom trunc}.

Combining all the above inequalities we conclude that 
\begin{align*}
\int_w^{2w} \int_{\frac{2z}{3}}^z \phi(x) dx dz \leq \left( \frac{a}{c} \left(1+ \frac{2 }{(1-\gamma)} \right) \right)^{1/2} w^{1/2}.
\end{align*} 
Moreover, by observing that 
\[
[{8w}/{9},w] \times [w,4w/3] \subset \{(x,z) \in \mathbb R^2_* : 2z/3<x<z, z \in [w,2w] \}
\]
we deduce that
\begin{align*}
\int_w^{2w } \int_{\frac{2z}{3}}^z  \phi(x) dx dz  \geq \int_w^{4w/3}  \int_{8w/9}^w \phi(x) dx  dz = {w}/{3} \int_{8w/9}^w \phi(x) dx .
\end{align*}
Consequently, adopting the notation $\tilde{C}= 3c^{-1/2}  \left(\frac{3-\gamma}{1-\gamma}\right)^{1/2}$, we conclude that for any $w \in [0, 2R]$
\begin{align*}
w \int_{8w/9}^w \phi(x) dx \leq \tilde{C} a^{1/2} w^{1/2}.
\end{align*}

Let us prove \eqref{decay bound a_eps_R}.
Thanks to inequality \eqref{decay bound best}, hypothesis \eqref{bound lemma ferreira g} of Lemma \ref{ferreira lemma 2.10} holds with $d=y,$ $b=8/9$, $L=2R$, $g(z)=C z^{-3/2} a^{1/2}$, and implies
\[
\int_y^{ 2R} \phi (x) dx \leq \left( C  \frac{a}{ R}\right)^{\frac{1}{2}}+ C  \frac{\int_y^{ 2 R} \left( \frac{ a }{ {z}^3}\right)^{\frac{1}{2}} dz}{\ln{\left(3/2 \right)}}. 
\]
Since 
\begin{align*}
\int_y^{ 2R}  \left( \frac{ a}{ {z}^3 }\right)^{\frac{1}{2}} dz  \leq \eu{2} \left( \frac{ a}{ y }\right)^{\frac{1}{2}}, \quad \left(   \frac{a}{ R}\right)^{\frac{1}{2}} \leq  2^{1/2}\left( \frac{a}{y}\right)^{1/2} \text{ and } \int_{2R}^\infty \Phi (dx) =0
\end{align*} 
and the result follows.  
\end{proof} 
\subsection{Proof of the existence of a self-similar solution} \label{sec:removal of the truncation}
The aim of this section is to prove that, as $R \rightarrow \infty$, $a \rightarrow \infty $ and $\varepsilon \rightarrow 0$, each solution, $\Phi_{\varepsilon, a, R}$, of \eqref{stationary solution trunc} with respect to $K_a$, $\zeta_R$, $\Xi_\varepsilon$ and $\eta_\varepsilon$, converges, in the \eu{weak$-*$} topology, to a measure $\Phi$ whose density is a self-similar profile as in Definition \ref{def:ss_sol}.

Lemma \ref{lem: Phi_eps_a} and Lemma \ref{lem: Phi_eps} describe the bounds and the properties that the limiting measures $\Phi_{\varepsilon,a}$ and $\Phi_{\varepsilon}$ respectively satisfy. In the proof of Theorem \ref{thm: existence of a ss sol} we will use these bounds and properties to prove the existence of a self-similar profile. 

\begin{lemma} \label{lem: Phi_eps_a}
Consider a sequence $\{ R_n\} \subset \mathbb R_*$ such that $\lim_{n \rightarrow \infty} R_n=\infty$.
Let $\Phi_{\varepsilon, a, R_n }$  be a solution of \eqref{stationary solution trunc} with respect to $K_a$, $\Xi_\varepsilon$, $\eta_\varepsilon$ and  $\zeta_{R_n}.$
There exists a measure $\Phi_{\varepsilon, a}\in \mathcal M_{+}(\mathbb R_*) $ such that 
\begin{equation}\label{eq:limit_Phi_eps_R}
\Phi_{\varepsilon, a, R_n} \rightharpoonup \Phi_{\varepsilon, a} \quad \text{ as } n \to \infty,  \quad \text{ in the \eu{weak$-*$} topology}.
\end{equation}
The measure $\Phi_{\varepsilon, a} $ 
is absolutely continuous with respect to the Lebesgue measure satisfies the equation
\begin{align} \label{ss with a, epsilon}
& \int_{\mathbb R_*}  \Phi_{\varepsilon, a}(dx) \varphi(x) = \int_{\mathbb R_*} \varphi (x) \eta_\varepsilon (dx) +   \frac{2}{1-\gamma} \int_{\mathbb R_*} \Xi_\varepsilon(x) \left( \varphi(x) -  x \varphi'(x)  \right)\Phi_{\varepsilon, a} (dx) \\
+ & \frac{1}{2} \int_{\mathbb R_*}  \int_{\mathbb R_*}  K_a(x,y) [\varphi(x+y)-\varphi(x) -\varphi(y) ] \Phi_{\varepsilon, a}(dx) \Phi_{\varepsilon, a}(dy).  \nonumber
\end{align}
for every $\varphi \in C_c^1(\mathbb R_*). $
Moreover, $\Phi_{\varepsilon,a}$ satisfies the growth bound
\begin{equation} \label{eq:estimate_eps_a}
 \frac{1}{z} \int_{[{8z}/{9}, z]} \Phi_{\varepsilon,a}( dx) \leq C  \left(\frac{1}{z^3\min\{a,z^\gamma \}}\right)^{1/2} \quad z>0,
\end{equation} 
for some positive $C.$
\end{lemma}
\begin{proof}
\eu{ We use the inequality \eqref{decay bound best}, proven in Lemma \ref{lem:bounds eps_a_R},} and apply Lemma \ref{ferreira lemma 2.10} with $d=\varepsilon,$ $b=8/9$, $L=2R_n$, $g(z)=C z^{-3/2} a^{1/2}$ to conclude that 
\[\int_{[\varepsilon, 2R_n]} \Phi_{\varepsilon,a,R_n} (dx) \leq 2 a \varepsilon^{-1/2}.  \]
From \eqref{support steady state} we deduce that the sequence $\{\Phi_{\varepsilon, a, R_n} \}$ is bounded in the \eu{weak$-*$} topology. 

By Banach-Alaoglu Theorem we deduce that the sequence $\{ \Phi_{\varepsilon,a,R_n} \}$ admits a  subsequence, $ \{\Phi_{\varepsilon,a,R_{n_k}} \}$, which converges in the \eu{weak$-*$} topology, namely, there exists a measure $\Phi_{\varepsilon,a}$ such that
\[
\Phi_{\varepsilon, a, R_{n_k}} \rightharpoonup \Phi_{\varepsilon, a} \quad \text{ as } k \to \infty \text{ in the \eu{weak$-*$} topology.}
\]

Since for every $n >0$
\[
\int_{\mathbb R_*}\Phi_{\varepsilon,a,R_n}(dx) \leq C_{\varepsilon,a}
\]
 we conclude, by passing to the limit as $n$ tends to infinity, that
\begin{equation}\label{bound on the 0 moment a_eps}
\int_{\mathbb R_*}\Phi_{\varepsilon,a}(dx) \leq C_{\varepsilon,a}.
\end{equation}

We would like to show that $\Phi_{\varepsilon,a}$ satisfies equation \eqref{ss with a, epsilon}. 
Since
$\Phi_{\varepsilon, a, R_n} \rightharpoonup \Phi_{\varepsilon,a}$ as $n \rightarrow \infty $ in the \eu{weak$-*$} topology, we immediately conclude that, for every $\varphi \in C_c^1(\mathbb R_*)$,
\[
\int_{\mathbb R_* } \varphi(x) \Phi_{\varepsilon, a, R_n} (dx)  \rightarrow \int_{\mathbb R_*} \varphi(x) \Phi_{\varepsilon, a} (dx) \text{ as } n \rightarrow \infty, 
\]
\[
\int_{\mathbb R_*} \Xi_\varepsilon(x) \varphi(x) \Phi_{\varepsilon, a, R_n} (dx)  \rightarrow \int_{\mathbb R_*} \Xi_\varepsilon(x) \varphi(x) \Phi_{\varepsilon, a} (dx) \text{ as } n \rightarrow \infty. 
\]
and 
\[
\int_{\mathbb R_*} \Xi_\varepsilon(x)  x\varphi'(x) \Phi_{\varepsilon, a, R_n} (dx)  \rightarrow \int_{\mathbb R_*} \Xi_\varepsilon(x)  x\varphi'(x) \Phi_{\varepsilon, a} (dx) \text{ as } n \rightarrow \infty. 
\]

It is not straightforward to conclude that
\begin{align*}
 & \int_{\mathbb R_*} \int_{\mathbb R_*}  K_a(x,y) [\zeta_{R_n}(x+y)\varphi(x+y)-\varphi(x) -\varphi(y) ] \Phi_{\varepsilon, a, R_n}(dx) \Phi_{\varepsilon,a,R_n}(dy) \rightarrow  \\
& \int_{\mathbb R_*} \int_{\mathbb R_*}  K_a(x,y) [\varphi(x+y)-\varphi(x) -\varphi(y) ] \Phi_{\varepsilon, a}(dx) \Phi_{\varepsilon, a}(dy) \text{ as } n \rightarrow \infty.
\end{align*} 
The main difficulty lies in the fact that the function 
\[
(x,y) \mapsto K_a(x,y) [\zeta_{R_n}(x+y)\varphi(x+y)-\varphi(x) -\varphi(y) ] 
\]
has not, in general, a compact support. Here the \eu{estimate} \eqref{decay bound a_eps_R} can be used. 
The details of the proof are shown in \cite{ferreira2019stationary} (proof of Theorem 2.3) and we omit them here. We conclude that $\Phi_{\varepsilon, a } $ is a solution of equation \eqref{ss with a, epsilon}.

An adaptation of the proof of Lemma \ref{lem:regularity} allows us to conclude that $\Phi_{\varepsilon,a} \ll \mathcal L.$
Indeed, if we consider the test function defined by \eqref{test function regularity}, then 
 for any $p, q \in \mathbb R_*$ such that $1/p + 1/q=1$, we deduce that for some $C(\varepsilon, a, \gamma) >0$ we have that
\begin{align*}
\frac{2}{1-\gamma} \left| \int_{\mathbb R_*}\chi(x) \Phi_{\varepsilon,a} (dx)\right|  \leq C(\varepsilon, a, \gamma) \left\| \frac{1}{z} \right\|_{L^p(\supp (\chi))} \|\chi \|_{L^q(\supp (\chi) )} .
\end{align*}
By the density of $C_c(\mathbb R_*)$  in $L^q(\mathbb R_*)$ we conclude that for any $q$ and any compact set $\mathcal K $ 
\begin{align*}
 \left| \int_{\mathbb R_*} \chi(x) \Phi_{\varepsilon,a} (dx)\right| \leq  C(\mathcal K, \gamma, \varepsilon,a) \|\chi \|_{L^q(\mathcal K)} \quad \forall \chi \in L^q(\mathcal K).
\end{align*} 
This implies that the measure $\Phi_{\varepsilon,a}$ is absolutely continuous with respect to the Lebesgue measure. 

Let us prove estimate \eqref{eq:estimate_eps_a}.
First of all, as in the proof of Lemma \ref{lem:flux a_eps_R}, we can conclude that, if $\Phi_{\varepsilon,a}$ satisfies \eqref{ss with a, epsilon}, then the density $\phi_{a,\varepsilon}$ satisfies 
\begin{align} \label{eq:truncated flux a eps}
\tilde{J}_{\phi_{\varepsilon,a}}(z) &= \int_0^z  x \eta_\varepsilon(dx) - \int_0^z  x\phi_{\varepsilon,a}(x) dx + \frac{2}{1-\gamma} \Xi_\varepsilon(z)  z^2 \phi_{\varepsilon,a} (z), 
\end{align}
for almost every $ z>0 $ and where 
\begin{equation}
\tilde{J}_{\phi_{\varepsilon,a}}(z) :=\int_0^z \int_{z-x}^\infty K_a(x,y) x \phi_{\varepsilon,a}(x) dx \phi_{\varepsilon,a}(y) dy. 
\end{equation} 

From \eqref{eq:truncated flux a eps}, it follows that for almost every $z >0$
\begin{equation}\label{eq:upper_bound_a_eps}
\tilde{J}_{\phi_{\varepsilon,a}}(z) \ \leq 1+ \frac{2}{1-\gamma}  z^2 \phi_{\varepsilon,a} (z). 
\end{equation}
Noting that $\left[{2z}/{3},z\right]^2 \subset \Omega_z$,
where $\Omega_z$ is defined by \eqref{eq:Omega_R},
as well as the condition on the kernel \eqref{kernel gamma lambda}, we write 
\begin{eqnarray}\label{eq:lower_bound_a_eps}
\tilde{J}_{\phi_{\varepsilon,a}}(z) \geq  c z \min \{ z^{\gamma}, a\} \left( \int_{\frac{2z}{3}}^z  \phi_{\varepsilon,a}(x) dx \right)^2 
\end{eqnarray}
for some constant $c>0$ independent of $a$ and $\varepsilon$. 

Combining \eqref{eq:upper_bound_a_eps} and \eqref{eq:lower_bound_a_eps} we conclude that
\begin{equation}\label{eq:first estimate_eps_a}
 \int_{\frac{2z}{3}}^z \phi_{\varepsilon,a}(x) dx \leq\left( \frac{1}{c}\right)^{1/2} \left( \frac{ 1 +\frac{2}{1-\gamma}  z^2 \phi_{\varepsilon,a} (z) }{ z \min \{ z^{\gamma}, a\}}\right)^{\frac{1}{2}},\quad a.e. \  z >0.
\end{equation} 
Since $ z^2 \phi_{\varepsilon,a} (z) \geq 0 $, by integrating \eqref{eq:first estimate_eps_a} over $[w,2w] $, with $w \geq 0 $, we obtain that there exists a constant $ \tilde{c}(\gamma) >0$ such that 
\begin{align*}
\int_w^{2w}  \int_{\frac{2z}{3}}^z  \phi_{\varepsilon,a}(x) dx dz \leq \tilde{c}(\gamma) \left(  \int_w^{2w}  \left( \frac{ 1}{ z \min \{ z^{\gamma}, a\} } \right)^{\frac{1}{2}} dz+ \int_w^{2w} \left(\frac{ z^2 \phi_{\varepsilon,a} (z) }{ z \min \{ z^{\gamma}, a\}}\right)^{\frac{1}{2}} dz\right) .
\end{align*} 
By Cauchy-Schwartz inequality we conclude that
\begin{align*}
\int_w^{2w}  \left( \frac{ 1}{ z \min \{ z^{\gamma}, a\} } \right)^{\frac{1}{2}} dz 
&\leq \left( \ln2\right)^{1/2} \cdot \left( \int_w^{2w} \frac{1}{ \min \{ z^{\gamma}, a\} }dz \right)^{1/2} 
\end{align*}  
and 
\begin{align*}
\int_w^{2w } \left(\frac{  z \phi_{\varepsilon,a} (z) }{ \min \{ z^{\gamma}, a\}}\right)^{\frac{1}{2}} dz 
\leq \left( \int_w^{2w} \frac{1}{ \min \{ z^{\gamma}, a\} }dz \right)^{1/2}. 
\end{align*}

Combining all the above inequalities we conclude that there exists a constant $c(\gamma)>0$
\begin{align*}
\int_w^{2w} \int_{\frac{2z}{3}}^z \phi_{\varepsilon,a }(x) dx dz \leq c(\gamma) \left( \int_w^{2w} \frac{1}{ \min \{ z^{\gamma}, a\} }dz \right)^{1/2}.
\end{align*} 
Moreover, by observing that $[\frac{8w}{9},w] \times [w,\frac{4w}{3}] \subset \{(x,z) \in \mathbb R^2_* : \frac{2z}{3}<x<z, z \in [w,2w] \}$
we deduce that
\begin{align*}
\int_w^{ 4w/3} \int_{8w/9}^w \phi_{\varepsilon,a}(x) dx dz  \geq \int_w^{ 4w/3 }  \int_{8w/9}^w \phi_{\varepsilon,a}(x) dx  dz = {w}/{3} \int_{8w/9}^w \phi_{\varepsilon,a}(x) dx .
\end{align*}
and, consequently, adopting the notation $\tilde{C}= 3c(\gamma)$ we conclude that for any $w >0$
\begin{align*}
w \int_{8w/9}^w \phi_{\varepsilon,a}(x) dx \leq \tilde{C}  \left( \int_w^{2w} \frac{1}{ \min \{ z^{\gamma}, a\} }dz \right)^{1/2}.
\end{align*}
Notice that $\tilde{C}$ is independent of $a, \varepsilon.$

If $a^{1/\gamma} \notin [w,2w]$, then 
\begin{align*}
w \int_{8w/9}^w \phi_{\varepsilon,a}(x) dx \leq 2^{1/2}  \tilde{C}\max\left\{1,\left( \frac{1}{1-\gamma}\right)^{1/2} \right\}\left(\frac{ w}{ \min \{ a, w^\gamma \} } \right)^{1/2}.
\end{align*}
Whereas if $a^{1 /\gamma } \in [w,2w]$, then  
\begin{align*}
& w \int_{8w/9}^w \phi_{\varepsilon,a}(x) dx \leq  \tilde{C} \left( \int_w^{a^{1/\gamma}} \frac{1}{  z^{\gamma} }dz + \int_{a^{1/\gamma}}^{2w} \frac{1}{a}dz\right)^{1/2}  \leq \tilde{C}\left(\frac{a^{1/\gamma-1}}{1-\gamma}  
 + \frac{2w}{a}\right)^{1/2}\\
&  \leq 2^{1/2} \tilde{C}  \left( \frac{w}{ a (1-\gamma) }+ \frac{ w }{a}\right)^{1/2}  \leq   2^{1/2} \tilde{C} \max\left\{1,\left( \frac{1}{1-\gamma}\right)^{1/2}\right\}  \left(\frac{ w}{ \min \{ a, w^\gamma \} } \right)^{1/2}. 
\end{align*}
The statement of the Lemma follows by selecting $C= 2^{1/2} \tilde{C} \max\left\{1,\left( \frac{1}{1-\gamma}\right)^{1/2} \right\} $. 
\end{proof} 
In the following definition we explain how we truncate the coagulation kernel $K$ to obtain a bounded coagulation kernel. 
Let us adopt the notation 
\begin{equation}\label{p}
p:=\max\left\{\lambda, - (\gamma +\lambda )\right\}. 
\end{equation} 
Each homogeneous coagulation kernel $K$ of parameters $\gamma, \lambda $ and with homogeneity $\gamma$ can be written as 
\begin{equation}\label{kernel alternative form}
K(x,y)= (x+y)^\gamma F\left( \frac{x}{x+y}\right)
\end{equation}
with $F:(0,1) \rightarrow \mathbb R_+$ being a smooth function \eu{ such that
\begin{equation}\label{kernel ineq}
F(s)=F(1-s) \quad \text{and} \quad \frac{ C_1 }{ s^p (1-s)^p}\leq F(s) \leq \frac{ C_2 }{ s^p (1-s)^p} 
\end{equation}
for any $s \in (0,1)$ and for some constants $C_1,C_2$ satisfying $0<C_1 \leq C_2 <\infty$.}
\begin{definition}\label{def:bounded coagulation kernel relative to K}
Assume $K$ to be a homogeneous coagulation kernel of parameters $\gamma, \lambda \in \mathbb  R$ and homogeneity $\gamma$.
We say that $K_a$ is the bounded coagulation kernel corresponding to $K $ and of bound $a>0$ if 
\[
K_a(x,y):=1/a + \min\left\{(x+y)^\gamma, a \right\} F_a \left( \frac{x}{x+y}\right) \quad x,y \in \mathbb R_*, 
\]
where $F_a$ is a smooth non negative function such that 
\[
 F_a \left(s \right) :=\begin{cases} F(s)  &\text{ if }  F(s)  \leq A a^\sigma \\
                                                     0 &\text{ if } F(s) \geq A a^\sigma, \end{cases} 
\] 
where $A>0$ is a constant independent of $a$, $\sigma =0$ if $p\leq 0$ while $\sigma >0$ if $p >0 $ and $\gamma \leq 0 $ and, finally, $0 < \sigma < \frac{p}{\gamma}$ if $p>0$ and $\gamma >0$.
\end{definition}  

This definition is taken from \cite{ferreira2019stationary} and even if it might seem odd, it allows to pass to the limit as $a $ goes to infinity in \eqref{ss with a, epsilon}. 
The main properties of this truncation of the kernel, that motivated us to introduce Definition \ref{def:bounded coagulation kernel relative to K}, are exposed in the following lemma.
\begin{lemma}\label{lem:ferreira limits}
Let $\{a_n\} \subset \mathbb R_*$ such that $a_n \rightarrow \infty$ as $n \rightarrow \infty$. 
For every $n$, let $K_{a_n}$ be a bounded coagulation kernel of bound $a_n$, corresponding to a homogeneous coagulation kernel $K$ of parameters $\gamma$ and $\lambda $, with homogeneity $\gamma < 1$ and $|\gamma + 2 \lambda |<1.$
Let $\mathcal C$ be a compact subset of $\mathbb R_*$ and $M>0.$
For each $x \in \mathcal C$ and $y> M$ it holds that 
\begin{enumerate}
\item if $\gamma, p \leq 0$, then 
\[
\min\left\{(x+y)^\gamma, { a_n}\right\} F_{a_n} \left( \frac{x}{x+y}\right) \leq  c_3 , \quad c_3>0;
\]
\item if $ p \leq 0$ and $\gamma >0$, then 
\begin{align*}
\min\left\{(x+y)^\gamma, { a_n}\right\} F_{a_n} \left( \frac{x}{x+y}\right) & \leq  c_4 \left( y^{-\lambda} + y^{\gamma +\lambda} \right) \chi_{\{y \leq {a_n}^{1/\gamma} \}}(y) \\
&+ c_4 a_n \left( y^{\lambda} + y^{-\gamma -\lambda} \right) \chi_{\{y > {a_n}^{1/\gamma} \}}(y),
\end{align*}
where $c_4>0$; 
\item if $p >0$ then 
\[
\min\left\{(x+y)^\gamma, { a_n}\right\} F_{a_n} \left( \frac{x}{x+y}\right)  \leq  c_5 \left( y^{-\lambda} + y^{\gamma +\lambda} \right) \chi_{\{y \leq c_* {a_n}^{-\sigma/p} \}}(y)
\]
where $c_5, c_*>0$.
\end{enumerate}
Let $\Phi_{\varepsilon, a_n} $ be a solution of \eqref{ss with a, epsilon}, with respect to $K_{a_n}$ and $\eta_\varepsilon$.  Then, for every $\varphi \in C_c(\mathbb R_*)$ we have that 
\begin{equation}\label{eq:horrible lemma}
 \int_{\mathbb R_*} \int_{(M,\infty)} \min\left\{(x+y)^\gamma, { a_n}\right\} F_{a_n} \left( \frac{x}{x+y}\right) \varphi(y) \Phi_{\varepsilon,{a_n}}(dx) \Phi_{\varepsilon,{a_n}} (dy) \rightarrow 0,
\end{equation}
as $M\rightarrow \infty $.
\end{lemma}
For the proof of Lemma \ref{lem:ferreira limits} we refer to \cite{ferreira2019stationary}, proof of Theorem 2.3. The main idea is that Definition \ref{def:bounded coagulation kernel relative to K} allows to prove the inequalities presented in the Lemma, which imply \eqref{eq:horrible lemma}. 
\begin{lemma} \label{lem: Phi_eps}
Assume $K$ to be a homogeneous symmetric coagulation kernel $K \in C(\mathbb R_* \times \mathbb R_*)$ satisfying \eqref{kernel gamma lambda} with homogeneity $\gamma <1 $ and $|\gamma + 2 \lambda | < 1. $
Consider a sequence $\{ a_n\} \subset \mathbb R_*$ such that $\lim_{n \rightarrow \infty} a_n=\infty$ and the sequence of bounded coagulation kernels $\{K_{a_n} \}$ corresponding to $K.$
Let $\Phi_{\varepsilon, a_n}$  be a solution of \eqref{ss with a, epsilon} with respect to $\eta_\varepsilon$, $\Xi_\varepsilon$ and with respect to $K_{a_n}.$
There exists a measure $\Phi_{\varepsilon} $ such that 
\begin{equation}\label{eq:limit_Phi_eps}
\Phi_{\varepsilon, a_n} \rightharpoonup \Phi_{\varepsilon} \quad \text{ as } n \to \infty,  \quad \text{ in the \eu{weak$-*$} topology}.
\end{equation}
The measure $\Phi_{\varepsilon} $ 
is absolutely continuous with respect to the Lebesgue measure, with density $\phi_{\varepsilon} $. It satisfies the bounds
\begin{equation}\label{moment bounds Phi_eps}
\int_{\mathbb R_*}  x^{\gamma +\lambda }\Phi_\varepsilon(dx) < \infty , \quad \int_{\mathbb R_*}  x^{-\lambda }\Phi_\varepsilon(dx) < \infty
\end{equation}
and 
\begin{equation} \label{eq:estimate_epsilon}
\frac{1}{z} \int_{[8z/9,z]} \Phi_{\varepsilon}(dx) \leq \left( \frac{ C }{ z^{3+\gamma} }\right)^{\frac{1}{2}},\quad z  >0.
\end{equation}
Moreover, it solves for every $\varphi \in C_c(\mathbb R_*)$ the equation
\begin{align} \label{ss with epsilon}
& \int_0^\infty \Phi_{\varepsilon}(dx) \varphi(x) = \int_0^\infty \varphi (x) \eta_\varepsilon (dx) + \frac{2}{1-\gamma} \int_0^\infty \Xi_\varepsilon(x) \left( \varphi(x)- x \varphi'(x) \right) \Phi_{\varepsilon} (dx) dx   \\
+ & \frac{1}{2} \int_0^\infty \int_0^\infty K (x,y) [\varphi(x+y)-\varphi(x) -\varphi(y) ] \Phi_{\varepsilon}(dx) \Phi_{\varepsilon}(dy).  \nonumber
\end{align}
\end{lemma}

\begin{proof}
We know that $\Phi_{\varepsilon, a} ((0, \varepsilon])=0$. Therefore inequality \eqref{eq:estimate_eps_a} is non-trivial when $z>\varepsilon.$
Let us consider the case $\gamma \leq 0$. In this case we have that $z^{\gamma} \leq \varepsilon^{\gamma}$. Since $\varepsilon$ is fixed and $a_n \rightarrow \infty $ as $n \rightarrow \infty $, there exists an $\overline n $ such that $z^\gamma < a_{n}$ for very $n \geq \overline n.$

Consequently, for every $n \geq \overline n$, we conclude that
\[
\frac{1}{z} \int_{[8z/9,z] } \Phi_{\varepsilon,a_n}(dx) \leq\frac{ C^{1/2}}{z^{(\gamma+3)/2}}.
\] 
Applying Lemma \ref{ferreira lemma 2.10} to the rescaled measure $x^{\gamma+\lambda}\Phi_{\varepsilon, a_n} (dx)$ we conclude that, if $n \geq \overline n$, then
\[
\int_{[\varepsilon, \infty)} x^{\gamma+\lambda} \Phi_{\varepsilon,a_n}(dx) \leq C^{1/2} \int_{[\varepsilon, \infty)}z^{-(\gamma+3)/2} z^{\gamma +\lambda} dz < C_\varepsilon
\] 
where $C_\varepsilon$ is a constant which depends only on $\varepsilon$ and where the last inequality comes from the bound $|\gamma+2\lambda |< 1$. 

 Since $\Phi_{\varepsilon, a}((0,\varepsilon])=0$, the sequence of rescaled measures $\{ x^{\gamma+\lambda}\Phi_{\varepsilon, a_n}(dx)\}_{n \geq \overline n}$ belongs to a compact set, and we conclude by Banach-Alaoglu Theorem that there exists a subsequence of $\{ x^{\gamma+\lambda} \Phi_{\varepsilon, a_n}(dx) \}$ which converges, in the \eu{weak$-*$} topology, to a measure $\mu.$
This implies that, if we denote by $\Phi_{\varepsilon} $ the measure defined by $\Phi_{\varepsilon}(dx):= x^{-(\gamma+\lambda)} \mu(dx) $, we have, up to a subsequence, that
\[
\Phi_{\varepsilon, a_n} \rightharpoonup \Phi_\varepsilon
\] 
as $n \rightarrow \infty $, in the \eu{weak$-*$} topology. 

Let us consider now the case $0 < \gamma <1.$ If $z^\gamma>a$ it holds that 
\[
\frac{1}{z} \int_{[8z/9,z] } \Phi_{\varepsilon,a}(dx) \leq\frac{ C^{1/2}}{ a z^{3/2}}.
\] 
If, instead, $z^{\gamma} \leq a $, then
\[
\frac{1}{z} \int_{[8z/9,z] } \Phi_{\varepsilon,a}(dx) \leq\frac{ C^{1/2}}{ z^{(3+\gamma)/2}}. 
\] 
By applying Lemma \ref{ferreira lemma 2.10} to the scaled measure $ \Phi_{\varepsilon, a_n}$ with $a_n \geq 1$, we conclude that 
\[
\int_{[\varepsilon,\infty)} \Phi_{\varepsilon,a_n}(dx)  \leq C^{1/2} \left( \varepsilon^{-1/2}+ \varepsilon^{(-1+\gamma)/2} \right). 
\]
Therefore, the sequence $\{\Phi_{\varepsilon, a_n}\}$ belongs to a compact set, and $\Phi_{\varepsilon, a_n} \rightharpoonup \Phi_\varepsilon$
as $n \rightarrow \infty $ in the \eu{weak$-*$} topology.

Passing to the limit as $a \rightarrow \infty$ in inequality \eqref{eq:estimate_eps_a} we obtain that
\[
\frac{1}{z} \int_{[8z/9,z]} \Phi_{\varepsilon}(dx) \leq \left( \frac{ C }{ z^{3+\gamma} }\right)^{\frac{1}{2}},\quad z  >0.
\]

By applying Lemma \ref{ferreira lemma 2.10} with $g(z)=z^{-\frac{3+\gamma}{2}}$ and using the assumption $| \gamma+ 2 \lambda| < 1 $, we deduce that
\[
\int_{\mathbb R_*} x^{-\lambda } \Phi_\varepsilon(dx) < \infty \text{ and } \int_{\mathbb R_*} x^{\lambda + \gamma } \Phi_\varepsilon(dx) < \infty. 
\] 

Let us pass to the limit as $n $ tends to infinity in all the terms of equation \eqref{ss with a, epsilon}. 

From the \eu{weak$-*$} convergence of $\Phi_{\varepsilon,a_n } $ to $\Phi_\varepsilon$, we conclude that, for every $\varphi \in C^1_c(\mathbb R_*),$
\[
\int_{\mathbb R_*}\varphi(x)\Phi_{\varepsilon, a_n}(dx) \rightarrow \int_{\mathbb R_*} \varphi(x) \Phi_{\varepsilon}(dx) \quad n \rightarrow \infty , 
\] 
\[
 \int_{\mathbb R_*} \Xi_\varepsilon(x) x \varphi'(x)\Phi_{\varepsilon, a_n}(dx) \rightarrow \int_{\mathbb R_*}  \Xi_\varepsilon(x) x\varphi'(x) \Phi_{\varepsilon}(dx) \quad n \rightarrow \infty
\]
and 
\[
 \int_{\mathbb R_*} \Xi_\varepsilon(x) \varphi(x)\Phi_{\varepsilon, a_n}(dx) \rightarrow \int_{\mathbb R_*}  \Xi_\varepsilon(x) \varphi(x) \Phi_{\varepsilon}(dx) \quad n \rightarrow \infty. 
\]

For the proof of the fact that, for every $\varphi \in C_c(\mathbb R_*)$, 
\[
\int_{\mathbb R_*} \int_{\mathbb R_*}  K_{a_n}(x,y)[\varphi(x+y)-\varphi(y)-\varphi(x)] \Phi_{\varepsilon,a_n}(dx) \Phi_{\varepsilon, a_n} (dy) 
\]
converges, as $n$ tends to infinity, to
\[
\int_{\mathbb R_*} \int_{\mathbb R_*}  K(x,y)[\varphi(x+y)-\varphi(y)-\varphi(x)] \Phi_{\varepsilon}(dx) \Phi_{\varepsilon} (dy)
\]
we refer to \cite{ferreira2019stationary}, proof of Theorem 2.3. 
Nevertheless we find instructive to provide here the main steps of the proof. 
First of all, since $K_a$ and $K $ are continuous functions we have that, for any compact subset of ${\mathbb R_*}^2 $, $(xy)^q K_a(x,y) \rightarrow (xy)^q K(x,y) $, uniformly in $x,y$ as $a \rightarrow \infty$ for any $q \in \mathbb R.$
Consequently, for any $\varphi \in C_c(\mathbb R_*)$
\[
 \int_{\mathbb R_*} \int_{\mathbb R_*}  K_{a_n}(x,y)\varphi(x+y) \Phi_{\varepsilon,{a_n}}(dx) \Phi_{\varepsilon,{a_n}} (dy) \rightarrow  \int_{\mathbb R_*} \int_{\mathbb R_*}  K(x,y)\varphi(x+y) \Phi_{\varepsilon}(dx) \Phi_{\varepsilon} (dy), 
\]
as $n\rightarrow \infty. $
For the same reason we know that 
\[
 \int_{\mathbb R_*} \int_{(0, M] }  K_{a_n}(x,y)\varphi(x) \Phi_{\varepsilon,{a_n}}(dy) \Phi_{\varepsilon,{a_n}} (dx) \rightarrow  \int_{\mathbb R_*} \int_{(0, M] }  K(x,y)\varphi(x) \Phi_{\varepsilon}(dy) \Phi_{\varepsilon} (dx)
\]
and 
\[
 \int_{\mathbb R_*} \int_{(0, M] }  K_{a_n}(x,y)\varphi(y) \Phi_{\varepsilon,{a_n}}(dx) \Phi_{\varepsilon,{a_n}} (dy) \rightarrow  \int_{\mathbb R_*} \int_{(0, M] }  K(x,y)\varphi(y) \Phi_{\varepsilon}(dx) \Phi_{\varepsilon} (dy).
\]

To conclude we just need now to show that, as $M\rightarrow \infty $
\[
 \int_{\mathbb R_*} \int_{(M,\infty)}  K_{a_n}(x,y)\varphi(y) \Phi_{\varepsilon,{a_n}}(dx) \Phi_{\varepsilon,{a_n}} (dy) \rightarrow 0,
\]
and 
\[
   \int_{\mathbb R_*} \int_{(M, \infty) }  K_{a_n}(x,y)\varphi(x) \Phi_{\varepsilon,{a_n}}(dy) \Phi_{\varepsilon,{a_n}} (dx) \rightarrow 0. 
\] 
Notice that, by the definition of $K_{a_n}$, 
\begin{align}\label{to pass to the limit a}
 & \int_{\mathbb R_*} \int_{(M,\infty)}  K_{a_n}(x,y)\varphi(x) \Phi_{\varepsilon,{a_n}}(dy) \Phi_{\varepsilon,{a_n}} (dx) \\
&\leq \frac{1}{a_n}  \int_{\mathbb R_*} \int_{(M,\infty)}\varphi(x) \Phi_{\varepsilon,{a_n}}(dy) \Phi_{\varepsilon,{a_n}} (dx) \nonumber \\ 
& +  \int_{\mathbb R_*} \int_{(M,\infty)} \min\{(x+y)^\gamma,a\} F_{a_n}\left(\frac{x}{x+y}\right)\varphi(x) \Phi_{\varepsilon,{a_n}}(dy) \Phi_{\varepsilon,{a_n}} (dx).  \nonumber
\end{align} 
Thanks to \eqref{decay bound a_eps_R}, we have that 
\[
\int_{\mathbb R_*} \int_{(M,\infty)}\varphi(x) \Phi_{\varepsilon,{a_n}}(dy) \Phi_{\varepsilon,{a_n}} (dx ) \leq c M^{-1/2} 
\]
where $c$ is just a positive constant. 

By Lemma \ref{lem:ferreira limits} we know that 
\[
\int_{\mathbb R_*} \int_{(M,\infty)} \min\{(x+y)^\gamma,a\} F_{a_n}\left(\frac{x}{x+y}\right)\varphi(x) \Phi_{\varepsilon,{a_n}}(dy) \Phi_{\varepsilon,{a_n}} (dx) \rightarrow 0
\]
as $M \rightarrow \infty.$
The same holds also for the term  \[
\int_{\mathbb R_*} \int_{(M,\infty)} \min\{(x+y)^\gamma,a\} F_{a_n}\left(\frac{x}{x+y}\right)\varphi(y) \Phi_{\varepsilon,{a_n}}(dx) \Phi_{\varepsilon,{a_n}} (dy).
\]

 We conclude that $\Phi_\varepsilon$ satisfies the following equation for every $\varphi \in C^1(\mathbb R_*)$
\begin{align*} 
&\int_0^\infty \Phi_{\varepsilon}(dx) \varphi(x)  = \int_0^\infty \varphi (x) \eta_\varepsilon (dx) + \frac{2}{1-\gamma} \int_0^\infty  \Xi_\varepsilon(x) \left( \varphi(x) - x \varphi'(x) \right) \Phi_{\varepsilon} (dx)  \\
+ & \frac{1}{2} \int_0^\infty \int_0^\infty K(x,y) [\varphi(x+y)-\varphi(x) -\varphi(y) ] \Phi_{\varepsilon}(dx) \Phi_{\varepsilon}(dy).
\end{align*}

As in the proof of Lemma \ref{lem:regularity}, we can choose the test function $\varphi$ given by expression \eqref{test function regularity} to conclude that, for any $p, q \in \mathbb R_*$ such that $1/p + 1/q=1$, we have that
\begin{align*}
&  \frac{2}{1-\gamma} \left| \int_{\mathbb R_*}\chi(x) \Xi_\varepsilon (x) \Phi_\varepsilon(dx)\right|  \leq
  \frac{3}{2}\int_{\mathbb R_*} x^{\gamma +\lambda} \Phi_{\varepsilon}(dx)\int_{\mathbb R_*} x^{-\lambda} \Phi_{\varepsilon}(dx)\left\|  \frac{1}{z} \right\|_{L^p(U)} \|\chi \|_{L^q(U)} \\
&+ \eta_\varepsilon ({\mathbb R_*}) \left\| \frac{1}{z} \right\|_{L^p(U)} \|\chi \|_{L^q(U)}  + \frac{\left| 1+\gamma\right| }{1-\gamma}\Phi_\varepsilon({\mathbb R_*}) \left\|  \frac{1}{z} \right\|_{L^p(U)} \|\chi \|_{L^q(U)}  
\end{align*}
where we are using the notation $\supp(\chi):=U.$
By the density of $C_c(\mathbb R_*)$  in $L^q(\mathbb R_*)$ we conclude that for any $q$ and any compact set $\mathcal K $ 
\begin{align*}
 \left| \int_{\mathbb R_*} \Xi_\varepsilon (x)  \chi(x) \Phi_\varepsilon (dx)\right| \leq  C(\mathcal K, \gamma, \varepsilon, \Phi_\varepsilon) \|\chi \|_{L^q(\mathcal K)} \quad \forall \chi \in L^q(\mathcal K).
\end{align*} 
Therefore, $\Phi_\varepsilon \ll \mathcal L$. 
\end{proof}
We now introduce some notation, which will be employed in the following and a Lemma, taken from \cite{ferreira2019stationary}, which will be important for the proof of Theorem \ref{thm: existence of a ss sol}.

For a given $\delta>0$, we  consider the partition  $\R_+^2 = \Sigma_1(\delta) \cup \Sigma_2(\delta) \cup \Sigma_3(\delta)$ with
\begin{eqnarray*}
\Sigma_1(\delta) =  \{(x,y)\ |\  y >  x/\delta  \},\quad \Sigma_2(\delta) =  \{(x,y)\ |\ \delta x\leq y \leq   x/\delta  \}, \quad \Sigma_3(\delta) = \{(x,y)\ |\ \ y <  \delta x \} 
\end{eqnarray*}
and, if $\mu \in \mathcal M_+(\mathbb R_*)$ is such that for every $z>0$
\[
 \int_{\Omega_z }K\left( x,y\right) x\mu \left( dx\right) \mu(dy) < \infty, 
\]
\eu{where $\Omega_z$ is defined by \eqref{eq:Omega_R},} then we define 
\begin{equation*} \label{eq:fluxeq_j}
J_j(z,\delta,\mu) := 
\int_{\Omega_z \cap \Sigma_j(\delta)}K\left( x,y\right) x\mu \left( dx\right)
\mu \left( dy\right)  \ \text{for }z >0
\end{equation*}
for $j=1,2,3$.

Notice that 
\[
\sum_{j=1}^3 J_j(z,\delta,\mu) = \int_{\Omega_z }K\left( x,y\right) x\mu \left( dx\right)
\mu \left( dy\right).
\]
\begin{lemma}\label{ferreira lemma 6.1}
Let $K $ be a homogeneous symmetric coagulation kernel $K \in C(\mathbb R_* \times \mathbb R_*)$ satisfying \eqref{kernel gamma lambda} with $|\gamma + 2 \lambda | < 1.$ Suppose that the measure $\mu \in \mathcal M_+(\mathbb R_*)$ satisfies 
\begin{equation} \label{hp:ferreira lemma 6.1}
\frac{1}{z} \int_{[z/2,z]} \mu(dx) \leq \frac{A}{z^{(\gamma+3)/2}} \text{ for all } z >0.
\end{equation}
Then, for every $\varepsilon>0$, there exists a $\delta_\varepsilon>0$ depending on $\varepsilon$, as well as on $\gamma$ and $\lambda $ and on the constants $c_1 $ and $c_2$ of inequality \eqref{kernel gamma lambda}, but independent of $A$, such that for any $\delta \leq \delta_\varepsilon$, we have that
\begin{equation} \label{bound J_1}
\sup_{z >0} J_1(z, \delta,\mu) \leq \varepsilon A^2 
\end{equation}
\begin{equation} \label{bound J_2}
\sup_{R>0}\frac{1}{R} \int_{[R,2R]} J_3(z,\delta,\mu) dz \leq \varepsilon A^2.
\end{equation}
\end{lemma}

\begin{proposition} \label{prop:Phi}
Assume $K$, $\eta$, $\lambda$ and $\gamma$ to be as in the assumptions of Theorem \ref{thm: existence of a ss sol}. 
There exists a $\Phi \in \mathcal M_{+}(\mathbb R_*)$ 
with 
\[
J_\Phi \in L^\infty_{loc}(\mathbb R_*)
\]
solving
\begin{equation} \label{eq:ss_eq not AC} 
\int_{\mathbb R_*} \varphi (z) \left( J_\Phi(z) dz - dz  + \int_{(0,z] } x \Phi(dx) dz - \frac{2}{1-\gamma} z^2 \Phi(dz) \right)= 0
\end{equation} 
for every test function $\varphi \in C_c(\mathbb R_*)$ and satisfying  \eqref{eq:estimate_Phi} and the inequalities 
\begin{equation}\label{first mom bound}
\int_{\mathbb R_*}  x \Phi(dx) \leq 1, 
\end{equation} 
\begin{equation} \label{bound p q}
 \int_{(1, \infty) }  x^{p }  \Phi(dx)  <\infty , \quad \int_{(0,1]} x^{q}  \Phi(dx)  < \infty 
\end{equation} 
with $q=\min\{ \gamma +\lambda +1, 1-\lambda \}$ and $p=\max \{\gamma+\lambda, -\lambda\} $. 
\end{proposition}
\begin{proof}
Consider the sequences $\{\varepsilon_n \}$, $\{a_m\}$ and $\{ R_k\} $, with $\lim_{n \rightarrow \infty} \varepsilon_n=0, $ $\lim_{m \rightarrow \infty} a_m=\infty$ and $\lim_{k \rightarrow \infty} R_n=\infty.$
Consider a sequence of measures $\{ \Phi_{\varepsilon_n,a_m,R_k} \} $ that solve \eqref{stationary solution trunc} with respect to $\zeta_{R_n}$, $K_{a_n}$ and $\eta_{\varepsilon_n}$. 
By Lemma \ref{lem: Phi_eps_a} and Lemma \ref{lem: Phi_eps} we know that $\Phi_{\varepsilon_n,a_m, R_k}$ converges in the \eu{weak$-*$} topology to the absolutely continuous measure $\Phi_{\varepsilon_n}$ as $m$ and $k$ go to infinity,  that solves \eqref{ss with epsilon} and satisfies the bound \eqref{eq:estimate_epsilon}.
We would like to prove that there exists a measure $\Phi$ such that $\Phi_{ \varepsilon_n} \rightharpoonup  \Phi \text{ as } n \rightarrow \infty$
in the \eu{weak$-*$} topology, and that $\Phi$ solves \eqref{eq:ss_eq not AC}, \eqref{bound p q} and \eqref{first mom bound}.  

Consequently we use the following diagonal argument, which is similar to the one used in Section 7 of \cite{ferreira2019stationary}. 
We notice that if $I_k:=[2^{-k }, 2^k]$, then $\mathbb R_*= \bigcup_{k=1}^\infty I_k.$
The restricted sequence $\left\{ {\Phi_{\varepsilon_n}}_{| I_k} \right\}$ on $I_k$ has a convergent subsequence, $\left\{ {\Phi_{\varepsilon_{n_l}}}_{| I_k} \right\}$.
Since if $k>m $, then $I_m \subset I_k$, we can use a diagonal argument to conclude that, up to a subsequence, there exists a measure $\Phi\in \mathcal M_{+}(\mathbb R_*)$ such that $\Phi_{\varepsilon_n} \rightharpoonup \Phi$ as $ n \rightarrow \infty$ in the \eu{weak$-*$} topology. 

Since $\Phi$ is the \eu{weak$-*$} limit of $\Phi_{\varepsilon_n}$ as $n \rightarrow \infty$ thanks to \eqref{ineq:first mom trunc}, we know that
\begin{equation*}
\int_{\mathbb R_*} x \Phi_{\varepsilon_n} (dx)  \leq  1.
\end{equation*}
Passing to the limit for $n \rightarrow \infty$ we deduce \eqref{first mom bound}. Similarly, passing to the limit in \eqref{moment bounds Phi_eps} we deduce \eqref{eq:estimate_Phi}.

Thanks to inequality \eqref{eq:estimate_Phi}, Lemma \ref{ferreira lemma 2.10} and the assumption $|\gamma+2\lambda |<1$ we can prove \eqref{bound p q}. 

We aim now at showing that $J_\Phi \in L^\infty_{loc}(\mathbb R_*)$. 
Notice that
\begin{align*}
J_\Phi(z) = \int_{(z, \infty) } \int_{(0,z]} x K(x,y)  \Phi(dx) \Phi(dy) + \int_{(0,z]}\int_{(z-y,z] } x K(x,y)  \Phi(dx) \Phi(dy) 
\end{align*} 
and using \eqref{bound p q} and \eqref{eq:estimate_Phi} we deduce that for every $z>0$
\begin{align*}
& \int_{(z, \infty) } \int_{(0,z]} x K(x,y)  \Phi(dx) \Phi(dy) \leq c_2  \int_{(z, \infty) }  y^{-\lambda}  \Phi(dy)  \int_{(0,z]} x^{1+\gamma+\lambda}  \Phi(dx)  \\
&+ c_2 \int_{(z, \infty) }  y^{\gamma+\lambda} \Phi(dy)  \int_{(0,z] }  x^{1-\lambda} \Phi(dx) < c_5 < \infty,
\end{align*}
where $c_5$ is a positive constant depending on $\gamma $ and $\lambda$ and independent of $z.$

On the other hand 
\begin{align*}
& \int_{(0,z]}\int_{(z-y,z] } x K(x,y)  \Phi(dx) \Phi(dy)\leq c_2 \int_{(0,z]}\int_{(z-y,z] } x^{1-\lambda}  \Phi(dx)  y^{\gamma+\lambda }\Phi(dy) \\
&+c_2 \int_{(0,z]}\int_{(z-y,z] } x^{1+\lambda+\gamma} \Phi(dx)  y^{-\lambda}  \Phi(dy).
\end{align*}
Again, from \eqref{eq:estimate_Phi} we conclude, applying Lemma \ref{ferreira lemma 2.10}, that
\[
\int_{(z-y ,z] } x^{1-\lambda}  \Phi(dx)  \leq c \int_{z-y}^z  x^{1-\lambda- (3+\gamma)/2} dx = z^{(1-2\lambda -\gamma)/2} - (z-y)^{(1-2\lambda -\gamma)/2}. 
\]
Notice that for every $p>0$, if $x \geq y $, then 
\begin{equation}\label{bound p x y}
(x+y)^p - x^p \leq C(p) x^{p-1} y
\end{equation}
while if $x \leq y $, 
\begin{equation}\label{bound p y x}
(x+y)^p - x^p \leq c(p) y^{p}
\end{equation}
where $c(p) $ and $C(p) $. 
We conclude that if $y \leq \frac{z}{2}, $ we have
\[
\int_{(z-y ,z] } x^{1-\lambda}  \Phi(dx)  \leq z^{(-1-2\lambda-\gamma)/2} y. 
\]
If instead $y > \frac{z}{2}, $ we obtain
\[
\int_{(z-y ,z] } x^{1-\lambda}  \Phi(dx)  \leq y^{(1-2\lambda-\gamma)/2}. 
\]
Consequently, for every $z>0$ we have that 
\[
 \int_{(0,z/2]}\int_{(z-y,z] } x^{1-\lambda}  \Phi(dx)  y^{\gamma+\lambda }\Phi(dy)  \leq  z^{(-1-2\lambda-\gamma)/2} \int_{(0,z]} y^{1+\gamma+\lambda }\Phi(dy) <c_6 < \infty, 
\]
where $c_6>0$ is independent of $z$, and 
\[
 \int_{(z/2,z]}\int_{(z-y,z] } x^{1-\lambda}  \Phi(dx)  y^{\gamma+\lambda }\Phi(dy) <  \int_{(z/2,z]} y^{(\gamma+1)/2 }\Phi(dy)  <c z^{\frac{\gamma+3}{2}}< \infty, 
\]
with $c(z)>0$ for every $z.$

We conclude that, if $\mathcal K$ is a compact subset of $\mathbb R_*$, then there exists a constant $c>0$ such that 
\[
\sup_{z \in \mathcal K } \int_{(0,z]}\int_{(z-y,z] } x^{1-\lambda}  \Phi(dx)  y^{\gamma+\lambda }\Phi(dy)< c   < \infty. 
\]
Similarly, it is possible to prove that there exists another constant $c>0$ such that 
\[
\int_{(0,z]}\int_{(z-y,z] } x^{1+\lambda+\gamma} \Phi(dx)  y^{-\lambda}  \Phi(dy)<c <  \infty
\] 
and we conclude that
$J_\Phi \in L^\infty_{loc}(\mathbb R_*) $.

We now prove that $\Phi$ solves \eqref{eq:ss_eq not AC} using a similar argument as in Section 7 in \cite{ferreira2019stationary}.
First of all we prove that the measure $\Phi_{\varepsilon_n}$ satisfies, for every test function $\varphi \in C_c(\mathbb R_*)$, the equality
\begin{align}\label{eq:weak_flux}
&\int_{\mathbb R_*} \varphi(z) J_{\Phi_{\varepsilon_n}}(z) dz=\int_{\mathbb R_*} \varphi(z)\int_{(0,z]} x \eta_{\varepsilon_n}(dx) dz-  \int_{\mathbb R_*} \varphi(z) \int_{(0,z]} x\Phi_{\varepsilon_n}(dx) dz\\
&
+ \frac{2}{1-\gamma} \int_{\mathbb R_*}  \Xi_\varepsilon(z) z^2 \varphi(z)\Phi_{\varepsilon_n} (dz). \nonumber
\end{align}
Since the measure $\Phi_\varepsilon$ satisfies \eqref{ss with epsilon}, we conclude, 
as in Lemma \ref{lem:flux a_eps_R}, that the density $\phi_{\varepsilon_n}$ of $\Phi_{\varepsilon_n}$ satisfies
\begin{align*}
 J_{\phi_{\varepsilon_n}}(z)&=\int_0^z  x \eta_{\varepsilon_n}(x) dx -  \int_0^z x\phi_{\varepsilon_n}(x) dx+ \frac{2}{1-\gamma} \Xi_\varepsilon(z) z^2 \phi_{\varepsilon_n} (z) \quad a.e.\  z>0. 
\end{align*}
 Integrating against a test function $\varphi \in C_c(\mathbb R_*)$ we conclude that equation \eqref{eq:weak_flux} holds. 
We aim now at passing to the limit as $n$ goes to infinity in equation \eqref{eq:weak_flux}. 

We plan to show that estimate \eqref{eq:estimate_epsilon} implies that $\Phi$ satisfies the hypothesis of Lemma \ref{ferreira lemma 6.1}. 
To this end, we only need to ensure that 
\begin{equation} \label{inequality for lemma 6.1}
\frac{1}{z} \int_{z/2}^z \phi_{\varepsilon_n}(x) dx \leq \left( \frac{ C}{ z^{3+\gamma} }\right)^{\frac{1}{2}},\quad z  >0,
\end{equation}
for some constant $C>0$. 

Notice that there exists $m>0$ such that \eu{$[z/2, z] \subset \bigcup_{i=1}^{m} [(8/9)^i z, (8/9)^{i-1} z ]$. }
Since for every $i$
\eu{
\[
\frac{1}{z} \int_{(8/9)^i z}^{ (8/9)^{i-1} z} \phi_{\varepsilon_n}(x) dx= \frac{1}{ z}\left(\frac{8}{9}\right)^{i-1} \int_{ z}^{ 8z/9 } \phi_{\varepsilon_n}(y) dy \leq \left( \frac{ C}{ z^{3+\gamma} }\right)^{\frac{1}{2}},\ z  >0,
\] }
then \eqref{inequality for lemma 6.1} holds and we can apply Lemma \ref{ferreira lemma 6.1} to conclude that for any $\bar \varepsilon>0$ there is a $\delta_0>0$ depending on $\bar \varepsilon$ and $\gamma$, such that  for any $\delta \leq \delta_0$ and any $\varphi \in  C_c(\R_*)$
\begin{equation}\label{eq:estimate_J13_eps}
\left|\sum\limits_{j \in \{1,3\}} \int_{(0, \infty)} J_j(z,\delta, \Phi_{\varepsilon_n}) \varphi(z) dz\right|\leq C\bar \varepsilon \| \varphi\|_{L^\infty(0,\infty)}.
\end{equation}

Since $\varphi$ is compactly supported and for every compact set $K$ the set $\bigcup_{z \in K} \Sigma_2 \cap \Omega_z $ is bounded, using the fact that $\Phi_{\varepsilon_n}$ converges to $\Phi$ in the \eu{weak$-*$} topology, we have the following limits as $n \to \infty $,
\begin{equation*}
\int_0^\infty  J_2(z,\delta,\Phi_{\varepsilon_n}) \varphi(z) dz \to \int_0^\infty J_2(z,\delta,\Phi) \varphi(z) dz,
\end{equation*}
\begin{equation*}
\int_0^\infty  \int_{(0,z)} x \Phi_{\varepsilon_n}(dx) \varphi(z)dz \to \int_0^\infty  \int_{(0,z)} x \Phi(dx) \varphi(z)dz
\end{equation*} 
\begin{equation*}
\int_0^\infty  \int_{(0,z)}  \Xi_{\varepsilon_n}(x) x \Phi_{\varepsilon_n}(dx) \varphi(z)dz \to \int_0^\infty  \int_{(0,z)} x \Phi(dx) \varphi(z)dz
\end{equation*} 
 and
\begin{equation*}
\int_0^\infty \Xi_{\varepsilon_n}(z) z^2 \varphi(z) \Phi_{\varepsilon_n}(z) dz  \to \int_0^\infty  z^2 \varphi(z) \Phi(z) dz.
\end{equation*} 
These limits, together with \eqref{eq:estimate_J13_eps}, imply  an upper estimate for the difference between the right-hand side of \eqref{eq:weak_flux}  and $\int_{(0, \infty)} J_2(z,\delta, \Phi) \varphi(z) dz$, 
 \begin{eqnarray}\label{eq:estimateLHS}
&\left|\int_0^\infty  J_2(z,\delta, \Phi) \varphi(z) dz-\int_0^\infty  \varphi(z) dz + \int_0^\infty \int_{(0,z]} x\Phi(dx) \varphi(z) dz  
- \frac{2}{1-\gamma}  \int_0^\infty  z^2 \varphi(z)\Phi (z )dz\right|\nonumber \\
& \leq C \bar \varepsilon \| \varphi\|_\infty. 
 \end{eqnarray}

Using \eqref{eq:estimate_Phi}, we can apply again Lemma \ref{ferreira lemma 6.1} to $\Phi$, to conclude that 
 \begin{equation}\label{eq:estimate_J13}
\left|\sum\limits_{j \in \{1,3\}} \int_0^\infty J_j(z,\delta, \Phi) \varphi(z) dz\right|\leq C\bar \varepsilon \| \varphi\|_\infty.
\end{equation}
Finally, estimates \eqref{eq:estimateLHS} and \eqref{eq:estimate_J13}  imply
\eu{\begin{eqnarray*}
\left|  \int_0^\infty \int_{\Omega_z} x K(x,y)\Phi(dx)\Phi(dy) \varphi(z) dz - \int_0^\infty  \varphi(z) dz +  \int_0^\infty  \int_{(0,z]} x\Phi(dx) \varphi(z) dz  \right. \\
+ \left. \frac{2}{1-\gamma}  \int_{\mathbb R_*}  z^2 \varphi(z)\Phi (dz) \right|
\leq C \bar \varepsilon \| \varphi\|_\infty 
\end{eqnarray*}}
for any $\bar \varepsilon>0$ and any $\varphi \in C_c(0,\infty)$, which implies that $\Phi$ satisfies \eqref{eq:ss_eq not AC}. 
\end{proof}

\section{Moment bounds satisfied by the self-similar profile} \label{sec:moments}
\begin{proposition}[Moment bounds]\label{prop:moments}
Assume $K$, $\eta$, $\lambda$ and $\gamma$ to be as in the assumptions of Theorem \ref{thm: existence of a ss sol}. 
The solution of \eqref{eq:ss_eq not AC} with respect to $K$, $\eta$, constructed in Proposition \eqref{prop:Phi}, satisfies for every $\mu \in \mathbb R$
 \begin{equation}\label{moment bounds}
\int_1^\infty x^\mu \Phi(dx)< \infty.
\end{equation}
\end{proposition}
\begin{proof}
The bound \eqref{moment bounds} for $\mu<1$ follows directly from \eqref{first mom bound}.

We now focus on proving the  bound for $\mu >1$. 
\eu{Integrating \eqref{eq:ss_eq} against a positive test function
$ \varphi \in L^1_{\text{loc}}(\R_+)$ and denoting with  $\psi(z):=\int_0^z \varphi(x) dx $}, we obtain
\begin{align} \label{useful ineq for growth bounds}
\frac{2}{1-\gamma} \int_{\mathbb R_*} z^2 \varphi(z)\Phi(dz) \leq  \int_{\mathbb R_*} \int_{\mathbb R_*}  [\psi(x+y) - \psi(x) ] x K(x,y) \Phi(dy) \Phi(dx). 
\end{align}
Let $\delta$ be a small positive constant satisfying the two conditions 
\begin{equation}\label{eq:cond_delta}
 \max \{\gamma,\ \gamma+\lambda,\ -\lambda\} + \delta\ \leq\ 1 \quad \text{ and } \quad 1-\lambda+ \delta \leq 1, \text{ if }\lambda >0  
\end{equation} 
 and choose $\varphi(x)=x^{\delta-1}\chi_{(1,\infty)}(x)$, where $\chi_A(x) = 1$ if $x \in A$ and $\chi_A(x) = 0$ otherwise.  Plugging $\varphi$ in \eqref{useful ineq for growth bounds}, we obtain
\begin{align}\label{eq:delta_main_ineq}
\frac{2}{1-\gamma} \eu{\int_{[1, \infty)} x^{1+\delta} \Phi(dx) } \leq \frac{1}{\delta} \int_{\R_*} \int_{[1, \infty)}  [(x+y)^{\delta} - x^{\delta} ] x K(x,y) \Phi(dx) \Phi(dy)\\
+ \frac{1}{\delta} \int_{\R_*} \int_{(1-y,1]}  (x+y)^{\delta}  x K(x,y) \Phi(dx) \Phi(dy). \nonumber
\end{align} 
Next we derive estimates to each term of the right-hand side. We start by dividing the domain of integration of the first term into two  regions defined by $y<1,\ x\geq 1$ and by $x,y\geq 1$, respectively. 

In the region $y<1,\ x\geq 1$, using the upper bound for the kernel \eqref{kernel gamma lambda} and the estimates \eqref{bound p x y}-\eqref{bound p y x}, we have that
\begin{align}\label{eq:integral2}
&\int_{(0,1)} \int_{[1, \infty)}  [(x+y)^{\delta} - x^{\delta} ] x K(x,y) \Phi(dx) \Phi(dy) \nonumber \\
& \leq C(\delta)\int_{(0,1)} \int_{[1, \infty)}   x^{\delta} y K(x,y) \Phi(dx) \Phi(dy) \nonumber \\
&\leq C(\delta)c_2 \int_{(0,1)} \int_{[1, \infty)} \left( x^{\delta + \gamma+\lambda} y^{1-\lambda} + y^{1+\gamma +\lambda} x^{\delta-\lambda} \right) \Phi(dx) \Phi(dy) 
\end{align}
which is finite by \eqref{first mom bound}-\eqref{bound p q}.

In the second region  $x,y\geq 1$ we use  the symmetry of the kernel, as well as \eqref{bound p x y} and \eqref{bound p y x} to obtain 
\begin{align}
&\iint_{[1,\infty)^2}  [(x+y)^{\delta} - x^{\delta}  ] x K(x,y) \Phi(dy) \Phi(dx) \leq  \label{eq:integral1} \\ 
&\leq C(\delta)\iint_{ \{ 1 \leq y \leq x \} } x^{\delta} y K(x,y) \Phi(dy) \Phi(dx)  + 
c(\delta)\iint_{ \{1 \leq x \leq y \} } y^{\delta} x K(x,y) \Phi(dy) \Phi(dx)   \nonumber\\
&= (C(\delta)+c(\delta)) \iint_{ \{ 1 \leq y \leq x  \} } x^{\delta} y K(x,y) \Phi(dy) \Phi(dx)  \nonumber .
\end{align}
The upper bound for the kernel \eqref{kernel gamma lambda} implies 
\[ x^{\delta} y K(x,y)  \leq c_2   \left( x^{\delta + \gamma+\lambda} y^{1-\lambda} + y^{1+\gamma +\lambda} x^{\delta-\lambda} \right),\] which yields the following estimates, for $1 \leq y \leq x$, 
\begin{align*}
& x^{\delta} y K(x,y)
  \leq  c_2 \left( x^{\delta +\gamma+\lambda} y +   x^{1+ \delta-\lambda} y^{\gamma +\lambda}\right),   \quad \lambda> 0 \\
 & x^{\delta} y K(x,y)
  \leq  c_2 \left( x^{\delta +\gamma} y \left( \frac{x}{y} \right)^\lambda  +   x^{\delta-\lambda+\gamma+\lambda} y \right),  \quad \lambda\leq 0,\ \gamma+\lambda>0\\
&  x^{\delta} y K(x,y)
  \leq  c_2 \left( x^{\delta +\gamma} y \left( \frac{x}{y} \right)^{\lambda} +  x^{\delta-\lambda} y\right),   \quad \lambda\leq 0,\ \gamma+\lambda \leq 0.
\end{align*}
The  bounds \eqref{first mom bound}-\eqref{bound p q}  together with the condition on $\delta$ \eqref{eq:cond_delta} then ensure that \eqref{eq:integral1} is finite.

We now focus on the second term of \eqref{eq:delta_main_ineq}. We divide the domain of integration into two regions defined by $y>1, \ x\leq 1$ and by $1-y< x\leq 1,\ y \leq 1$. 
In the first region $y>1, x\leq 1$, it holds 
\begin{align*}
\int_{(1,\infty)} \int_{(0,1]}  (x+y)^{\delta}  x K(x,y) \Phi(dx) \Phi(dy) 
\leq 2^\delta \int_{(1,\infty)} \int_{(0,1]}  y^{\delta}  x K(x,y) \Phi(dx) \Phi(dy). 
\end{align*}
Similarly to the region $y<1,x \geq 1$ (cf.\ \eqref{eq:integral2}), this  integral is finite  by \eqref{kernel gamma lambda} and \eqref{first mom bound}-\eqref{bound p q}.
In the second region $1-y < x\leq 1,\ y \leq 1$, we notice that  $(x+y)^\delta \leq 2^\delta$. Therefore,
\begin{align*}
\int_{(0,1]} \int_{(1-y,1]}  (x+y)^{\delta}  x K(x,y) \Phi(dx) \Phi(dy) 
\leq 2^\delta\int_{(0,1]} \int_{(1-x,1]}  x K(x,y) \Phi(dx) \Phi(dy) ,
\end{align*}
which is bounded by $J_\Phi(1) $ that is finite by Proposition \ref{prop:Phi}.

Thus, we conclude that
\begin{equation} \label{case 1:bound on 1+d}
\frac{2}{1-\gamma} \int_{[1,\infty)} z^{1+\delta} \Phi(dz)  < \infty.
\end{equation}

Let us select now $\varphi(x)=x^{n \delta-1} \chi_{(1, \infty)}(x)$ in \eqref{useful ineq for growth bounds}
\begin{align*}
&\frac{2}{1-\gamma} \int_{\R_*} x^{1+n\delta} \phi(x) dx \leq \frac{1}{n\delta} \int_{\R_*} \int_{[1, \infty)}  [(x+y)^{n\delta} - x^{n\delta} ] x K(x,y) \Phi(dx) \Phi(dy)  \\
&+ \frac{1}{n\delta} \int_{\R_*} \int_{(1-y,1]}  (x+y)^{n\delta}  x K(x,y) \Phi(dx) \Phi(dy). 
\end{align*}
We divide the domains of integration as before and conclude that, for some constant $ c_{\delta,n}$ that depends on $\delta$ and $n$, the following estimate holds
\begin{align*}
& \frac{2}{1-\gamma} \int_{\R_*} x^{1+n\delta} \phi(x) dx \leq c_{\delta,n} ( \int_{(0,1)} \int_{[1, \infty)}   x^{n\delta} y K(x,y)\phi(y) \phi(x) dx dy
 \\
& +  \iint_{ \{x \geq y \geq 1 \} } x^{n\delta} y K(x,y) \Phi(dx) \Phi(dy) 
 + \int_{(1,\infty)} \int_{(0,1]}  y^{n\delta}  x K(x,y) \Phi(dx) \Phi(dy) \\
& + \int_{(0,1]} \int_{(1-x,1]}  x K(x,y)\Phi(dx) \Phi(dy) ).
\end{align*} 
The last term is bounded by $ c_{\delta,n} J_\Phi(1)< \infty$. The third term may be estimated exactly  as the first term. The first  and second terms may be estimated as  before  using the upper bound for the kernel \eqref{kernel gamma lambda} 
$$ x^{n\delta} y K(x,y) \leq c_2\left( x^{(n-1)\delta + \delta + \gamma+\lambda} y^{1-\lambda} +  y^{1+\gamma +\lambda} x^{(n-1)\delta + \delta-\lambda}\right).$$
It then \eu{follows by} the choice of $\delta$ 
that, for some constant $\tilde c_{\delta,n}$ depending on $\delta$ and $n$, as well as on the parameters of the kernel, we have
\begin{align} \label{case 1:bound on n+d} 
\int_{(1,\infty)}  x^{1+n \delta} \Phi(dx)  \leq 3\tilde c_{\delta,n} \int_{(1,\infty)}  x^{1+(n-1) \delta} \Phi(dx) + c_{\delta,n} J_\Phi(1). 
\end{align}
The bound \eqref{moment bounds} follows by induction combining  \eqref{case 1:bound on 1+d} and \eqref{case 1:bound on n+d}. 
\end{proof}
\begin{proposition}\label{prop:expo}
Assume $K$, $\eta$, $\lambda$ and $\gamma$ to be as in the assumptions of Theorem \ref{thm: existence of a ss sol}. 
The solution $\Phi$ of \eqref{eq:ss_eq not AC}, constructed in Proposition \ref{prop:Phi}, satisfies \eqref{exp bound} for some positive constant $L.$
\end{proposition}
\begin{proof}
We now prove the exponential bound following the approach of \cite{fournier2006local}. 
Let us introduce the notation $\Psi_a(L):=\int_{\mathbb R_*} \frac{x^2}{\min\{ x,a \}} e^{L\min\{x,a\}}\Phi(dx)  . $
Notice that 
\begin{equation}\label{psi_a_0}
\Psi_a(0)= \int_{(0,a)} \frac{x^2}{\min\{ x,a \}} \Phi(dx) + \int_{[a,\infty)} \frac{x^2}{a} \eu{\Phi(dx)} \leq 1+ \frac{M_2}{a}
\end{equation}
and 
\[
\Psi'_a(L)=\int_{\mathbb R_*} x^2e^{L\min\{x,a\}} \Phi(dx) . 
\]
Considering the test function $\varphi'(x) = e^{L\min\{x,a\}} $ in \eqref{useful ineq for growth bounds}, we deduce 
\begin{align*}
& \Psi'_a(L) \leq c(\gamma) \int_{\mathbb R_*} \int_{\mathbb R_*} x K(x,y) \int_x^{x+y} e^{L \min\{z,a\} } dz  \Phi(dy)   \Phi(dx)   \\
&\eu{\leq } c(\gamma) \int_{\mathbb R_*}  x^{\gamma +\lambda} \mu_{a,L}(d x) \int_{\mathbb R_*}  y^{-\lambda}  \mu_{a,L}(dy) 
\end{align*} 
where we are using the notation $\mu_{a,L}(dx):= x  e^{L \min\{x,a\} } \Phi(dx) . $

Since $\gamma+\lambda < 1 $ and $-\lambda <1 $, the maps $x \mapsto x^{-\frac{1}{\lambda }} $ and $x \mapsto x^{\frac{1}{\lambda +\gamma}} $ are convex functions. By Jensen inequality we obtain that
\eu{
\begin{align*}
& \| \mu_{L,a} \|_{1} \int_{\mathbb R_*} x^{\gamma+\lambda} \frac{\mu_{L,a} (dx) }{\| \mu_{L,a} \|_{1}} \leq 
\| \mu_{L,a} \|_1^{1-\gamma -\lambda} \left( \int_{\mathbb R_*} x \mu_{L,a} (dx) \right)^{\gamma+\lambda} \\
&\leq \Psi_a(L)^{1-\gamma-\lambda} \Psi'_a(L)^{\gamma+\lambda} \nonumber
\end{align*} }
and similarly
\begin{align*}
& \| \mu_{L,a} \|_{1} \int_{\mathbb R_*} x^{-\lambda} \frac{\mu_{L,a} (dx) }{\| \mu_{L,a} \|_{1}} \leq \Psi_a(L)^{1+\lambda} \Psi'_a(L)^{-\lambda} \nonumber. 
\end{align*} 
As a consequence $ \Psi'_a(L) \leq c(\gamma) \Psi_a(L)^{2-\gamma} \Psi'_a(L)^{\gamma}$,
or equivalently 
\[  \Psi'_a(L) \leq \Psi_a(L)^{\frac{2-\gamma}{1-\gamma}} c(\gamma)^{\frac{1}{1-\gamma}}.\]
This implies that if $L \leq \Psi_a(0)^{-\frac{1}{1-\gamma}} c(\gamma)^{-\frac{1}{1-\gamma}}$, then
\[
 \Psi_a(L) \leq \left(  \Psi_a(0)^{-\frac{1}{1-\gamma}} -  c(\gamma)^{\frac{1}{1-\gamma}} L \right)^{\gamma-1 }. 
\]
Since by \eqref{psi_a_0} it follows that $\Psi_a(0) \rightarrow 1 \text{ as } a \rightarrow \infty,$
we deduce that if $L \leq  c(\gamma)^{-\frac{1}{1-\gamma}}$ then 
\[
 \int_{(1, \infty)} e^{ L x} \Phi(dx) \leq \limsup_{a \rightarrow \infty}  \Psi_a(L) < \infty. 
\]
\end{proof}

\begin{proof}[Proof of Theorem \ref{thm: existence of a ss sol}]
Let $\Phi $ be a solution of \eqref{eq:ss_eq not AC} constructed in Proposition \ref{prop:Phi}.  We know that $\Phi$ satisfies \eqref{exp bound} and \eqref{eq:estimate_Phi}.  
We now prove that $\Phi$ is absolutely continuous.
 Since $\Phi $ is a solution of \eqref{eq:ss_eq not AC}, by selecting $\varphi(x)=\frac{1}{z^2}\chi(x)$ with $\chi \in C_c(\mathbb R_*)$, we conclude that
\begin{align*}
\frac{2}{1-\gamma} \left|\int_{\mathbb R_*} \chi(z) \Phi(dz)\right| & \leq \left| \int_{\mathbb R_*} J_{\Phi}(z) \frac{1}{z^2} \chi(z) \varphi(z) dz \right| + \left| \int_{\mathbb R_*} \int_{(z, \infty) } x \Phi(dx)\frac{1}{z^2} \chi(z) dz \right| \\
& \leq \left( 1+ \| J_{|\supp \chi} \|_\infty  \right)\| 1/z^2 \|_{L^p(\supp \chi)} \| \chi \|_{L^q(\supp \chi)} <\infty.
\end{align*} 
Therefore, by density we know that for every $q>1$ and any compact set $\mathcal K \subset \mathbb R_*$
\begin{align*}
 \left|\int_{\mathbb R_*} \chi(z) \Phi(dz)\right| & \leq  C(\eta, \lambda, \gamma, \Phi) \| \chi \|_{L^q(\mathcal K)},  \quad \chi \in L^q(\mathcal K). 
\end{align*} 

This implies that $\Phi$ is absolutely continuous and that its density $\phi$ satisfies \eqref{eq:ss_eq}.
We now show \eqref{point estimate}.
\eu{ First of all we notice that 
	\begin{equation*}
	J_\phi (z) = \int_z^\infty x \phi dx + \frac{2}{1-\gamma} z^2 \phi(z) \text{ for a.e. } \ z >0.
	\end{equation*}
As a consequence we deduce that
\begin{equation}\label{phi as a function of flux}
 \phi(z) \leq \frac{1-\gamma}{2  z^2} J_\phi (z)  \text{ for a.e. } z >0.
\end{equation}
Assume $z>1$, hence we can rewrite $J_\phi $ in the following way
\begin{align}\label{split the flux}
& J_\phi(z) = \int_0^z \int_{z-x}^\infty K(x,y) x \phi(y)  \phi(x) dy dx \nonumber \\
& =  \int_0^1 \int_{z-x}^\infty K(x,y) x \phi(y)  \phi(x) dy dx +  \int_1^z \int_{z-x}^\infty K(x,y) x \phi(y)  \phi(x) dy dx.
\end{align}
Using \eqref{bound p q}, we have that, if $z$ is large enough, then there exists a constant $\tilde{L}>0$ such that
\begin{align*}
& \int_0^1 \int_{z-x}^\infty K(x,y) x \phi(y)  \phi(x) dy dx \leq c_2 \int_0^1 x^{\gamma + \lambda +1 } \phi(x) \int_{z-1}^\infty y^{- \lambda} \phi(y)   dy dx \\
 & + c_2
  \int_0^1 x^{1-\lambda } \phi(x) \int_{z-1}^\infty y^{\gamma+ \lambda} \phi(y)   dy dx \leq e^{- z \tilde{L}} 
\end{align*}}

\eu{We now focus on the second term of \eqref{split the flux}, that for large enough $z$ can be written as a sum of two terms in the following way
\begin{align} \label{sum splitting}
& \int_1^z \int_{z-x}^\infty K(x,y) x \phi(y)  \phi(x) dy dx= \int_{z-1}^z \int_{z-x}^\infty K(x,y) x \phi(y)  \phi(x) dy dx \nonumber \\
&+ \int_1^{z-1} \int_{z-x}^\infty K(x,y) x \phi(y)  \phi(x) dy dx. 
\end{align}
Moreover, since 
\begin{align*}
&\int_1^{z-1} \int_{z-x}^\infty K(x,y) x \phi(y)  \phi(x) dy dx \leq c_2 \int_{z/2}^\infty  x^{\gamma + \lambda +1 }\phi(x) dx \int_1^\infty y^{-\lambda} \phi(y) dy \\
& +c_2 \int_{z/2}^\infty x^{1-\lambda } \phi(x) dx \int_1^\infty y^{\gamma +\lambda} \phi(y) dy +c_2 \int_1^\infty   x^{\gamma+\lambda +1} \phi(x)   dx \int_{z/2}^\infty y^{-\lambda} \phi(y) dy  \\
&+ c_2 \int_1^\infty    x^{1-\lambda} \phi(x) dx  \int_{z/2}^\infty y^{\gamma + \lambda} \phi(y)   dy 
\end{align*}
from this we deduce, thanks to \eqref{exp bound}, that 
\[
\int_1^{z-1} \int_{z-x}^\infty K(x,y) x \phi(y)  \phi(x) dy dx \leq c e^{- z L_1}
\]
for suitable constants $L_1>0$ and $c>0.$
We focus now on the first term of \eqref{sum splitting}. 
Notice that there exists a positive constant $L_4 >0$ such that
\begin{align*}
\int_{z-1}^z \int_{1}^\infty x K(x,y) \phi(y) \phi(x)  dy dx \leq c e^{- L_4 z }. 
\end{align*}
Using \eqref{phi as a function of flux} we deduce that 
\begin{align*}
\int_{z-1}^z \int_{z-x}^1 K(x,y) x \phi(y)  \phi(x) dy dx 
\leq c(\gamma)
\int_{z-1}^z \int_{z-x}^1 \frac{K(x,y)}{x} \phi(y)  J_\phi(x) dy dx 
\end{align*}
Combining all the above inequalities we have that there exists a $\rho >0$ such that
\begin{align*}
J_\phi(x) \leq c e^{-z \rho}+ c(\gamma)
\int_{z-1}^z \int_{z-x}^1 \frac{K(x,y)}{x} \phi(y) J_\phi(x)  dy dx
\end{align*}} 

%

\eu{
On the other hand using Lemma \ref{ferreira lemma 2.10} we deduce that
\begin{align*}
&\int_{z-1}^z \int_{z-x}^1\frac{K(x,y)}{x} \phi(y) J_\phi(x)  dy dx \leq \sup_{x \in [z-1,z]} J_\phi(x) \int_{z-1}^z \int_{z-x}^1 \frac{K(x,y)}{x} \phi(y) dy dx \\
& \leq \sup_{x \in [z-1,z]} J_\phi(x) \left[  \int_{z-1}^z \frac{x^{\gamma + \lambda}}{x} \int_{z-x}^1 y^{ \frac{-2 \lambda - \gamma -3}{2}} dy dx +\int_{z-1}^z  \frac{x^{- \lambda}}{x}  \int_{z-x}^1 y^{ \frac{2 \lambda + \gamma -3}{2}} dy dx \right] \\
& \leq c \sup_{x \in [z-1,z]} J_\phi(x) \left[  \int_{z-1}^z x^{\gamma + \lambda-1} (z-x)^{ \frac{-2 \lambda - \gamma -1}{2}}  dx +\int_{z-1}^z  x^{- \lambda-1}  (z-x)^{ \frac{2 \lambda + \gamma -1}{2}} dx \right] \\
&\leq c  \sup_{x \in [z-1,z]} J_\phi(x) \left[ (z-1)^{\gamma + \lambda-1} \int_{0}^1  y^{ \frac{-2 \lambda - \gamma -1}{2}}  dy +(z-1)^{- \lambda-1}\int_0^1  y^{ \frac{2 \lambda + \gamma -1}{2}}  dy \right]
\end{align*}
This allows us to deduce that there exists a $N_\varepsilon$ such that every $z> N_\varepsilon$ 
\begin{align} \label{delay}
J_\phi(z) \leq c e^{-z \rho}+ \varepsilon\sup_{x \in [z-1,z]} J_\phi(x) 
\end{align} 
with $e^{-\rho} >\frac{\varepsilon}{1-\varepsilon} >0$.
Let 
\[
\tilde{J}_n := \sup_{z \in [n-1,n]} J_\phi(z) 
\]
then by \eqref{delay} we deduce that 
\begin{equation}\label{recursive}
\tilde{J_n} \leq \varepsilon \tilde{J}_{n-1} +\varepsilon \tilde{J}_{n}+ c e^{- n \rho} 
\end{equation}
and by the local boundedness of $J_\phi$
we have that there exists a $c_0>0$ such that $\tilde{J}_{N_\varepsilon} \leq c_0 $. 
Formula \eqref{recursive} implies that for every $k >0$
\[
\tilde{J}_{N_\varepsilon+k} \leq \frac{c}{1-\varepsilon} e^{- \rho (N_\varepsilon+k)} \sum_{i=0}^k  e^{\rho i} \left( \frac{\varepsilon}{1-\varepsilon} \right)^i + c_0 e^{-\rho k} 
\]
As a consequence, for large $z$ we have that $J_\phi(z)  \leq c e^{- \rho z\label{key}}$ where $c$ is a positive constant.}


The inequality $\phi (z) \leq c e^{-\rho z}$ as $z \rightarrow \infty $ follows by the fact that $J_\phi(z) \leq c e^{-\rho z}$ and by equation \eqref{eq:ss_eq}.

We now prove the lower power-law bound \eqref{eq:estimate_Phi above}. 
Let $C$ be any positive constant. For  small enough positive constants $z_0$ and $\varepsilon$ such that $1 - C\varepsilon - \int_0^{z_0} x \phi(x)dx \geq \frac{1}{2}$, it follows by Lemma \ref{ferreira lemma 6.1} that there is a $\delta_0>0$ depending on $\varepsilon$ and $\gamma$ such that for any $\delta\leq \delta_0$ and any $\varphi \in C_c(\R_*)$ 
\begin{equation} \label{eq:J2_lowerbound}
J_2(z,\delta, \phi)   \geq 1+ \frac{2}{1-\gamma}z^2 \phi(z) - C \varepsilon-\int_0^z x\phi(x)xdx \geq \frac{1}{2}, \quad a.e. \ z \in (0,z_0]. 
\end{equation}
\eu{ Moreover,  using a geometrical argument, we deduce that there exist a constant $b \in (0,1)$ depending on $\delta$ such that, for all $z \in [R,2R]$,
\begin{equation*}
\Omega_z \cap \Sigma_2(\delta) \subset \left(\frac{\delta}{\delta+1}z,z\right ]\times \left(\frac{\delta}{\delta+1}z,\frac{z}{\delta}\right] \subset \left(\sqrt{b}R, R/\sqrt{b}\right]^2,
\end{equation*}
with $\Omega_z$ defined by \eqref{eq:Omega_R}, which together with the upper bound for the kernel \eqref{kernel gamma lambda},  $xK(x,y) \leq c R^{\gamma+1}$, yields 
\begin{equation*}
\int_{[R,2R]} J_2(z,\delta, \phi) dz \leq c  R^{\gamma+2} \left( \int_{(\sqrt{b}R, R/\sqrt{b}]} \phi(x)dx \right)^2, \quad R>0 
\end{equation*}
for a constant $c>0$ depending only on $\gamma, \lambda, c_1$ and $c_2$.
This together with \eqref{eq:J2_lowerbound} implies
\begin{equation*}
\frac{R}{2} \leq c  R^{\gamma+2} \left( \int_{(\sqrt{b}R, R/\sqrt{b}]} \phi(x)dx \right)^2, \quad R \in (0,z_0].
\end{equation*}
Hence, the result follows after substituting $R/\sqrt{b}$ by $z$. }

Finally $\int_{\mathbb R_*} x \phi(x) dx =1$ follows by the fact that for every sequence $\{ z_n\}$ such that $\lim_{n \rightarrow \infty} z_n=\infty $ we have that $\lim_{n \rightarrow \infty} J_\phi(z_n)=0$ and, thanks to \eqref{exp bound}, up to a subsequence $\lim_{n \rightarrow \infty }z_n^2\phi(z_n)=0$.
\end{proof}

\section{Regularity of the self-similar profiles} \label{sec:regularity ss}
\begin{proof}[Proof of Theorem \ref{thm:regularity}] 

We divide the proof into steps.

\textbf{Step 1}: for every $\psi \in C^\infty_c(\mathbb R)$, every $\beta >0$ and $0<s<1$ it holds that
\begin{equation}\label{ineq difference}
\| \psi(\cdot+y ) - \psi(\cdot) \|_{H^{-\beta} (\mathbb R)} \leq |y|^s \| \psi \|_{H^{s-\beta}}(\mathbb R) \quad y \in \mathbb R.
\end{equation}
This \eu{follows by} the fact that 
\begin{align*}
&\| \Delta_y \psi(\cdot) \|^2_{H^{-\beta} (\mathbb R)} = \int_{\mathbb R }  (1+|x |^2)^{-\beta} | \Delta_y \hat{\psi}(x)|^2 dx \\
&= \int_{\mathbb R }  (1+|x |^2)^{-\beta} |\hat{\psi}(x) |^2 | e^{ixy}-1 |^2 dx \leq c\int_{\mathbb R }  (1+|x |^2)^{-\beta} |\hat{\psi}(x) |^2 |x|^{2s} |y|^{2s} dx  \\
&  \leq \eu{c} |y|^{2s}   \int_{\mathbb R } |\hat{\psi}(x) |^2 (1+|x|^2)^{s-\beta} dx =\eu{c} |y|^{2s} \| \psi(\cdot) \|^2_{H^{s-\beta} (\mathbb R)}. 
\end{align*}
where we are using the notation $\Delta_y \psi(\cdot) :=  \psi(\cdot+y ) - \psi(\cdot)$.

\textbf{Step 2}: let $U$ be an open set and let us define the operator $T:C^\infty_c(\mathbb R) \rightarrow C^1(U) $ as 
\[
T[\psi](x):=\int_0^1 K(x,y)\phi(y) \Delta_y \psi(x) dy \quad x \in U. 
\]
Then 
\begin{equation} \label{norm of T in L2}
\|T[\psi] \|_{L^2 (U) } \leq C \|\psi \|_{H^{\overline s}(\mathbb R)}, 
\end{equation}
for some positive constant $C>0$ and 
\begin{equation}\label{barra s}
 1 > \overline s  >  \frac{\gamma+1}{2} - \min\{\gamma+\lambda, -\lambda \} \geq \frac{1}{2}. 
\end{equation}
To prove \eqref{norm of T in L2} we notice that
\begin{align*}
\|T[\psi] \|_{L^2 (U) }  &\leq \int_0^1 \left\| K(x,y)\phi(y) \Delta_y \psi(x) \right\|_{L^2 (U)} dy \\ 
& \leq C \int_0^1 y^{\min\{\gamma+\lambda, -\lambda \} } \phi(y) \left\| \Delta_y \psi(x)  \right\|_{L^2 (\mathbb R)} dy \\
&\leq C \int_0^1 y^{\min\{\gamma+\lambda, -\lambda \} +\overline{s} } \phi(y) dy \| \psi  \|_{H^{\overline{s}}(\mathbb R)}. 
\end{align*}
Where, for the last inequality, we use the fact that $\| \Delta_y \psi(x)  \|_{L^2 (\mathbb R)} \leq C |y|^{\overline s } \| \psi  \|_{H^{\overline s}(\mathbb R)}$,
see \cite{simon1990sobolev}.

Inequality \eqref{eq:estimate_Phi} and \eqref{barra s} imply that 
$ \int_0^1 y^{\min\{\gamma+\lambda, -\lambda \} +\overline{s} } \phi(y) dy  < \infty$
and, therefore, \eqref{norm of T in L2} follows.  

\textbf{Step 3}: if $l \geq \beta >1/2$, then 
\begin{equation} \label{norm of T in Hs}
\|T[\psi] \|_{H^{-\beta} (U) } \leq C \|\psi \|_{H^{\overline s-\beta} (\mathbb R) }. 
\end{equation}

First of all we notice that if $f \in H^l (U)$ and $g \in H^{-\beta}(\mathbb R) $, then $ fg \in H^{-\beta }(U)$. 
This is due to the fact that, if $\beta > 1/2$
\begin{align*}
&\left\| fg \right\|_{H^{-\beta}(U ) }= \inf_{\{E \in H^{-\beta }(\mathbb R): E_{|U}=fg \} } \left\| E \right\|_{H^{-\beta}(\mathbb R) } \leq \inf_{\{F \in H^{-\beta }(\mathbb R): F_{|U}=f \} } \left\| g F \right\|_{H^{-\beta}(\mathbb R) } \\
& \leq \inf_{\{F \in H^{-\beta }(\mathbb R): F_{|U}=f \} } \left\| F \right\|_{H^{-\beta}(\mathbb R) } \| g \|_{H^{-\beta}(\mathbb R) }  =   \left\| f \right\|_{H^{-\beta}(U) }  \| g \|_{H^{-\beta}(\mathbb R) }  < \infty. 
\end{align*} 
Since $K \in H^l (\mathbb R_*) $ as well as the fact that it satisfies \eqref{kernel derivatives}, we deduce that 
\begin{align*}
& \|T[\psi] \|_{H^{-\beta} (U) }  \leq \left\| \int_0^1 K(x,y)\phi(y) \Delta_y\psi(x) dy \right\|_{H^{-\beta} (U)} \\
& \leq \int_0^1 \left\|  K(x,y) \Delta_y\psi(x)  \right\|_{H^{-\beta} (U)} \phi(y)  dy  \leq C \int_0^1 y^{\min\{\gamma+\lambda, -\lambda \}} \phi(y) \left\|  \Delta_y\psi(x) \right\|_{H^{-\beta} (\mathbb R)} dy \\
& \leq  C \left\| \psi \right\|_{H^{\overline s-\beta} (\mathbb R)}   \int_0^1 y^{\min\{\gamma+\lambda, -\lambda \}+\overline s} \phi(y) dy \leq  C  \left\| \psi \right\|_{H^{\overline s-\beta} (\mathbb R)}.
\end{align*}
for some positive constant $C$. To deduce the second last inequality, we applied \eqref{ineq difference}. 

\textbf{Step 4}

In this last step we show how to combine the results of the previous steps to prove the regularity of the self-similar profile. 
Considering a test function $\varphi(z)=\psi'(z) $ with $\psi \in C^\infty_c(\mathbb R_*)$ and $\supp (\psi)=[a,b]$ for some $a>b>0$ in \eqref{eq:ss_eq not AC}, we deduce that 
\begin{align*}
&\frac{2}{1-\gamma} \int_{\mathbb R_*} z^2 \psi'(z) \phi(z) dz =\int_{\mathbb R_*}x \phi(x) \left( \psi(x) +\int_{\mathbb R_*}  \Delta_y\psi(x)   K(x,y) \ \phi(y) \right) dx dy. 
\end{align*}
Therefore
\begin{align*}
&\frac{2}{1-\gamma} \int_{\mathbb R_*} \left|  z^2 \psi'(z) \phi(z) \right| dz \leq  \| \psi \|_{L^2([a,b])} \| x \phi(x) \|_{L^2([a,b])} \\
&+ \int_a^b \int_1^\infty   \Delta_y \psi(x)  x K(x,y) \phi(x) \phi(y)  dy dx  + \int_a^b x T[\psi](x)  \phi(x)   dx \\ 
&+ \int_0^a  \int_{a-x}^\infty \psi(x+y) x K(x,y) \phi(x) \phi(y) dx dy. 
\end{align*}
We notice that, for some $C>0$, 
\begin{align*}
& \int_a^b \int_1^\infty  \Delta_y \psi(x) x K(x,y) \phi(x) \phi(y)  dy dx \\
&\leq C  \| \psi \|_{L^2([a,b])} \int_a^b \int_1^\infty  x K(x,y) \phi(x) \phi(y)  dy dx  \leq C \| \psi \|_{L^2([a,b])}
\end{align*} 
and 
\begin{align*}
& \int_0^a  \int_{a-x}^\infty \psi(x+y) x K(x,y) \phi(x) \phi(y) dy dx   \leq  \| \psi \|_{L^2([a,b])} J_\phi(a) \leq C  \| \psi \|_{L^2([a,b])}. 
\end{align*} 

Thanks to Step 2 we deduce that 
\begin{align*}
& \int_a^b x T[\psi](x)  \phi(x) dx \leq \| x \phi(x)  \|_{L^2([a,b] )} \|  T[\psi] \|_{L^2([a,b] )} \leq  \| x \phi(x)  \|_{L^2([a,b] )} \|  T[\psi]\|_{L^2(\mathbb R )}\\
& \leq  \| x \phi(x)  \|_{L^2([a,b] )} \| \psi \|_{H^{\overline s} (\mathbb R)}. 
\end{align*}
By denoting with $\Theta $ the function $z \mapsto z^2 \phi(z) $ we conclude that 
\begin{align*}
 \int_0^\infty \psi'(z) \Theta(z) dz \leq C  \| \psi \|_{H^{\overline s} (\mathbb R)}. 
\end{align*} 
This inequality implies that $\Theta '(z) \in (H^{\overline s})' (\mathbb R)=H^{-{\overline s}} (\mathbb R)$ and therefore $\Theta \in H^{1-{\overline s} }(\mathbb R).$ 
For every test function $\zeta \in C^\infty_c(\mathbb R) $, we have that $\Theta \zeta \in H^{1-{\overline s} }(\mathbb R)$. 

Assume now that  $ \zeta \Theta  \in H^{n(1-{\overline s})}(\mathbb R)$ for any test function $\zeta$ and for some $n \geq 1$.  Then if $l \geq n({\overline s}-1)$, considering a test function which is equal to $\frac{1}{x}$ on $[a,b]$ we deduce that
\begin{align*}
&\int_a^b x T[\psi](x)  \phi(x) dx \leq C  \int_a^b \zeta(x) x^2 T[\psi](x)  \phi(x) dx \\
&\leq \| \zeta \Theta \|_{H^{n(1- {\overline s}) }(\mathbb R)} \|  T[\psi](x) \|_{H^{n({\overline s}-1)}(a,b )} \leq\| \zeta \Theta \|_{H^{n(1-{\overline s})}(\mathbb R)}\| \psi \|_{H^{(n+1){\overline s}-n}(\mathbb R)} . 
\end{align*}
This implies that 
\begin{align*}
 \int_0^\infty z^2   \psi'(z) \phi(z) dz \leq c \| \psi \|_{H^{(n+1){\overline s}-n}(\mathbb R)}  
\end{align*}
and, therefore, $ \Theta \in  H^{(n+1)(1- {\overline s } )}(\mathbb R)$ and the desired result follows. 

Recalling the Sobolev embeddings (\cite{duistermaatdistributions}) and,
differentiating \eqref{eq:ss_eq}, we deduce that $\phi$ satisfies \eqref{coag eq self sim}. 
\end{proof}

\section{Self-similar solutions for the coagulation with constant flux coming from the origin} \label{sec:coag with influx}
\begin{proof}[Proof of Theorem \ref{thm:asymptotic}]
The fact that $F$ satisfies \eqref{F bound} follows by \eqref{bound p q}. 

We notice that for every $\varepsilon >0 $ and every $\varphi \in C^1([0,T], C^1_c(\mathbb R_*))$
\begin{align*}
&\int_{\mathbb R_*} \xi \varphi(\varepsilon,\xi) F(\varepsilon,d \xi) = \int_{\mathbb R_*}  \xi \varphi(\varepsilon,\xi) \varepsilon^{-\frac{\gamma+3}{1-\gamma}} \phi(\xi \varepsilon^{-\frac{2}{1-\gamma}} ) d\xi \leq 
\int_{\mathbb R_*} \varphi(\varepsilon, y \varepsilon^{\frac{2}{1-\gamma}}  ) \varepsilon \phi(y) dy \\
& \leq \varepsilon \| \varphi \|_\infty  
\end{align*}
therefore,$\int_{\mathbb R_*} \xi \varphi(t,\xi) F(t,\xi) d\xi \rightarrow 0 \text{ as } t \rightarrow 0. $
Consequently, via a change of variables and integral manipulations, we deduce that 
\begin{align*}
&\int_{\mathbb R_*}\xi \varphi(t, \xi)  F(t,\xi) d\xi  -   \int_0^t  \int_{\mathbb R_*}  \xi \partial_s \varphi(s, \xi)  F(s,\xi) d\xi \\
&= \int_0^t \int_{\mathbb  R_*} x \varphi(s, x s^{\frac{2}{1-\gamma}}) \phi(x) dx ds +\frac{2}{1-\gamma} \int_0^t \int_{\mathbb  R_*} x^2 \partial_x \varphi(s, x s^{\frac{2}{1-\gamma}}) \phi(x) dx  ds \\
&= \int_0^t \int_{\mathbb  R_*} \int_{(0,x]}\partial_y \varphi(s, y s^{\frac{2}{1-\gamma}}) dy x \phi(x)  dx ds 
+\frac{2}{1-\gamma} \int_0^t \int_{\mathbb  R_*} x^2 \partial_x \varphi(s, x s^{\frac{2}{1-\gamma}}) \phi(x) dx   ds \\
&= \int_0^t \int_{\mathbb  R_*}  \partial_z \varphi(s, z s^{\frac{2}{1-\gamma}})  \left( \int_{(z, \infty)} x \phi(x) dx
+ \frac{2}{1-\gamma}  z ^2  \phi(z) dz  \right) ds. \\
\end{align*} 
Since $\Phi $ solves \eqref{eq:ss_eq not AC}, by considering the test function $\psi(x) =\partial_x \varphi(s, x s^{\frac{2}{1-\gamma}}) $ for a fixed $s>0$, we deduce that 
\begin{align*}
&\int_{\mathbb R_*}\xi \varphi(t, \xi)  F(t,\xi) d\xi  -   \int_0^t  \int_{\mathbb R_*}  \xi \partial_s \varphi(s, \xi)  F(s,\xi) d\xi = \int_0^t \int_{\mathbb  R_*}  \partial_z \varphi(s, z s^{\frac{2}{1-\gamma}})  J_\phi(z) dz . 
\end{align*} 
Since 
\begin{align*}
&\int_{\mathbb  R_*}  \partial_z \varphi(s, z s^{\frac{2}{1-\gamma}})  J_\phi(z) dz  \\
&=\int_{\mathbb R_*} \int_{ \mathbb R_*} \left( \varphi(s, (z+\xi) s^{\frac{2}{1-\gamma}})  - \varphi(s, z s^{\frac{2}{1-\gamma}})  \right) z K(z, \xi) \phi(z) dz \phi(\xi) d\xi \\
&=\int_{\mathbb R_*} \int_{ \mathbb R_*} \frac{K(y,x)}{2} \left( (y+x ) \varphi(s, x+y)  - x \varphi(s, x) - y\varphi(s,y)  \right) F(s,dy) F(s, dx), 
\end{align*}
thus $F$ solves \eqref{eq:F} for every $ \varphi \in C^1([0,T], C^1_c(\mathbb R_*))$. 

Every test function $\varphi \in C^1([0,T], C^1_c(\mathbb R_+))$ can be approximated by a sequence of functions $ \{ \varphi_n \} \subset  C^1([0,T], C_c(\mathbb R_*))$ defined by \eu{$\varphi_n(s,x)=\zeta(xn) \varphi(s,x)$} with  $\zeta \in C^\infty (\mathbb R_+)$ such that $\zeta(x)=1$ if $x \geq 1$ and $\zeta(x)=0$  if $x \leq 1/2.$ 
For every $n \in \mathbb N $ it holds that 
\begin{align}\label{eq:Fn}
 &\int_{\mathbb R_*} \xi \varphi_n(t, \xi)  F(t,\xi) d\xi =  \int_0^t  \int_{\mathbb R_+}  \xi \partial_s \varphi_n(s, \xi)  F(s,\xi) d\xi  ds \\
&+ \int_0^t  \int_{\mathbb R_+} \partial_{\xi} \varphi_n(s,\xi) J_{F(s,\cdot)}(\xi) d\xi ds \nonumber
\end{align}  
Since $\partial_{\xi} \varphi_n(s,\xi ) = \varphi(s,\xi) n  \zeta'(\xi n) + \zeta(\xi n) \partial_\xi \varphi(s,\xi) $, $J_{F(s , \xi)}= J_\phi (\xi s^{-\frac{2}{1-\gamma}}) < \infty $, and moreover $ \int_0^\infty \xi F(t,\xi ) d\xi \leq t$, by Lebesgue's dominated convergence theorem we deduce that
\begin{align*}
 \int_0^t  \int_{\mathbb R_+}  \zeta(\xi n) \partial_\xi \varphi(s,\xi)  J_{F(s,\cdot)}(\xi) d\xi ds \rightarrow  \int_0^t  \int_{\mathbb R_+} \partial_\xi \varphi(s,\xi)  J_{F(s,\cdot)}(\xi) d\xi ds   \text{ as } n \rightarrow \infty, 
\end{align*} 
\begin{align*}
\int_{\mathbb R_*} \xi \varphi_n(t, \xi)  F(t,\xi) d\xi \rightarrow \int_{\mathbb R_*} \xi \varphi(t, \xi)  F(t,\xi) d\xi \text{ as } n \rightarrow \infty
\end{align*} 
and 
\begin{align*}
\int_0^t  \int_{\mathbb R_+}  \xi \partial_s \varphi_n(s, \xi)  F(s,\xi) d\xi  ds \rightarrow \int_0^t  \int_{\mathbb R_+}  \xi \partial_s \varphi(s, \xi)  F(s,\xi) d\xi  ds \text{ as } n \rightarrow \infty. 
\end{align*}
For any $n >0$
\begin{align*}
 \int_0^\varepsilon  \int_{\mathbb R_+} n  \zeta'(\xi n) \varphi(s,\xi)  J_{F(s,\cdot)}(\xi) d\xi ds \rightarrow 0 \text{ as } \varepsilon \rightarrow 0. 
\end{align*} 

\eu{
This, together with the fact that for any $n \in \mathbb N$ it holds that $  \int_0^\infty n  \zeta'(\xi n) d\xi =1 $, that $\supp(n  \zeta'(\cdot n) )\subset \left[\frac{1}{2n}, \frac{1}{n}\right] $  and that $\int_0^\infty x \phi(x) dx =1 $ imply 
\begin{align*}
& \int_\varepsilon^t  \int_{\mathbb R_+} n  \zeta'(\xi n) \varphi(s,\xi)  J_{F(s,\cdot)}(\xi) d\xi ds  \\
& = \int_\varepsilon^t  \int_{\mathbb R_+} n  \zeta'(\xi n) \varphi(s,\xi)  \left[ 1- \int_0^{\xi s^{- \frac{2}{1-\gamma}}}  x \phi(x) dx + \frac{2}{1-\gamma} \left( \xi s^{-\frac{2}{1-\gamma}}\right)^2 \phi(\xi s^{-\frac{2}{1-\gamma}})  \right] d\xi ds  \\
& \rightarrow   \int_0^t  \varphi(s,0 ) ds \text{ as } n \rightarrow \infty \text{ and } \varepsilon \rightarrow 0. 
\end{align*}  }
\eu{ Notice that we have used the fact that 
\begin{align*}
& \int_\varepsilon^t  \int_{\mathbb R_+} n  \zeta'(\xi n) \varphi(s,\xi) \int_0^{\xi s^{- \frac{2}{1-\gamma}}}  x \phi(x) dx d\xi  ds  \\
& \leq \int_\varepsilon^t  \int_{\mathbb R_+} n  \zeta'(\xi n) \varphi(s,\xi) \int_0^{\frac{s^{- \frac{2}{1-\gamma}}}{n}}  x \phi(x) dx  d\xi ds \rightarrow 0 \text{ as } n \rightarrow \infty 
\end{align*} 
and that, by Lemma \ref{ferreira lemma 6.1}, 
\begin{align*} 
&\int_\varepsilon^t  \int_{\mathbb R_+} n  \zeta'(\xi n) \varphi(s,\xi) \left( \xi s^{-\frac{2}{1-\gamma}}\right)^2 \phi(\xi s^{-\frac{2}{1-\gamma}}) d\xi  ds \\
& \leq \int_\varepsilon^t  \int_{1/2n}^{1/n}  \varphi(s,\xi) \left( \xi s^{-\frac{2}{1-\gamma}}\right)^2 \phi(\xi s^{-\frac{2}{1-\gamma}}) d\xi  ds \leq c \left(\frac{1}{n}\right)^{\frac{3-\gamma}{2}} \rightarrow 0 \text{ as } n \rightarrow \infty. 
\end{align*} 
} 
Passing to the limit as $n \rightarrow \infty $ in equation \eqref{eq:Fn}, we deduce that $F $ is a solution of the coagulation with constant flux coming from the origin in the sense of Definition \ref{def:coag eq flu origin}.  
\end{proof}

\appendix
\renewcommand{\theequation}{A.\arabic{equation}}
\section{Appendix} \label{appendix}  \setcounter{equation}{0}

In this Section we write the proof of some auxiliary Lemmas.
\begin{proof}[Proof of Lemma \ref{lem:continuity F}]
First of all we show that for each $f \in C([1,T], \mathcal M_{+,b} (\mathbb R_*))$ and for every $t \in [1,T]$ the functional $\mathcal F[f] (t)$ is a linear and continuous functional on $C_0(\mathbb R_*)$, and consequently, defines a measure in $\mathcal M_{+,b}(\mathbb R_*). $
The linearity follows directly from the definition. 

To check the continuity we notice that for every $\varphi_1, \varphi_2 \in C_c(\mathbb R_*)$
\begin{align*}
\left| 
\langle \mathcal F_1[f](t) , \varphi_1 - \varphi_2 \rangle \right|
\leq \|\varphi_1 - \varphi_2  \|_\infty \| \Phi_0\| 
\end{align*}
and
\begin{align*}
&\langle \mathcal F_2[f](t) , \varphi_1 - \varphi_2 \rangle
 \leq   \frac{a}{2} \|\varphi_1 - \varphi_2  \|_\infty  \max \left\{ \frac{1}{| \beta \gamma|}t^{\beta \gamma +1} , t \right\} \| f\|^2_{[1,T]},
\end{align*}
\begin{align*}
\langle \mathcal F_3[f](t) , \varphi_1 - \varphi_2 \rangle 
\leq \|\varphi_1 - \varphi_2  \|_\infty   \| \eta_\varepsilon \| \tilde{T}, 
\end{align*}
where $\tilde{T}:=T-1.$
Combining all the above inequalities we conclude that
\begin{align*}
\langle \mathcal F[f](t) , \varphi_1 - \varphi_1 \rangle \leq \|\varphi_1 - \varphi_2  \| \left( \| \Phi_0\|  + \frac{a}{2} \max\left\{ \frac{1}{|\beta \gamma|}t^{\beta \gamma +1}, t\right\}   \| f\|^2_{[1,T]}+  \| \eta_\varepsilon \| \tilde{T} \right). 
\end{align*}
Given the fact that $C_0^* (\mathbb R_*) $ is isomorphic with the space $\mathcal M_{b}(\mathbb R_*) $ (see for instance \cite{rudin2006real}) we conclude that the operator $\mathcal F[f](t)$ defines a measure. \\

We now check that $\mathcal F$ maps $ C([1,T], \mathcal X_\varepsilon)$ into itself.  We already showed that $\mathcal F[f](t) \in \mathcal M_{+,b}(\mathbb R_*)$. The fact that $\mathcal F[f](t)((0,\varepsilon])=0$ follows easily by the fact that $\mathcal F_i(t)((0,\varepsilon])=0$ for $i=1,2$. Indeed $\mathcal F_i(t)$ are defined as integrals over measures that are equal to zero in the set $(0,\varepsilon]$. 
Moreover, since for any $t>0$ we have that $\ell(t,x,y) \geq \max\{ x,y \}$, we deduce that for every test function with support contained in $(0,\varepsilon]$ we have that if $x,y>\varepsilon $ then $\Lambda[\varphi] =0$. Thus, for this choice of test function
\begin{align*}
&\langle \mathcal F_2 [f] (t) , \varphi \rangle= \int_1^t \int_{(\varepsilon, \infty)} \int_{(\varepsilon, \infty )} \Lambda[\varphi](s,x,y) e^{- \int_s^t b[f] (\xi,x) d\xi } \frac{K_{a,T}(s,x, y)}{2} f(s,dx)f(s,dy)ds=0, 
\end{align*} 
hence $\mathcal F_2(t)((0,\varepsilon])=0$.

Let us firstly show that $\mathcal F[f]$ is a continuous map from $[1, T] $ to $\mathcal M_{+,b}(\mathbb R_*).$ To this end we notice that for every $t_2\geq t_1>0$  for every $\varphi \in C_0(\mathbb R_*)$ with $\| \varphi \| \leq 1$ we have that
\begin{align*}
\left|  \langle \varphi, \mathcal F_1[f] (t_2) - \mathcal F_1[f] (t_1) \rangle  \right| 
\leq a \max\left\{  \frac{1}{| \gamma \beta |} \left( {t_2}^{\gamma \beta+1} -{ t_1}^{\gamma \beta+1} \right) , t_2-t_1\right\} \| \Phi_0 \| \| f \|_{[1, T]}
\end{align*}

On the other side, we have 
\begin{align*}
& \left|  \langle \varphi, \mathcal F_2[f] (t_2) - \mathcal F_2[f] (t_1) \rangle  \right| 
 \leq \frac{a}{2} \| f\|_{[1,T]}^2    \max\left\{ \frac{1}{|\gamma \beta |}\left(t_1^{\gamma\beta+1} - t_2^{\gamma\beta+1}\right), t_1-t_2 \right\} \cdot \\
&\cdot  \left( 1  + a \max \left\{ \frac{1}{ \gamma \beta } \tilde{T}^{\gamma \beta+1}  ,\tilde{T} \right\} \| f\|_{[1,T]} \right)
\end{align*} 
Moreover, 
\begin{align*}
& \left|  \langle \varphi, \mathcal F_3[f] (t_2) - \mathcal F_3[f] (t_1) \rangle  \right| \leq (t_2 - t_1)\left( a \max\{\tilde{T}^{\gamma \beta},\tilde{T}\}   \|f\|_{[1,T]} +1 \right) \| \eta_\varepsilon \| . 
\end{align*} 
If $\gamma <0 $, then 
\[
\max\left\{ \frac{1}{|\gamma \beta |}\left(t_1^{\gamma\beta+1} - t_2^{\gamma\beta+1}\right), t_1-t_2 \right\}  \leq \max\left\{1, \frac{1}{|\gamma \beta |} \right\}\left( t_1-t_2 \right)
\] 
and, if $\gamma\geq 0$, then 
\[
\max\left\{ \frac{1}{|\gamma \beta |}\left(t_1^{\gamma\beta+1} - t_2^{\gamma\beta+1}\right), t_1-t_2 \right\}  \leq \max\left\{1, \frac{1}{|\gamma \beta |} \right\} \left(t_1^{\gamma\beta+1} - t_2^{\gamma\beta+1}\right).
\] 
Therefore, the continuity of $\mathcal F[f]$ follows by the above inequalities.
\end{proof}
 \begin{proof}[Proof of Lemma \ref{lem:F contraction}]
For every $\varphi \in C_c(\mathbb R_*)$ with $\| \varphi \| \leq 1 $ we have 
\begin{align*}
\left| \langle \varphi, \mathcal F_1[f] (t)- \mathcal F_1 [g] (t) \rangle  \right| 
 \leq   a \max\left\{\frac{1}{\gamma \beta} \tilde{T}^{\gamma \beta +1} ,\tilde{T}\right\} \| f-g \|_{[1,T]} \| f_1\| ,
\end{align*}  
\begin{align*}
\left| \langle \varphi, \mathcal F_3[f] (t)- \mathcal F_3 [g] (t) \rangle  \right| 
\leq a \| f-g \|_{[1,T]}  \max\left\{ \frac{ 1}{\gamma \beta}  \tilde{T}^{\gamma \beta+1} ,\tilde{T} \right\}\|  \eta_\varepsilon \|
\end{align*} 
and
\begin{align*}
 \langle \varphi, \mathcal F_2[f] (t)- \mathcal F_2 [g] (t) \rangle &  \leq 2 a \left( 1+2\| f_1\| \right) \| f-g\|_{[1,T]}     \max \left\{  \frac{1}{\gamma \beta } \tilde{T}^{\gamma \beta+1},\tilde{T}\right\} \\
& +  a^2 \|f-g  \|_{[1,T]} \left( 1+ 2 \| f_1\| \right)^2  \max\left\{ \frac{1}{\gamma^2 \beta^2}  \tilde{T}^{2(\gamma \beta +1)} ,\tilde{T}^2 \right\}.
\end{align*}
where $\tilde{T}:=T-1$. To obtain the above inequalities we have used the fact that $| e^{-x_1} - e^{- x_2 } | \leq | x_1 -x_2 | $, whenever $x_1>0$ and $x_2>0$.

Summarizing
\[
\sup_{t \in[1,T]} \langle \varphi, \mathcal F[f](t)-  \mathcal F[g](t)\rangle \leq  C_T \| f-g\|_{[1,T]}
\]
where 
\begin{align}
& C_T:= a   \max \left\{  \frac{1}{|\gamma \beta| } \tilde{T}^{\gamma \beta+1},\tilde{T}\right\} \cdot \\
&\cdot  \left( \| \eta_\varepsilon\| + \left( 1+ 2 \| f_1 \| \right)^2 a   \max \left\{  \frac{1}{| \gamma \beta| } \tilde{T}^{\gamma \beta+1},\tilde{T} \right\}   + 3 \left( 1+ 2 \| f_1 \| \right)   \right).\nonumber
\end{align}
It is possible to verify that if 
\begin{equation}\label{bound for gamma>0}
\tilde{T}  < \frac{1}{\min\left\{1, \frac{1}{|\gamma \beta|}\right\} } \left( \frac{1}{ 10a } \min\left\{ \frac{1}{\| \eta_\varepsilon \|}, \frac{1}{1+2\| f_1 \|} \right\} \right),
\end{equation} 
then $C_T < \frac{1}{2}.$

The inequalities 
\[
\| \mathcal F_1[f]- f \|_{[1,T]}  \leq a \| f \|_{[1,T]}^2 \max\left\{ \frac{1}{|\gamma \beta|} \tilde{T}^{\gamma \beta+1},\tilde{T} \right\}, 
\] 
\[
\| \mathcal F_2[f] \|_{[1,T]} \leq a\| f \|_{[1,T]}^2 \max\left\{ \frac{1}{|\gamma \beta|} \tilde{T}^{\gamma \beta+1},\tilde{T} \right\}
\] 
and
\[
\| \mathcal F_3[f] \|_{[1,T]} \leq \| \eta_\varepsilon\|  T \leq \| \eta_\varepsilon\|   \max\left\{ \frac{1}{|\gamma \beta|} \tilde{T}^{\gamma \beta+1},\tilde{T} \right\}  
\] 
imply that if $T$ satisfies \eqref{bound for gamma>0}, then \eqref{D_T} holds for $D_T < \frac{1}{2}.$
\end{proof} 

\section*{Acknowledgments} The authors would like to thank Aleksis Vuoksenmaa for the useful discussions and his careful reading of the paper. 
The authors gratefully acknowledge the support of the Hausdorff Research Institute for Mathematics (Bonn), through the \textit{Junior Trimester Program on Kinetic Theory}, of the CRC 1060,
\textit{The mathematics of emergent effects} at the University of Bonn funded through the German
Science Foundation (DFG), the ERC Advanced Grant [number 741487] and the \textit{Atmospheric Mathematics} (AtMath) collaboration of the Faculty of Science of University of Helsinki.
The funders had no role in study design, analysis, decision to publish, or preparation of the
manuscript.

\textbf{Declaration of interest}: none

\vspace{ 2 cm}

\textbf{Marina A. Ferreira} University of Helsinki, Department of Mathematics and Statis- 

tics, P.O. Box 68, FI-00014 Helsingin yliopisto, Helsinki, Finland 

E-mail: marina.ferreira@helsinki.fi 
 
 \vspace{0.5 cm}
 \textbf{Eugenia Franco} University of Helsinki, Department of Mathematics and Statistics,
  
 P.O. Box 68, FI-00014 Helsingin yliopisto, Helsinki, Finland 
 
 E-mail: eugenia.franco@helsinki.fi 
 
  \vspace{0.5 cm}

\textbf{Juan J.L. Vel\'azquez} Institute for Applied Mathematics, University of Bonn, En-

denicher Allee 60, D-53115, Bonn, Germany 

 E-mail: velazquez@iam.uni-bonn.de





  \bibliographystyle{plain}
  \bibliography{bibliography}





\end{document}